\documentclass[11pt]{article}
\usepackage{amsmath,amsfonts,amsthm,amssymb, authblk,color,fullpage}
\usepackage{url}
\usepackage{graphicx,amssymb,amsmath,amsthm}
\usepackage{enumerate}
\usepackage{dsfont}
\usepackage{booktabs}
\usepackage{mathrsfs}
\usepackage{verbatim}
\usepackage{epsfig}
\usepackage{lscape}
\usepackage{subfigure}
\usepackage[colorlinks, linkcolor=red, anchorcolor=blue, citecolor=green]{hyperref}
\usepackage{epstopdf}
\usepackage{color}
\usepackage{caption}
\usepackage{algorithm}
\usepackage{algpseudocode}

\usepackage{graphicx}
\usepackage{subfigure}
\usepackage{cleveref}
\graphicspath{{Figures/}{logo/}}

\providecommand{\keywords}[1]
{{\small\textbf{\textbf{Keywords:}} #1}}

\providecommand{\MSC}[1]
{{\small\textbf{\textbf{Mathematics Subject Classification:}} #1}}

\newcommand\blfootnote[1]{%
  \begingroup
  \renewcommand\thefootnote{}\footnote{#1}%
  \addtocounter{footnote}{-1}%
  \endgroup
}

\newtheorem{definition}{Definition}[section]

\newtheorem{prop}{Proposition}[section]
\newtheorem{theorem}{Theorem}[section]
\newtheorem{lemma}{Lemma}[section]

\newtheorem{remark}{Remark}[section]

\newtheorem{problem}{Problem}[section]

\numberwithin{equation}{section}
\date{}

\newcommand{\Addresses}{{
\bigskip
\footnotesize

Jian-Feng Cai, \textsc{Department of Mathematics, The Hong Kong University of Science and Technology, Clear Water Bay, Kowloon, Hong Kong SAR, China}\par\nopagebreak
\textit{E-mail address}, Jian-Feng Cai: \texttt{jfcai@ust.hk}

\medskip

Zhiqiang Xu, \textsc{LSEC, Inst.~Comp.~Math., Academy of
Mathematics and System Science,  Chinese Academy of Sciences, Beijing, 100091, China
\newline
School of Mathematical Sciences, University of Chinese Academy of Sciences, Beijing 100049, China}\par\nopagebreak
\textit{E-mail address}, Zhiqiang Xu: \texttt{xuzq@lsec.cc.ac.cn}

\medskip

Zili Xu, \textsc{School of Mathematical Sciences, East China Normal University, Shanghai 200241, China;
Shanghai Key Laboratory of PMMP, East China Normal University, Shanghai 200241, China; and
Key Laboratory of MEA, Ministry of Education, East China Normal University, Shanghai 200241, China}\par\nopagebreak
\textit{E-mail address}, Zili Xu: \texttt{zlxu@math.ecnu.edu.cn}
}}

\begin{document}
\baselineskip 14pt
\bibliographystyle{plain}

\title{ 
Interlacing Polynomial Method for Matrix Approximation via Generalized Column and Row Selection
\blfootnote{The work of Jian-Feng Cai was supported in part by the Hong Kong Research Grants Council GRF under Grants 16307023 and 16306124, and in part by Hong Kong Innovation and Technology Fund MHP/009/20. The work of Zhiqiang Xu is supported by the National Science Fund for Distinguished Young Scholars (12025108) and the National Nature Science Foundation of China (12471361, 12021001, 12288201).}
}

\author{Jian-Feng Cai, Zhiqiang Xu, and Zili Xu}

\maketitle

\begin{abstract}
This paper delves into the spectral norm aspect of the Generalized Column and Row Subset Selection (GCRSS) problem.
Given a target matrix $\mathbf{A}\in \mathbb{R}^{n\times d}$,
the objective of GCRSS is to select a column submatrix $\mathbf{B}_{:,S}\in\mathbb{R}^{n\times k}$ from the source matrix $\mathbf{B}\in\mathbb{R}^{n\times d_B}$ and a row submatrix $\mathbf{C}_{R,:}\in\mathbb{R}^{r\times d}$ from the source matrix $\mathbf{C}\in\mathbb{R}^{n_C\times d}$, such that the residual matrix $(\mathbf{I}_n-\mathbf{B}_{:,S}\mathbf{B}_{:,S}^{\dagger})\mathbf{A}(\mathbf{I}_d-\mathbf{C}_{R,:}^{\dagger} \mathbf{C}_{R,:})$ has a small spectral norm.
By employing the method of interlacing polynomials, we show that 
the smallest possible spectral norm of a residual matrix can be bounded by the largest root of a related expected characteristic polynomial.
A deterministic polynomial time algorithm is provided for the spectral norm case of the GCRSS problem.
We next apply our results to two specific GCRSS scenarios, one where $r=0$, simplifying the problem to the Generalized Column Subset Selection (GCSS) problem, and the other where $\mathbf{B}=\mathbf{C}=\mathbf{I}_d$, reducing the problem to the submatrix selection problem. 
In the GCSS scenario, we connect the expected characteristic polynomials to the convolution of multi-affine polynomials, leading to the derivation of the first provable reconstruction bound on the spectral norm of a residual matrix. 
In the submatrix selection scenario, we show that for any sufficiently small $\varepsilon>0$ and any square matrix $\mathbf{A}\in\mathbb{R}^{d\times d}$, there exist two subsets $S\subset [d]$ and $R\subset [d]$ of sizes $O(d\cdot \varepsilon^2)$ such that $\Vert\mathbf{A}_{S,R}\Vert_2\leq \varepsilon\cdot \Vert\mathbf{A}\Vert_2$. 
Unlike previous studies that have produced comparable results for very special cases where the matrix is either a zero-diagonal or a positive semidefinite matrix, our results apply universally to any square matrix $\mathbf{A}$.
\end{abstract}


\MSC{15A60, 90C27}

\section{Introduction}

High-dimensional datasets are commonly encountered in machine learning, often requiring the use of dimensionality reduction techniques as an initial step in the data processing pipeline. One traditional approach to obtaining a compact low-dimensional approximation of an input matrix is through singular value decomposition (SVD). However, the resulting factors obtained from SVD combine the rows or columns of the input matrix in a complex manner, making it difficult to intuitively interpret them. To address this challenge, one solution is to selectively choose a small number of columns and rows from specific source matrices that contain interpretable factors capable of approximating the span of the input data matrix. This technique is referred to as the generalized columns and rows selection problem (GCRSS). It provides an efficient and interpretable method for dimensionality reduction, which plays a crucial role in uncovering patterns within high-dimensional data. 
In this paper, we study the error analysis of GCRSS through the method of interlacing polynomials.

\subsection{Generalized Column and Row Subset Selection}\label{section: intro}

For an integer $n$ we write $[n]:=\{1,\ldots,n\}$.
For a matrix $\mathbf{A}\in\mathbb{R}^{n\times d}$, we use $\mathbf{A}_{R,S}$ to denote the submatrix of $\mathbf{A}$ consisting of rows indexed in the set $R$ and columns indexed in the set $S$. 
For simplicity, if $R=[n]$ we write $\mathbf{A}_{R,S}$ as $\mathbf{A}_{:,S}$, and if $S=[d]$ we write $\mathbf{A}_{R,S}$ as $\mathbf{A}_{R,:}$. We use $\mathbf{A}_{R,S}^{\dagger}$ to denote the Moore-Penrose pseudoinverse of $\mathbf{A}_{R,S}$, i.e., $\mathbf{A}_{R,S}^{\dagger}=(\mathbf{A}_{R,S})^{\dagger}$.
If $S$ or $R$ is an empty set, then we consider $\mathbf{A}_{R,S}^{\dagger} \mathbf{A}_{R,S}$ as a zero matrix.
The GCRSS problem is stated as follows:
\begin{problem}[Generalized Column and Row Subset Selection (GCRSS)]\label{pr3}
Given a target matrix $\mathbf{A}\in\mathbb{R}^{n\times d}$, two source matrices $\mathbf{B}\in\mathbb{R}^{n\times d_{{B}}}$ and $\mathbf{C}\in\mathbb{R}^{n_{{C}}\times d}$, and two nonnegative integers $k\leq \mathrm{rank}(\mathbf{B})$ and $r\leq \mathrm{rank}(\mathbf{C})$, find a subset $S\subset[d_{{B}}]$ of size $k$ and a subset $R\subset[n_{{C}}]$ of size $r$ such that
\begin{equation*}
\Vert(\mathbf{I}_n-\mathbf{B}_{:,S}\mathbf{B}_{:,S}^{\dagger})\mathbf{A}(\mathbf{I}_d-\mathbf{C}_{R,:}^{\dagger} \mathbf{C}_{R,:})\Vert_{\xi}
\end{equation*}
is minimized over all possible $\binom{d_B}{k}$ choices for the subset $S$ and over all possible $\binom{n_C}{r}$ choices for the subset $R$. Here $\xi=2$ or $\rm F$ denotes the spectral or Frobenius norm, respectively.
\end{problem}

Here we use the term \emph{generalized} to distinguish Problem \ref{pr3} from the traditional column and row subset selection where the columns and rows are selected from the target matrix itself \cite{BMD09,DR10,DRVW06,DMM08}. 
In the following, we present several specific instances of Problem \ref{pr3} and discuss their connections to other problems in matrix approximation theory, signal processing, and operator theory.

\subsubsection{Generalized Column Subset Selection}

By setting $r=0$, we have $R=\emptyset$ and Problem \ref{pr3} is reduced to the Generalized Column Subset Selection (GCSS) problem.
Here, we set $\mathbf{I}_d-\mathbf{C}_{\emptyset,:}^{\dagger} \mathbf{C}_{\emptyset,:}=\mathbf{I}_d$.
The GCSS problem can be stated  as follows:
\begin{problem}[Generalized Column Subset Selection (GCSS)]\label{pr1}
Given a target matrix $\mathbf{A}\in\mathbb{R}^{n\times d}$, a source matrix $\mathbf{B}\in\mathbb{R}^{n\times d_B}$ and a positive integer $k\leq \mathrm{rank}(\mathbf{B})$, the object is to  find a subset $S\subset[d_B]$ of size $k$ such that the following residual
\begin{equation*}
\Vert\mathbf{A}-\mathbf{B}_{:,S}\mathbf{B}_{:,S}^{\dagger}\mathbf{A} \Vert_{\xi}= \min\limits_{ \mathbf{X}\in\mathbb{R}^{k\times d}} \Vert\mathbf{A}-\mathbf{B}_{:,S}\mathbf{X} \Vert_{\xi}
\end{equation*}
is minimized over all possible $\binom{d_B}{k}$ choices for the subset $S$. Here, $\xi=2$ or $\rm F$ denotes the spectral or Frobenius norm, respectively.
\end{problem}

The GCSS has attracted much attention over the past few years due to its diverse applications in machine learning \cite{BRN10,BBC20,BMD09,CH92,DMM08,FGK13,GE03,HP92,OMG22}.
It is a powerful and straightforward method for approximating an input matrix $\mathbf{A}$, and it produces results that are easily interpretable in relation to the source matrix.

Several problems can be formulated as Generalized Column Subset Selection (GCSS), as mentioned in \cite{FGK13,OMG22}, including the classical Column Subset Selection (CSS) problem \cite{BDM14,BMD09,CXX23,GE96,WS18},  sparse approximation \cite{LBRN06,OF97},  distributed column subset selection \cite{FEGK13} and sparse canonical correlation analysis \cite{Hot36,OMG22}. Let us take column subset selection and  sparse approximation  as examples.
If we set $\mathbf{B}=\mathbf{A}$, Problem \ref{pr1} reduces to the classical Column Subset Selection (CSS) problem \cite{FEGK13}. In CSS, the objective is to select a subset of $k$ columns from $\mathbf{A}$ such that the discrepancy between $\mathbf{A}$ and its projection onto the subspace spanned by the selected columns is minimized.
If we consider $\mathbf{A}=\mathbf{a}\in\mathbb{R}^{n\times 1}$ as an $n$-dimensional vector and $\mathbf{B}\in\mathbb{R}^{n\times d_B}$, Problem \ref{pr1} can be interpreted as finding a sparse approximation to the target vector $\mathbf{a}$ by minimizing the difference between $\mathbf{a}$ and its projection onto the subspace spanned by the selected $k$ columns in $\mathbf{B}$. This problem is closely related to sparse coding and dictionary selection, which have wide applications in signal processing \cite{DK11,LBRN06,OF97}.

\subsubsection{Submatrix Selection}

Submatrix selection is another specific instance of Problem \ref{pr3}, which arises when we set $\mathbf{B}=\mathbf{I}_n$ and $\mathbf{C}=\mathbf{I}_d$.
In this setting, we obtain
\begin{equation*}
\Vert(\mathbf{I}_n-\mathbf{B}_{:,S}\mathbf{B}_{:,S}^{\dagger})\mathbf{A}(\mathbf{I}_d-\mathbf{C}_{R,:}^{\dagger} \mathbf{C}_{R,:})\Vert_{\xi}=\Vert \mathbf{A}_{S^C,R^C}\Vert_{\xi}
\end{equation*}
for any $S\subset[d_{{B}}]$ and $R\subset[n_{{C}}]$.
  Consequently, Problem \ref{pr3} transforms into the task of selecting an $(n-k)\times (d-r)$ submatrix $\mathbf{A}_{S^C,R^C}$ with the smallest Frobenius or spectral norm. To be more precise, the submatrix selection problem can be formulated as follows.

\begin{problem}[Submatrix Selection]\label{pr2}
Given a matrix $\mathbf{A}\in\mathbb{R}^{n\times d}$ and two positive integers $k\in [n]$ and $r\in [d]$, find a subset $S\subset[n]$ of size $k$ and a subset $R\subset[d]$ of size $r$ such that
\begin{equation*}
\Vert \mathbf{A}_{S,R}\Vert_{\xi}
\end{equation*}
is minimized over all possible $\binom{n}{k}$ choices for the subset $S$ and over all possible $\binom{d}{r}$ choices for the subset $R$. Here, $\xi=2$ or $\rm F$ denotes the spectral or Frobenius norm.

\end{problem}

The problem of selecting a submatrix with bounded norm has been extensively studied in various fields, including matrix approximation, Banach space theory, and more recently, theoretical computer science \cite{BT87,BT89,CT,CEKP07,inter3,You2,ravi1,RV07,SS12,XX21}. In addition to the spectral and Frobenius norms, the cut-norm of submatrices has also been investigated in numerous papers, because it is closely related to the problem of additive approximation of the MAX-CSP problems (for more detailed information, please refer to \cite{ADKK03,RV07}).

The principal submatrix selection is an interesting case of Problem \ref{pr2}.
If we take $\xi=2$, $n=d$ and require $S$ and $R$ to be the same subset, then Problem \ref{pr2} becomes the problem of selecting a $k\times k$ principal submatrix with the smallest spectral norm. This problem has a close relationship with the restricted invertibility principle of Bourgain and Tzafriri \cite{BT87,BT89,inter3,You2,ravi1,SS12}. The case where $\mathbf{A}$ is a zero-diagonal matrix is studied in \cite[Corollary 1.2]{BT91} and \cite[Corollary 1.5]{ravi3}. For the case where $\mathbf{A}$ is positive semidefinite we refer to \cite{inter3,You2} and \cite[Theorem 1.7]{ravi1}.

If we set $r=d$ and $\xi=2$, Problem \ref{pr2} can be transformed into the problem of selecting a row submatrix $\mathbf{A}_{S,:}\in\mathbb{R}^{k\times d}$ with the smallest spectral norm. This row selection problem has been studied in \cite{KT93,RV07}. An optimal estimate on the expectation of $\Vert \mathbf{A}_{S,:}\Vert_{2}$ was given by Rudelson and Vershynin in \cite[Theorem 1.8]{RV07}. Observe that $\Vert \mathbf{A}_{S,:}\Vert_{2}^2=\Vert (\mathbf{A}\mathbf{A}^{\rm T})_{S,S}\Vert_{2}$. Hence, this row selection problem can be further transferred into the principal submatrix selection for positive semidefinite matrices.

\subsection{Contributions}

This paper mainly focuses on the spectral norm case of the GCRSS problem (Problem \ref{pr3}). 
The spectral norm measures the largest singular value of the residual matrix, which captures the worst-case stretch of the matrix. 
This property makes the spectral norm particularly useful when the goal is to ensure that the sketch matrix preserves the dominant singular structure of the original matrix. 
In contrast, the Frobenius norm, as an ``averaging'' norm that sums contributions over all singular directions, is less sensitive to extreme values. 
Moreover, the degradation of spectral norms under column and row sampling has been extensively researched, but their behavior under projection onto the column or row space of a different source matrix remains largely unknown.
For example, the Frobenius norm case of the generalized column subset selection problem (Problem \ref{pr1}) has been investigated \cite{ABFMRZ16,FGK13,OMG22}, but the spectral norm case remains an open question. 
Motivated by these gaps, this paper concentrates on the spectral norm case of the GCRSS problem.

To present our main result, we first introduce some definitions. For any matrix $\mathbf{M}$, throughout this paper we set $\det[\mathbf{M}_{\emptyset,\emptyset}]=1$ and consider $\mathbf{M}_{\emptyset,:}^{\dagger} \mathbf{M}_{\emptyset,:}$ and $\mathbf{M}_{:,\emptyset}\mathbf{M}_{:,\emptyset}^{\dagger}$ as zero matrices. 

\begin{definition}\label{def2}
Let $\mathbf{A}\in\mathbb{R}^{n\times d}$, $\mathbf{B}\in\mathbb{R}^{n\times d_{{B}}}$, and $\mathbf{C}\in\mathbb{R}^{n_{{C}}\times d}$.  For any subset $S\subset[d_B]$ and any subset $R\subset[n_C]$, we define the degree-$d$ polynomial $p_{S,R}(x;\mathbf{A},\mathbf{B},\mathbf{C})$ as follows:
\begin{equation*}
p_{S,R}(x;\mathbf{A},\mathbf{B},\mathbf{C}):=\det\Big[x\cdot \mathbf{I}_{d}-\Big((\mathbf{I}_n-\mathbf{B}_{:,S}\mathbf{B}_{:,S}^{\dagger}) \mathbf{A}(\mathbf{I}_d-\mathbf{C}_{R,:}^{\dagger} \mathbf{C}_{R,:}) \Big)^{\rm T}\Big((\mathbf{I}_n-\mathbf{B}_{:,S}\mathbf{B}_{:,S}^{\dagger}) \mathbf{A}(\mathbf{I}_d-\mathbf{C}_{R,:}^{\dagger} \mathbf{C}_{R,:}) \Big)\Big ].
\end{equation*}
For any two integers $k$ and $r$ that satisfy $0\leq k\leq \mathrm{rank}(\mathbf{B})$ and $0\leq r\leq \mathrm{rank}(\mathbf{C})$, we define the degree-$d$ polynomial $P_{k,r}(x;\mathbf{A},\mathbf{B},\mathbf{C})$ as follows:
\begin{equation*}
P_{k,r}(x;\mathbf{A},\mathbf{B},\mathbf{C}):=\sum_{R\subset[n_C],|R|=r}	 \sum_{S\subset[d_B],|S|=k}\det[\mathbf{C}_{R,:}(\mathbf{C}_{R,:})^{\rm T}] \cdot \det[(\mathbf{B}_{:,S})^{\rm T}\mathbf{B}_{:,S}]\cdot  p_{S,R}(x;\mathbf{A},\mathbf{B},\mathbf{C}).
\end{equation*}
For any integer $0\leq k\leq \mathrm{rank}(\mathbf{B})$, we define the degree-$d$ polynomial $P_{k}(x;\mathbf{A},\mathbf{B})$ as follows:
\begin{equation*}
P_{k}(x;\mathbf{A},\mathbf{B}):=P_{k,0}(x;\mathbf{A},\mathbf{B},\mathbf{0})=\sum_{S\subset[d_B],|S|=k}\det[(\mathbf{B}_{:,S})^{\rm T}\mathbf{B}_{:,S}]\cdot  \det[x\cdot \mathbf{I}_{d}- \mathbf{A}^{\rm T}(\mathbf{I}_n-\mathbf{B}_{:,S}\mathbf{B}_{:,S}^{\dagger})\mathbf{A}].
\end{equation*}

\end{definition}

Recall that for a given matrix $\mathbf{M}\in\mathbb{R}^{n\times m}$ and a positive integer $1\leq k\leq \mathrm{rank}(\mathbf{M})$, volume sampling refers to the method of selecting a $k$-subset of $[m]$ such that the probability of selection is proportional to $\det[(\mathbf{M}_{:,S})^{\rm T}\mathbf{M}_{:,S}]$ \cite{DR10,DRVW06}.
As such, the polynomial $P_{k,r}(x;\mathbf{A},\mathbf{B},\mathbf{C})$ can be interpreted as the expectation of $p_{S,R}(x;\mathbf{A},\mathbf{B},\mathbf{C})$ under volume sampling on the columns of $\mathbf{B}$ and rows of $\mathbf{C}$, up to a constant.

The following theorem is our main result for the spectral norm case of Problem \ref{pr3}.
 It states that the largest root of the expected polynomial $P_{k,r}(x;\mathbf{A},\mathbf{B},\mathbf{C})$ serves as an upper bound for the minimum value among all the largest roots of $p_{S,R}(x;\mathbf{A},\mathbf{B},\mathbf{C})$.
 
\begin{theorem}\label{mth1-spr}
Let $\mathbf{A}\in\mathbb{R}^{n\times d}$, $\mathbf{B}\in\mathbb{R}^{n\times d_{{B}}}$, and $\mathbf{C}\in\mathbb{R}^{n_{{C}}\times d}$. For any two non-negative integers $k\leq\mathrm{rank}(\mathbf{B})$ and $r\leq\mathrm{rank}(\mathbf{C})$, there exist a $k$-subset $\widehat{S}\subset[d_B]$ and an $r$-subset $\widehat{R}\subset[n_C]$ that can be iteratively selected such that $\mathrm{rank}(\mathbf{B}_{:,\widehat{S}})=k$, $\mathrm{rank}(\mathbf{C}_{\widehat{R},:})=r$ and
\begin{equation*}
\Vert (\mathbf{I}_n-\mathbf{B}_{:,\widehat{S}}\mathbf{B}_{:,\widehat{S}}^{\dagger})  \mathbf{A}  (\mathbf{I}_d-\mathbf{C}_{\widehat{R},:}^{\dagger} \mathbf{C}_{\widehat{R},:}) \Vert_{2}^2
\leq \mathrm{maxroot}\ P_{k,r}(x;\mathbf{A},\mathbf{B},\mathbf{C}).
\end{equation*}
	
\end{theorem}

 The proof of Theorem \ref{mth1-spr} relies on the method of interlacing polynomials introduced by Marcus, Spielman, and Srivastava in their influential resolution of the Kadison-Singer problem \cite{inter2}. In Section \ref{section: basic prop}, we will demonstrate that the polynomial $P_{k,r}(x;\mathbf{A},\mathbf{B},\mathbf{C})$ has only nonnegative real roots when $k\leq\mathrm{rank}(\mathbf{B})$ and $r\leq\mathrm{rank}(\mathbf{C})$. Thus, the largest root of $P_{k,r}(x;\mathbf{A},\mathbf{B},\mathbf{C})$ is well defined.

Moreover, in Proposition \ref{P-three-exp} we show that $P_{k,r}(x;\mathbf{A},\mathbf{B},\mathbf{C})$ can be expressed as
\begin{equation*}
P_{k,r}(x;\mathbf{A},\mathbf{B},\mathbf{C})=	\frac{(-1)^{k+r}\cdot x^{r}}{ (d_B-k)!\cdot (n_C-r)!} \cdot \partial_y^{d_B-k}\partial_z^{n_C-r}\ \det\left[\begin{matrix}
\mathbf{I}_n  & \mathbf{0} & \mathbf{B} & \mathbf{A}  \\
 \mathbf{0} & z\cdot \mathbf{I}_{n_C} & \mathbf{0}  & \mathbf{C}  \\
\mathbf{B}^{\rm T}  & \mathbf{0}  & y\cdot \mathbf{I}_{d_B} & \mathbf{0}  \\
\mathbf{A}^{\rm T}  & \mathbf{C}^{\rm T} & \mathbf{0}  & x\cdot  \mathbf{I}_d
\end{matrix}\right]  \ \Bigg\vert_{y=z=0}.
\end{equation*}
This expression can be further simplified, especially in the contexts of the GCSS problem (Problem \ref{pr1}) and submatrix selection problem (Problem \ref{pr2}). This allows us to efficiently compute $P_{k,r}(x;\mathbf{A},\mathbf{B},\mathbf{C})$ and its largest root.

In Section \ref{section: algorithm for GCRSS}, we introduce a deterministic polynomial time algorithm for the spectral norm case of the GCRSS problem; see Algorithm \ref{alg-spr-GCRSS}. The following theorem establishes the approximation error guarantee and computational complexity of Algorithm \ref{alg-spr-GCRSS}.

\begin{theorem}\label{mth:alg:spr}
Let  $\mathbf{A}\in\mathbb{R}^{n\times d}$, $\mathbf{B}\in\mathbb{R}^{n\times d_B}$ and $\mathbf{C}\in\mathbb{R}^{n_C\times d}$. 
Let $k$ and $r$ be two integers satisfying that  $0\leq k\leq \mathrm{rank}(\mathbf{B})$ and $0\leq r\leq \mathrm{rank}(\mathbf{C})$.
Algorithm \ref{alg-spr-GCRSS} outputs two subsets $\widehat{S}=\{j_1,\ldots,j_k\}\subset [d_B]$ and $\widehat{R}=\{j_1',\ldots,j_r'\}\subset [n_C]$, such that
\begin{equation}\label{alg:eq1-GCRSS}
\Vert(\mathbf{I}_n-\mathbf{B}_{:, \widehat{S}}\mathbf{B}_{:, \widehat{S}}^{\dagger})\mathbf{A}(\mathbf{I}_d-\mathbf{C}_{\widehat{R}, :}^{\dagger}\mathbf{C}_{\widehat{R}, :}) \Vert_{2}^2\leq  2(k+r)\eta+\mathrm{maxroot}\ P_{k,r}(x;\mathbf{A},\mathbf{B},\mathbf{C}),
\end{equation}
where $\eta>0$ is an error parameter.
The running time of Algorithm \ref{alg-spr-GCRSS} is
\begin{equation*}
O\Big(dn^2+r\cdot n_C\cdot  d(n+n_C)+(k\cdot d_B+r\cdot n_C)\cdot \big((n+n_C)^{w+2}+d^2\log\frac{1}{\eta}\big)\Big),	
\end{equation*} 
where $w \in  (2, 2.373)$ is the matrix multiplication exponent.

\end{theorem}

In the following, we focus on two specific instances of Problem \ref{pr3}: the case where $r=0$ (discussed in Section \ref{sec:122}) and the case where $\mathbf{B}=\mathbf{C}=\mathbf{I}_d$  (examined in Section \ref{sec:123}). As previously outlined in Section \ref{section: intro}, the former simplifies Problem \ref{pr3} to the GCSS problem, while the latter transforms it into the submatrix selection problem.
For these two cases, we provide an estimate for the largest root of $P_{k,r}(x;\mathbf{A},\mathbf{B},\mathbf{C})$ and apply the result of Theorem \ref{mth1-spr} to both the GCSS and submatrix selection problems. 
\textcolor{black}{
 Moreover, Algorithm \ref{alg-spr-GCRSS} can be effectively applied to both the GCSS and submatrix selection problems. 
 Notably, the time complexity can be further reduced because the related expected polynomial has a simplified expression in these two cases. 
 For a detailed discussion, please refer to Remark \ref{re:41}.}

\subsubsection{Generalized Column Subset Selection}\label{sec:122}

 We begin by considering the case where $r=0$. In this case we study the expected polynomial $P_{k}(x;\mathbf{A},\mathbf{B})$, as defined in Definition \ref{def2}. We show that $P_{k}(x;\mathbf{A},\mathbf{B})$ can be related to the additive convolution of multi-affine polynomials introduced in \cite{ravi4}. In order to estimate the largest root of $P_{k}(x;\mathbf{A},\mathbf{B})$, we utilize the multivariate barrier function method, as introduced in \cite{inter2,SS12}.

The following theorem is our main result for the spectral norm case of the GCSS.  To the best of our knowledge, this is the first provable reconstruction bound for the spectral norm of a residual matrix.

\begin{theorem}\label{main thm-spr}
Let $\mathbf{A}$ be a matrix in $\mathbb{R}^{n\times d}$, and let $\mathbf{B}\in\mathbb{R}^{n\times d_B}$ be a rank-$m$ matrix.  
Let $\mathbf{B}\mathbf{B}^{\rm T}=\mathbf{U}\cdot \mathrm{diag}(b_1,\ldots,b_m)\cdot \mathbf{U}^{\rm T}$ be the singular value decomposition of $\mathbf{B}\mathbf{B}^{\rm T}$, where $\mathbf{U}\in\mathbb{R}^{n\times m}$ satisfies $\mathbf{U}^{\rm T}\mathbf{U}=\mathbf{I}_m$
and $b_i$ is the $i$-th largest squared singular value of $\mathbf{B}$. 
Let $\alpha$ be the largest diagonal element of the matrix $\frac{\mathbf{U}^{\rm T}\mathbf{A}\mathbf{A}^{\rm T}\mathbf{U}}{\Vert \mathbf{U}^{\rm T}\mathbf{A}\mathbf{A}^{\rm T}\mathbf{U}\Vert_2}$. 
Set 
\begin{equation*}
\delta_{k}:= \frac{\sum_{S\subset[m-1],\vert S\vert=k} \mathbf{b}^{S}}{	 \sum_{S\subset[m],\vert S\vert=k} \mathbf{b}^{S}},
\end{equation*}
where  $k\in[m-1]$.
If $\delta_{k}\leq(1-\sqrt{\alpha})^2$, then we can iteratively select a subset $\widehat{S}\subset [d_B]$ of size $k$ such that $\mathrm{rank}(\mathbf{B}_{:,\widehat{S}})=k$ and
\begin{equation}\label{mth2:bound}
\Vert\mathbf{A}-\mathbf{B}_{:, \widehat{S}}\mathbf{B}_{:, \widehat{S}}^{\dagger}\mathbf{A} \Vert_{2}^2
\leq 
\Vert\mathbf{A}-\mathbf{B}_{}\mathbf{B}_{}^{\dagger}\mathbf{A} \Vert_{2}^2
+\varepsilon\cdot \Vert \mathbf{B}_{}\mathbf{B}_{}^{\dagger}\mathbf{A}\Vert_2^2,
\end{equation}	
where
\begin{equation}\label{eq:epsilon}
\varepsilon:=\big(\sqrt{\alpha\cdot t_{}}+\sqrt{(1-\alpha)(1-t_{})}\big )^2
\quad\text{and}\quad  t_{}:=\big(1-\sqrt{\delta_{k}}\big)^2.
\end{equation}
\end{theorem}

In the following we make some remarks on the result of Theorem \ref{main thm-spr}.
\begin{remark}
The proof of Theorem \ref{main thm-spr} relies on representing the expected polynomial in terms of the multivariate differential operator $\prod_{i=1}^m \partial_{z_i}$ and then estimating its largest root using the multivariate barrier method.
However, this approach can produce loose estimates when $\alpha$ is large.
Thus, we impose the condition $\delta_{k}\leq(1-\sqrt{\alpha})^2$ 
  to ensure that the upper bound \eqref{mth2:bound} remains well-controlled. More precisely, this condition guarantees that the upper bound \eqref{mth2:bound} does not exceed $\|\mathbf{A}\|_2^2$  when $\mathbf{B}_{}\mathbf{B}_{}^{\dagger}\mathbf{A}=\mathbf{A}$. 
In this case, our upper bound reduces to the simpler form $\varepsilon\cdot \Vert \mathbf{A}\Vert_2^2$.
Observe that $\varepsilon$ decreases from $1$ to $\alpha$ as  $\delta_k$ decreases from  $(1-\sqrt{\alpha})^2$ to $0$, so the upper bound $\varepsilon\cdot \Vert \mathbf{A}\Vert_2^2$ does not exceed $\Vert \mathbf{A}\Vert_2^2$ when $\delta_{k}\leq(1-\sqrt{\alpha})^2$.
However, when $\mathbf{B}_{}\mathbf{B}_{}^{\dagger}\mathbf{A}\neq \mathbf{A}$, the condition $\delta_k \leq (1 - \sqrt{\alpha})^2$ is not sufficient. Due to the additional term $\Vert\mathbf{A}-\mathbf{B}_{}\mathbf{B}_{}^{\dagger}\mathbf{A} \Vert_{2}^2$,  $\delta_{k}$ must be much smaller than $(1-\sqrt{\alpha})^2$ to ensure the bound \eqref{mth2:bound} remains under $\Vert\mathbf{A}\Vert_2^2$.
This limitation arises from the proof technique, and we leave this issue for future research.

\end{remark}

\begin{remark}
When the target matrix $\mathbf{A}$ and the source matrix $\mathbf{B}$ are specified, we can consider $\alpha\in [0,1]$ as a fixed number.
Using the assumption that $b_1\geq b_2\geq \cdots\geq b_m>0$, we can prove that
\begin{equation}\label{estimate deltak}
1-\frac{k}{m}
\leq \delta_k
\leq \min\Big\{1, \kappa(\mathbf{B})^2\cdot (1-\frac{k}{m})\Big \},
\end{equation}
where $\kappa(\mathbf{B})=\| \mathbf{B}\|_2\cdot \| \mathbf{B}^{\dagger}\|_2=\sqrt{\frac{b_1}{b_m}}$ represents the generalized condition number of $\mathbf{B}$ 
{\rm (}please refer to Lemma \ref{lemma-estimate deltak} for the proof of \eqref{estimate deltak}{\rm )}.
Therefore, if $\kappa(\mathbf{B})^2\cdot (1-\frac{k}{m})\leq (1-\sqrt{\alpha})^2$, i.e.,
\begin{equation*}
\frac{k}{m}\geq 1-\frac{(1-\sqrt{\alpha})^2}{\kappa(\mathbf{B})^2},	
\end{equation*}
then the condition $\delta_k\leq (1-\sqrt{\alpha})^2$ is satisfied.
Furthermore, combining \eqref{estimate deltak} with the fact that 
$\alpha(1-\alpha)\leq \frac14$ and $0\leq 1-t=2 \sqrt{\delta_k}-\delta_k\leq 2 \sqrt{\delta_k}$, 
we can estimate $\varepsilon$,  which is defined in (\ref{eq:epsilon}), as follows:
\begin{equation*}
\begin{aligned}
\varepsilon
&=\alpha+(1-2\alpha)(1-t)+2\sqrt{\alpha(1-\alpha)t(1-t)}\\
&\leq \alpha+(1-t)+\sqrt{1-t}
\leq \alpha+4\cdot \delta_k^{1/4}
\leq \alpha+4\sqrt{\kappa(\mathbf{B})}\cdot \big(1-\frac{k}{m}\big)^{1/4}.
\end{aligned}
\end{equation*} 
Consequently, we can derive a more relaxed yet more straightforward upper bound from \eqref{mth2:bound}:
\begin{equation}\label{newupperbound}
\Vert\mathbf{A}-\mathbf{B}_{:, \widehat{S}}\mathbf{B}_{:, \widehat{S}}^{\dagger}\mathbf{A} \Vert_{2}^2
\leq 
\Vert\mathbf{A}-\mathbf{B}_{}\mathbf{B}_{}^{\dagger}\mathbf{A} \Vert_{2}^2
+\Big(
\alpha+4  \sqrt{\kappa(\mathbf{B})}\cdot \big(1-\frac{k}{m}\big)^{1/4}
\Big)\cdot \Vert \mathbf{B}_{}\mathbf{B}_{}^{\dagger}\mathbf{A}\Vert_2^2.
\end{equation}

\end{remark}

\begin{remark} 
Unfortunately, when $\mathbf{B}=\mathbf{A}$, the condition $\delta_k\leq (1-\sqrt{\alpha})^2$
becomes unattainable. This is due to the fact that the matrix $\frac{\mathbf{U}^{\rm T}\mathbf{A}\mathbf{A}^{\rm T}\mathbf{U}}{\Vert \mathbf{U}^{\rm T}\mathbf{A}\mathbf{A}^{\rm T}\mathbf{U}\Vert_2}$
  becomes a diagonal matrix when $\mathbf{B}=\mathbf{A}$, resulting in $\alpha=1$ and $(1-\sqrt{\alpha})^2=0$. Meanwhile, $\delta_k$
  is strictly positive for all $k\in[m-1]$.
We addressed this issue in a separate paper \cite{CXX23} by simplifying $P_{k}(x;\mathbf{A},\mathbf{A})$
  in terms of the univariate differential operator $\partial_{x}^k$. Specifically, we showed that
\begin{equation*}
P_{k}(x;\mathbf{A},\mathbf{A})=\frac{(-1)^k}{k!}\cdot\mathcal{R}_{x,d}^+\cdot  \partial_x^{k}\cdot  \mathcal{R}_{x,d}^+ \det[x\cdot\mathbf{I}_d-  \mathbf{A}^{\rm T}\mathbf{A}].
\end{equation*}
Here, the flip operator $\mathcal{R}_{x,d}^+$ is defined as 
$\mathcal{R}_{x,d}^+\ p(x):=x^d\cdot p(1/x)$
 for any polynomial $p(x)$ of degree at most $d$ 
 (see also Remark \ref{CSS-GCSS}).
This simplified form allowed us to apply the more accurate univariate barrier method for estimating the largest root of $P_{k}(x;\mathbf{A},\mathbf{A})$. Further details can be found in \cite{CXX23}.

\end{remark}

\begin{remark}
Consider the case $\mathbf{B}=\mathbf{I}_n$.
In this case, we have $\Vert\mathbf{A}-\mathbf{B}_{:, {S}}\mathbf{B}_{:, {S}}^{\dagger}\mathbf{A} \Vert_{2}^2=\Vert\mathbf{A}_{{S}^C,:} \Vert_{2}^2=\Vert (\mathbf{A}\mathbf{A}^{\rm T})_{{S}^C,{S}^C} \Vert_2$ for any subset $S\subset[n]$. Thus, the GCSS problem (Problem \ref{pr1}) is transferred into the principal submatrix selection problem.
For convenience, we assume that $\|\mathbf{A}\|_2=1$.
Note that $\delta_k=1-\frac{k}{n}$, and $\alpha$ becomes the largest diagonal element of $\mathbf{A}\mathbf{A}^{\rm T}\in\mathbb{R}^{n\times n}$ by taking $\mathbf{U}=\mathbf{I}_n$.
Therefore, the upper bound \eqref{newupperbound} simplifies to
\begin{equation}\label{special case of main thm-spr-oct2}
\Vert (\mathbf{A}\mathbf{A}^{\rm T})_{\widehat{S}^C,\widehat{S}^C} \Vert_2=\Vert\mathbf{A}-\mathbf{B}_{:, \widehat{S}}\mathbf{B}_{:, \widehat{S}}^{\dagger}\mathbf{A} \Vert_{2}^2
\leq \alpha+4\cdot \big(1-\frac{k}{n}\big)^{1/4},
\end{equation}
where $k\in[n-1]$ satisfies $1-\frac{k}{n}\leq (1-\sqrt{\alpha})^2$.
As $k$ increases to $n-1$, the above bound decreases to $\alpha+\frac{4}{n^{1/4}}$.
This behavior aligns with the observation that $\alpha $ is the largest diagonal element of $\mathbf{A}\mathbf{A}^{\rm T}\in\mathbb{R}^{n\times n}$, which serves as an upper bound on $\Vert (\mathbf{A}\mathbf{A}^{\rm T})_{{S}^C,{S}^C} \Vert_2$ over all $(n-1)$-subsets $S\subset[n]$.
In particular, if the diagonal elements of $\mathbf{A}\mathbf{A}^{\rm T}$ are all equal and independent with $n$, then the term $\alpha$ is indispensable in \eqref{special case of main thm-spr-oct2}, because we have $\Vert (\mathbf{A}\mathbf{A}^{\rm T})_{{S}^C,{S}^C} \Vert_2=\alpha $ for each $(n-1)$-subset $S\subset[n]$.
However, the upper bound \eqref{special case of main thm-spr-oct2} can be further improved by leveraging the fact that the expected polynomial $P_{k}(x;\mathbf{A},\mathbf{I}_n)$ can also be expressed in terms of the univariate differential operator $\partial_x^k$:
\begin{equation*}
P_k(x;\mathbf{A},\mathbf{I}_n)
=x^{d-n+k}\sum_{S\subset[n],|S|=k} \det[x\cdot \mathbf{I}_{n-k}-(\mathbf{A}\mathbf{A}^{\rm T})_{S^C,S^C}]
=\frac{1}{k!}\cdot x^{d-n+k}\cdot \partial_x^k \det[x\cdot \mathbf{I}_{n}-\mathbf{A}\mathbf{A}^{\rm T}].	
\end{equation*}
Using the  univariate barrier method, Ravichandran \cite[Theorem 1.7]{ravi1} derived a tighter bound 
\begin{equation*}\label{special case of main thm-spr-oct22}
\Vert (\mathbf{A}\mathbf{A}^{\rm T})_{\widehat{S}^C,\widehat{S}^C} \Vert_2
\leq \Bigg(\sqrt{\frac{k}{n}\cdot \beta}+\sqrt{(1-\beta)(1-\frac{k}{n})}\Bigg)^2 
= \beta+O\Big((1-\frac{k}{n})^{1/2}\Big),
\end{equation*} 
where $\beta=\frac{\mathrm{Tr}(\mathbf{A}\mathbf{A}^{\rm T})}{n}$.
Nevertheless, it is unknown whether the expected polynomial $P_{k}(x;\mathbf{A},\mathbf{B})$ can be expressed using the univariate differential operator $\partial_x^k$ when $\mathbf{B} \neq \mathbf{A}$ and $\mathbf{B} \neq \mathbf{I}_n$.
Therefore, despite certain limitations, Theorem \ref{main thm-spr} covers a broader range of matrix classes for the source matrix $\mathbf{B}$.

\end{remark}

\subsubsection{Submatrix Selection}\label{sec:123}

We then proceed to examine the case where $\mathbf{A}\in\mathbb{R}^{d\times d}$ is a square matrix and $\mathbf{B}=\mathbf{C}=\mathbf{I}_d$. Furthermore, we specifically focus on the case where $r=k$. Under these conditions, we demonstrate that the expected polynomial $P_{k,k}(x;\mathbf{A},\mathbf{I}_d,\mathbf{I}_d)$ can be expressed using the Laguerre derivative operator $\partial_x\cdot x\cdot \partial_x$, where $\partial_{x}$ denotes the derivative operator $\partial\slash\partial_{x}$. To be precise, we establish the following proposition (see Proposition \ref{P-exp-I}):
\begin{equation*}\label{Pkk-Laguerre}
P_{k,k}(x;\mathbf{A},\mathbf{I}_d,\mathbf{I}_d)=\frac{1}{(k!)^2}\cdot x^k\cdot   (\partial_x\cdot x\cdot \partial_x)^k\  \det[x\cdot \mathbf{I}_{d}-\mathbf{A}^{\rm T}\mathbf{A}].
\end{equation*}

The Laguerre derivative operator $\partial_x\cdot x\cdot \partial_x$ exhibits a significant connection with the asymmetric additive convolution introduced by Marcus, Spielman, and Srivastava in \cite{inter5}. In our analysis, we will adopt their barrier function approach to estimate the largest root of $P_{k,k}(x;\mathbf{A},\mathbf{I}_d,\mathbf{I}_d)$. The following  theorem presents our primary result for the submatrix selection problem.

\begin{theorem}\label{main thm-submatrix}
Let $\mathbf{A}$ be a square matrix in $\mathbb{R}^{d\times d}$ satisfying $\Vert\mathbf{A}\Vert_2\leq 1$, and let $\beta_{}:=\frac{1}{d}\Vert\mathbf{A}\Vert_{\rm F}^2=\frac{1}{d}\sum_{i=1}^d\lambda_i\in[0,1]$, where $\lambda_i$ is the $i$-th largest eigenvalue of $\mathbf{A}^{\rm T}\mathbf{A}$. For any positive integer $k<\frac{d}{\beta+1}$, there exist two subsets $\widehat{S},\widehat{R}\subset[d]$ of size $k$, such that
\begin{equation}\label{mth2:bound-submatrix}
\Vert\mathbf{A}_{\widehat{S},\widehat{R}}\Vert_2
\leq \Big(1-\beta\Big)\cdot\frac{k}{d}+2\sqrt{\Big(1-\frac{k}{d}\Big)\cdot \frac{k}{d}\cdot \beta}.
\end{equation}	
\end{theorem}

\begin{remark}
The condition $k<\frac{d}{\beta+1}$
  ensures that the upper bound in \eqref{mth2:bound-submatrix} does not exceed 1. 
 Observe that when  $\frac{k}{d}$ is significantly small, the upper bound \eqref{mth2:bound-submatrix} becomes $O(\sqrt{\beta\cdot \frac{k}{d}})$. Consequently, we can restate Theorem \ref{main thm-submatrix} as follows: Given any square matrix $\mathbf{A}\in \mathbb{R}^{d\times d}$ and a sufficiently small positive number $\varepsilon$, there exist two subsets $\widehat{S}$ and $\widehat{R}$ of size $O(d\cdot \varepsilon^2)$ such that $\Vert\mathbf{A}_{\widehat{S},\widehat{R}}\Vert_2\leq \varepsilon\cdot \Vert\mathbf{A}\Vert_2$.
\end{remark}

\subsection{Related Work}

\subsubsection{Generalized Column Subset Selection}

Let $\mathbf{A}=[\mathbf{a}_1,\ldots,\mathbf{a}_d]\in\mathbb{R}^{n\times d}$ and $\mathbf{B}=[\mathbf{b}_1,\ldots,\mathbf{b}_{d_B}]\in\mathbb{R}^{n\times d_B}$. Most of the prior work focuses on the Frobenius norm case of GCSS. Farahat, Ghodsi, and Kamel \cite{FGK13} were the first to formalize GCSS into Problem \ref{pr1}. They demonstrated that for every subset $S\subset[d_B]$ and for every $i\in[d_B]\backslash S$,
\begin{equation}\label{related work:eq1}
\Vert \mathbf{A}-\mathbf{B}_{:,S\cup\{i\}}\mathbf{B}_{:,S\cup\{i\}}^{\dagger}\mathbf{A} \Vert_{\rm F}^2=\Vert \mathbf{A}-\mathbf{B}_{:,S}\mathbf{B}_{:,S}^{\dagger}\mathbf{A} \Vert_{\rm F}^2-\frac{\Vert\mathbf{A}^{\rm T}\mathbf{Q}_S\mathbf{b}_i\Vert^2}{\Vert\mathbf{Q}_S\mathbf{b}_i\Vert^2},
\end{equation}
where $\mathbf{Q}_S:=\mathbf{I}_n-\mathbf{B}_{:,S}\mathbf{B}_{:,S}^{\dagger}$. Based on \eqref{related work:eq1}, they proposed a greedy algorithm that selects one column $\mathbf{b}_j$ at each iteration such that
\begin{equation*}
j=\mathop{\rm argmin}_{i\in[d_B]\backslash S}	\Vert \mathbf{A}-\mathbf{B}_{:,S\cup\{i\}}\mathbf{B}_{:,S\cup\{i\}}^{\dagger}\mathbf{A} \Vert_{\rm F}^2=\mathop{\rm argmax}_{i\in[d_B]\backslash S} \frac{\Vert\mathbf{A}^{\rm T}\mathbf{Q}_S\mathbf{b}_i\Vert^2}{\Vert\mathbf{Q}_S\mathbf{b}_i\Vert^2}.
\end{equation*}
The theoretical performance of this greedy algorithm was subsequently analyzed and presented in \cite{ABFMRZ16}.
Let $S_{\mathrm{opt},r}$ be an $r$-subset of $[d_B]$ that minimizes the Frobenius norm $\Vert\mathbf{A}-\mathbf{B}_{:,S}\mathbf{B}_{:,S}^{\dagger}\mathbf{A} \Vert_{\rm F}$ with $\vert S\vert=r$. Let $\widehat{S}$ be the $k$-subset of $[d_B]$ output by the above greedy algorithm after $k$ iterations, where $k=\frac{16r}{\varepsilon\cdot \sigma}$ is the sampling size, $\varepsilon>0$ is an error parameter and $\sigma$ is the smallest squared singular value of $\mathbf{B}_{:,S_{\mathrm{opt},r}}$. 
Then the authors of \cite[Theorem 1]{ABFMRZ16} showed that
\begin{equation*}
\Vert \mathbf{B}_{:,\widehat{S}} \mathbf{B}_{:,\widehat{S}}^{\dagger}\mathbf{A}\Vert_{\rm F}^2\geq (1-\varepsilon)	\Vert \mathbf{B}_{:,S_{\mathrm{opt},r}} \mathbf{B}_{:,S_{\mathrm{opt},r}}^{\dagger}\mathbf{A}\Vert_{\rm F}^2,
\end{equation*}
which implies the following error bound on the residual $\Vert \mathbf{A}-\mathbf{B}_{:,\widehat{S}} \mathbf{B}_{:,\widehat{S}}^{\dagger}\mathbf{A}\Vert_{\rm F}^2$:
\begin{equation*}\label{related work:eq23}
\begin{aligned}
\Vert \mathbf{A}-\mathbf{B}_{:,\widehat{S}} \mathbf{B}_{:,\widehat{S}}^{\dagger}\mathbf{A}\Vert_{\rm F}^2=\Vert \mathbf{A}\Vert_{\rm F}^2-\Vert \mathbf{B}_{:,\widehat{S}} \mathbf{B}_{:,\widehat{S}}^{\dagger}\mathbf{A}\Vert_{\rm F}^2&\leq \Vert \mathbf{A}\Vert_{\rm F}^2-(1-\varepsilon)	\Vert \mathbf{B}_{:,S_{\mathrm{opt},r}} \mathbf{B}_{:,S_{\mathrm{opt},r}}^{\dagger}\mathbf{A}\Vert_{\rm F}^2\\
&=\varepsilon\Vert \mathbf{A}\Vert_{\rm F}^2+(1-\varepsilon)	\Vert \mathbf{A}-\mathbf{B}_{:,S_{\mathrm{opt},r}} \mathbf{B}_{:,S_{\mathrm{opt},r}}^{\dagger}\mathbf{A}\Vert_{\rm F}^2.	
\end{aligned}
\end{equation*}
Recently, Ordozgoiti-Matakos-Gionis \cite[Theorem 5.1]{OMG22} obtained a different type of error bound by introducing the concept of generalized leverage scores. They proposed a deterministic algorithm that outputs a subset $\widehat{S}\subset[d_B]$ such that
\begin{equation}\label{related work:eq2}
\Vert \mathbf{B}_{:,\widehat{S}} \mathbf{B}_{:,\widehat{S}}^{\dagger}\mathbf{A}\Vert_{\rm F}^2\geq (1-\delta)(1-\theta)	\Vert \mathbf{A}\Vert_{\rm F}^2,
\end{equation}
where $\delta>0$ and $\theta>0$ are two error parameters. The size of the subset $\widehat{S}$ depends on $\delta$, $\varepsilon$ and the decay of the generalized leverage scores. Note that \eqref{related work:eq2} implies the following approximation error:
\begin{equation*}\label{related work:eq22}
\Vert \mathbf{A}-\mathbf{B}_{:,\widehat{S}} \mathbf{B}_{:,\widehat{S}}^{\dagger}\mathbf{A}\Vert_{\rm F}^2=\Vert \mathbf{A}\Vert_{\rm F}^2-\Vert \mathbf{B}_{:,\widehat{S}} \mathbf{B}_{:,\widehat{S}}^{\dagger}\mathbf{A}\Vert_{\rm F}^2\leq \big(1-(1-\delta)(1-\theta)\big)	\Vert \mathbf{A}\Vert_{\rm F}^2.
\end{equation*}

The above mentioned results are specifically relevant to the Frobenius norm case of GCSS. To our knowledge, there has been no theoretical analysis conducted on the spectral norm case of GCSS. Theorem \ref{main thm-spr} presents the first provable error bound for the spectral norm of the residual matrix.

\subsubsection{Submatrix Selection}

{\bf Row submatrix selection problem:}
If we take $r=d$ and $\xi=2$, then Problem \ref{pr2} becomes the problem of selecting a row submatrix $\mathbf{A}_{S,:}\in\mathbb{R}^{k\times d}$ with the smallest spectral norm.
For the case where $\mathbf{A}\in\mathbb{R}^{d\times d}$ is a square matrix,
Kashin and Tzafriri (see \cite[Theorem 2.5]{Ver01} and \cite{KT93,RV07}) proved that there exists a subset $S\subset[d]$ of size $k$ such that
\begin{equation*}
\Vert \mathbf{A}_{S,:}\Vert_2\leq O\bigg(\sqrt{\frac{k}{d}}\cdot \Vert \mathbf{A}\Vert_2+\frac{1}{\sqrt{d}}\cdot \Vert \mathbf{A}\Vert_{\rm F} \bigg).
\end{equation*}
Vershynin and Rudelson \cite[Theorem 1.8]{RV07} employed tools from random matrix theory to derive an optimal estimate for the expectation of the spectral norm of $\mathbf{A}_{S,:}$.
In particular, they demonstrate that if $\mathbf{A}\in\mathbb{R}^{d\times d}$ is a square matrix and $S\subset[d]$ is a random subset with an expected cardinality of $k$, then
\begin{equation*}
\mathbb{E} 	\Vert \mathbf{A}_{S,:}\Vert_2\leq O\bigg(\sqrt{\frac{k}{d}}\cdot \Vert \mathbf{A}\Vert_2+\sqrt{\log k}\cdot  \Vert \mathbf{A}\Vert_{(d/k)}\bigg),
\end{equation*}
where $\Vert \mathbf{A}\Vert_{(d/k)}$ denotes the average of $d/k$ biggest Euclidean lengths of the columns of $\mathbf{A}$.

{\bf  Principal submatrix selection problem:}
 If we set $\xi=2$, $n=d$ and $S=R$, then Problem \ref{pr2} simplifies to the principal submatrix selection.
Bourgain-Tzafriri \cite[Corollary 1.2]{BT91} established the existence of a universal constant $c>0$, implying that for any $\varepsilon\in(0,1)$ and any zero-diagonal matrix $\mathbf{A}\in\mathbb{R}^{d\times d}$, there is a subset $S\subset[d]$ of minimum size $c\cdot d \varepsilon^2$ for which $\Vert \mathbf{A}_{S,S} \Vert_2\leq \varepsilon\cdot \Vert \mathbf{A} \Vert_2$ holds true.
Pierre Youssef \cite{You14note} gave a constructive solution to this problem for Hermitian matrices with the constant $c\approx 0.015$. This answered a question proposed by Assaf Naor in \cite{Naor12}. Later, Ravichandran-Srivastava improved the constant to $c=\frac{1}{12}$ in \cite[Corollary 1.5]{ravi3} by using the method of interlacing polynomial. Ravichandran studied the case where $\mathbf{A}$ is positive semidefinite satisfying that $\mathbf{0}_{}\preceq\mathbf{A}\preceq\mathbf{I}_{d}$ \cite[Theorem 1.7]{ravi1}. He showed that for any integer $k\leq (1-\beta)d$, where $\beta=\frac{\mathrm{Tr}[\mathbf{A}]}{d}$, there exists a subset $S\subset[d]$ of size $|S|=k$ such that
\begin{equation}\label{Ravichandran-bound}
\Vert \mathbf{A}_{S,S} \Vert_2\leq \Bigg(\sqrt{(1-\frac{k}{d})\beta}+\sqrt{(1-\beta)\frac{k}{d}}\Bigg)^2.
\end{equation}
The positive semidefinite case of the principal submatrix selection can be directly related to the restricted invertibility principle of Bourgain and Tzafriri. Given a matrix $\mathbf{B}\in\mathbb{R}^{d\times m}$ with $\Vert \mathbf{B}\Vert_2\leq 1$ and given a positive integer $k$, the {\em restricted invertibility problem} aims to find a $k$-subset $S\subset[m]$ such that the smallest singular value of the column submatrix $\mathbf{B}_{:,S}$, denoted by $\sigma_{\min}(\mathbf{B}_{:,S})$, is maximized. Roughly speaking, this problem aims to select $k$ columns such that the selected columns are as ``linear independent'' as possible.  Note that
\begin{equation*}
\sigma_{\min}(\mathbf{B}_{:,S})^2=\lambda_{\min}((\mathbf{B}_{:,S})^{\rm T}\mathbf{B}_{:,S})=\lambda_{\min}((\mathbf{B}^{\rm T}\mathbf{B})_{S,S})	=1-\Vert(\mathbf{I}_m-\mathbf{B}^{\rm T}\mathbf{B})_{S,S}\Vert_2,
\end{equation*}
where $\lambda_{\min}(\mathbf{M})$ denotes the smallest eigenvalue of a symmetric square matrix $\mathbf{M}$. Then, it can be observed that the problem of restricted invertibility can be transformed into the problem of selecting a $k\times k$ principal submatrix from the positive semidefinite matrix $\mathbf{I}_m-\mathbf{B}^{\rm T}\mathbf{B}$, where the goal is to minimize its spectral norm. For more details we refer to \cite{inter3,You2,SS12}.

{\bf Theorem \ref{main thm-spr}, Theorem \ref{main thm-submatrix}, and their relationship with previous results:}
Both Theorem \ref{main thm-spr} and Theorem \ref{main thm-submatrix} can be related to the submatrix selection problem. For convenience, we fix $\|\mathbf{A}\|_2=1$.
In the setting of selecting a $(n-k)\times (n-k)$ principal submatrix of $\mathbf{A}\mathbf{A}^{\rm T}\in\mathbb{R}^{n\times n}$ with bounded spectral norm, Theorem \ref{main thm-spr} provides
a suboptimal bound \eqref{special case of main thm-spr-oct2}, i.e., $\alpha+4\cdot (1-\frac{k}{n})^{1/4}$, where $\alpha$ is the largest diagonal element of $\mathbf{A}\mathbf{A}^{\rm T}$.
In contrast, 
Ravichandran' bound \eqref{Ravichandran-bound} becomes
\begin{equation*}\label{tighterbound}
\Vert (\mathbf{A}\mathbf{A}^{\rm T})_{{S'},{S'}} \Vert_2\leq \Bigg(\sqrt{\frac{k}{n}\cdot \beta}+\sqrt{(1-\beta)(1-\frac{k}{n})}\Bigg)^2,
\end{equation*}
where $S'\subset[n]$ has size $n-k$ and $\beta=\frac{\mathrm{Tr}(\mathbf{A}\mathbf{A}^{\rm T})}{n}$. When $k\to n$, i.e., $1-\frac{k}{n}\to 0$, Ravichandran' bound becomes $\beta+O((1-\frac{k}{n})^{1/2})$, which is better than our bound \eqref{special case of main thm-spr-oct2}. The advantage of Theorem \ref{main thm-spr} lies in its applicability to the generalized column subset selection.

Theorem \ref{main thm-submatrix} establishes the existence of a submatrix $\mathbf{A}_{\widehat{S},\widehat{R}}$ with dimensions $O(d\varepsilon^2)\times O(d\varepsilon^2)$, such that $\Vert\mathbf{A}_{\widehat{S},\widehat{R}}\Vert_2\leq \varepsilon\cdot \Vert\mathbf{A}\Vert_2$. The sampling size in this result, $O(d\varepsilon^2)$, is of the same order as the sampling size obtained in the zero-diagonal case of principal submatrix selection, as demonstrated in Bourgain-Tzafriri's work \cite[Corollary 1.2]{BT91}.
Compared to the results mentioned above, Theorem \ref{main thm-submatrix} applies to a broader case. Specifically, it does not impose any restrictions on the input square matrix $\mathbf{A}$, whereas the Bourgain-Tzafriri's result requires that $\mathbf{A}$ is zero-diagonal.

\section{Preliminaries}

\subsection{Notations}

We first introduce some notations. For any positive integer $n$, we set $[n]:=\{1,2,\ldots,n\}$.
For a set $S\subset[n]$, we use $S^C$ to denote the complement of $S$.
We use $\mathbf{I}_n$ to denote the identity matrix of size $n$, and use $\mathbf{0}_{n\times d}$ to denote the zero matrix of size $n\times d$. Sometimes we omit the subscript and just write $\mathbf{0}$ to represent the zero matrix of appropriate size. 
For a matrix $\mathbf{A}\in\mathbb{R}^{n\times n}$, the trace of $\mathbf{A}$ is denoted by $\text{\rm Tr}[\mathbf{A}]$. 
The Euclidean norm of a vector $\mathbf{x}$ is denoted by $\Vert\mathbf{x}\Vert$. For a matrix $\mathbf{A}\in\mathbb{R}^{n\times d}$ and for two subsets $R\subset[n]$, $S\subset[d]$, we use $\mathbf{A}_{R,S}$ to denote the submatrix of $\mathbf{A}$ consisting of rows indexed in the set $R$ and columns indexed in the set $S$. For simplicity, we write $\mathbf{A}_{R,S}$ as $\mathbf{A}_{:,S}$ if $R=[n]$, and write $\mathbf{A}_{R,S}$ as $\mathbf{A}_{R,:}$ if $S=[d]$. We use $\mathbf{A}_{R,S}^{\dagger}\in \mathbb{R}^{k\times n}$ to denote the Moore-Penrose pseudoinverse of $\mathbf{A}_{R,S}$, i.e., $\mathbf{A}_{R,S}^{\dagger}=(\mathbf{A}_{R,S})^{\dagger}$. If $S$ or $R$ is empty, then we consider $\mathbf{A}_{R,S}$ as an empty matrix, and we set $\mathbf{A}_{R,S}^{\dagger}\mathbf{A}_{R,S}=\mathbf{0}$ and $\mathbf{A}_{R,S}\mathbf{A}_{R,S}^{\dagger}=\mathbf{0}$.
 For convenience, we set $\det[\mathbf{A}_{\emptyset,\emptyset}]=1$ for any matrix $\mathbf{A}$.
 
We denote by $\mathbb{R}[z_1,\ldots,z_n]$ the family of multivariate polynomials in $z_1,\ldots,z_n$ with only real coefficients. We use $\partial_{z_i}$ to indicate the partial differential $\partial\slash\partial_{z_i}$. For each subset $S\subset[n]$, we write $\mathbf{Z}^S:=\prod_{i\in S}z_i$ and $\partial_{\mathbf{Z}^S}:=\prod_{i\in S}\partial_{z_i}$. Similarly, for each vector $\mathbf{b}=(b_1,\ldots,b_n)\in\mathbb{R}^n$ and for each subset $S\subset[n]$, we use $\mathbf{b}^S$ to denote the real number $\prod_{i\in S}b_i$.

\subsection{Linear Algebra}

 In this subsection, we will introduce several lemmas from linear algebra that are essential for our proof. It is well known that, for any two matrices $\mathbf{A}\in\mathbb{R}^{n\times d}$ and $\mathbf{B}\in\mathbb{R}^{d\times n}$, we have $\text{\rm Tr}[\mathbf{A}\mathbf{B}]=\text{\rm Tr}[\mathbf{B}\mathbf{A}]$ \cite[Chapter 3]{Meyer00}, and by the Weinstein-Aronszajn identity we have
\begin{equation}\label{Weinstein-Aronszajn}
\det[\mathbf{I}_d-\mathbf{B}\mathbf{A}]=\det[\mathbf{I}_n-\mathbf{A}\mathbf{B}]
\end{equation}
and
\begin{equation}\label{Weinstein-Aronszajn2}
\det[x\cdot \mathbf{I}_d-\mathbf{B}\mathbf{A}]=x^{d-n}\det[x\cdot \mathbf{I}_n-\mathbf{A}\mathbf{B}].
\end{equation}
Let $\mathbf{Z}=\mathrm{diag}(z_1,\ldots,z_n)$ be a diagonal matrix. For each square matrix $\mathbf{A}\in\mathbb{R}^{n\times n}$, we can expand $\det[\mathbf{Z}-\mathbf{A}]$ as \cite{ravi2}
\begin{equation}\label{expand det}
\det[\mathbf{Z}-\mathbf{A}]=\sum_{S\subset[n]}(-1)^{\vert S\vert}\cdot \det[\mathbf{A}_{S,S}]\cdot \mathbf{Z}^{S^C},
\end{equation}
where $\mathbf{Z}^{S^C}=\prod_{i\notin S}z_i$. Substituting $z_1=\cdots=z_n=x$ into \eqref{expand det}, we further have
\begin{equation}\label{expand det-x}
\det[x\cdot \mathbf{I}_n-\mathbf{A}]=\sum_{k=0}^{n}(-1)^{k}\cdot e_i(\mathbf{A})\cdot  x^{n-k},
\end{equation}
where $e_i(\mathbf{A})$ denotes the sum of all principal $k\times k$ minors of $\mathbf{A}$.

The following lemma is also useful for our argument, which is essentially proved in \cite[Lemma 4.2]{inter2}.

\begin{lemma}
\label{rank-one takeout}
{\rm {\cite[Lemma 4.2]{inter2}}} Let $\mathbf{A}\in\mathbb{R}^{d\times d}$ be a square matrix and let $\mathbf{v}\in\mathbb{R}^d$ be a random vector. Then we have
\begin{equation*}
\mathbb{E}\ \det[\mathbf{A}+\mathbf{v}\mathbf{v}^{\rm T}]=(1+\partial_t)\ \det[\mathbf{A}+t\cdot \mathbb{E}\mathbf{v}\mathbf{v}^{\rm T}]\ \Big|_{t=0}.
\end{equation*}	
\end{lemma}

We next introduce some basic properties of the Moore-Penrose pseudoinverse that will be utilized later.

\begin{lemma}\label{lemma2.1}
{\rm {\cite[Theorem 4.5]{YTT11}}} Let $\mathbf{B}\in\mathbb{R}^{n\times d}$ and $\mathbf{C}\in\mathbb{R}^{n\times m}$. Set $\mathbf{M}:=[\mathbf{B},\mathbf{C}]\in\mathbb{R}^{n\times (d+m)}$ and $\mathbf{Q}=\mathbf{I}_n-\mathbf{B}\mathbf{B}^{\dagger}$. Then we have
\begin{equation*}
\mathbf{M}\mathbf{M}^{\dagger}=\mathbf{B}\mathbf{B}^{\dagger}+(\mathbf{Q}\mathbf{C})(\mathbf{Q}\mathbf{C})^{\dagger}.
\end{equation*}

\end{lemma}

\begin{lemma}\label{lemma2.2}
{\rm {\cite[Lemma 12]{DR10}}} Let $\mathbf{B}\in\mathbb{R}^{n\times d}$ and $\mathbf{C}\in\mathbb{R}^{n\times m}$.
Set $\mathbf{M}=[\mathbf{B},\mathbf{C}]\in\mathbb{R}^{n\times (d+m)}$ and $\mathbf{Q}=\mathbf{I}_n-\mathbf{B}\mathbf{B}^{\dagger}$.
Then, we have
\begin{equation*}
\det[\mathbf{M}^{\rm T} \mathbf{M}]=\det[\mathbf{C}^{\rm T}\mathbf{Q} \mathbf{C}]\cdot \det[\mathbf{B}^{\rm T} \mathbf{B}].
\end{equation*}	
\end{lemma}

In the following lemma we present a proof of equation \eqref{estimate deltak}.

\begin{lemma}\label{lemma-estimate deltak}
Assume that $b_1,\ldots,b_m$ are real numbers satisfying $b_1\geq b_2\geq \cdots\geq b_m>0$. Then we have
\begin{equation*}
1-\frac{k}{m}\leq \delta_k:=\frac{\sum_{S\subset[m-1],\vert S\vert=k} \mathbf{b}^{S}}{	 \sum_{S\subset[m],\vert S\vert=k} \mathbf{b}^{S}}\leq \min \Big\{1, \frac{b_1}{b_m}\cdot (1-\frac{k}{m})\Big\}.
\end{equation*}
	
\end{lemma}

\begin{proof}
Note that
\begin{equation*}
\sum_{S\subset[m],\vert S\vert=k} \mathbf{b}^{S}>
\sum_{S\subset[m-1],\vert S\vert=k} \mathbf{b}^{S}
\geq \frac{1}{m}\sum_{j=1}^m\sum_{S\subset[m]\backslash\{j\},\vert S\vert=k} \mathbf{b}^{S}
=\frac{m-k}{m} \sum_{S\subset[m],\vert S\vert=k} \mathbf{b}^{S},
\end{equation*}	
so we have $1-\frac{k}{m}\leq \delta_k< 1$.
It remains to prove $\delta_k\leq \frac{b_1}{b_m}\cdot (1-\frac{k}{m})$. Observe that
\begin{equation*}
{\sum_{S\subset[m-1],\vert S\vert=k} \mathbf{b}^{S}}
={\frac{1}{k}	  \sum_{S\subset[m-1],\vert S\vert=k-1}\Big(\sum_{i\notin S}   b_i\Big)\cdot \mathbf{b}^{S}}
\leq \frac{(m-k)\cdot b_1}{k}\cdot \sum_{S\subset[m-1],\vert S\vert=k-1}\mathbf{b}^{S}.
\end{equation*}
Therefore, we arrive at
\begin{equation*}
\begin{aligned}
\delta_k
&=\frac{\sum_{S\subset[m-1],\vert S\vert=k} \mathbf{b}^{S}}{	 \sum_{S\subset[m],\vert S\vert=k} \mathbf{b}^{S}}
=\frac{\sum_{S\subset[m-1],\vert S\vert=k} \mathbf{b}^{S}}{	 \sum_{S\subset[m-1],\vert S\vert=k} \mathbf{b}^{S}+b_m\cdot \sum_{S\subset[m-1],\vert S\vert=k-1} \mathbf{b}^{S}}\\
&=\frac{1}{	1+b_m\cdot \frac{\sum_{S\subset[m-1],\vert S\vert=k-1} \mathbf{b}^{S}}{\sum_{S\subset[m-1],\vert S\vert=k} \mathbf{b}^{S}}}\leq \frac{1}{1+\frac{b_m}{(\frac{m}{k}-1) b_1}}
\leq \frac{b_1}{b_m}\cdot (1-\frac{k}{m}).			
\end{aligned}
\end{equation*}
This completes the proof.

\end{proof}

\subsection{Interlacing and Real Stability}

Recall that a univariate real polynomial, denoted as $p(x)$, is said to be {\em real-rooted} if all of its roots are real. We next introduce the definition  of common interlacing.

\begin{definition}\label{def2.1}
Assume that $f(x)=a_0\cdot \prod_{i=1}^{d-1}(x-a_i)$ and $p(x)=b_0\cdot \prod_{i=1}^{d}(x-b_i)$ are two real-rooted polynomials. We say that $f$ interlaces $p$ if
\begin{equation*}
b_1\leq a_1\leq b_2\leq a_2\leq \cdots \leq a_{d-1}\leq b_{d}.
\end{equation*}
We say that a collection of real-rooted polynomials $p_1(x),\ldots,p_m(x)$ have a common interlacing if there exists a polynomial $f(x)$ such that $f(x)$ interlaces $p_i(x)$ for each $i\in[m]$.
\end{definition}

Given a collection of polynomials $p_1(x),\ldots,p_m(x)$, it is easy to see that if any two of them have a common interlacing, then $p_1(x),\ldots,p_m(x)$ also have a common interlacing. By combining this with the result of Fell \cite[Theorem 2]{Fell}, we can conclude that to prove a collection of polynomials $p_1(x),\ldots,p_m(x)$ have a common interlacing, it is sufficient to prove that the polynomial $\mu\cdot p_i(x)+(1-\mu)\cdot p_j(x)$ is real-rooted for any two distinct integers $i,j\in[m]$ and for any $\mu\in[0,1]$.

When a set of polynomials have a common interlacing, we have the following lemma:
\begin{lemma}\label{lemma2.6}
{\rm {\cite[Lemma 4.2]{inter1}}} Assume that $p_1(x),\ldots,p_m(x)$ are real-rooted polynomials with the same degree and positive leading coefficients. If $p_1(x),\ldots,p_m(x)$ have a common interlacing, then there exists an integer $j\in[m]$ such that
\begin{equation*}
\text{\rm maxroot}\ p_j(x)\leq \text{\rm maxroot}\ \sum\limits_{i=1}^{m}p_i(x).
\end{equation*}
\end{lemma}

We next introduce the notion of real stable polynomials, which generalizes the definition of real-rooted polynomials to the multivariate case. For more details about real stable polynomials we refer to \cite{Branden01,wag}.

\begin{definition}\label{defstable}
A multivariate polynomial $p\in\mathbb{R}[z_1,\ldots, z_n]$ is said to be \emph{real stable} if  $p(z_1,\ldots,z_n)\neq 0$ for all $(z_1,\ldots,z_n)\in\mathbb{C}^n$ with $\mathbf{Im}(z_i)>0, i=1,\ldots,n$.
\end{definition}

A univariate polynomial is real stable if and only if it is real-rooted. The following two lemmas will help us to prove that the polynomials of interest in this paper are real stable.

\begin{lemma}\label{real stable}
{\rm {\cite[Proposition 2.4]{Branden}}} For any Hermitian matrix $\mathbf{B}\in\mathbb{C}^{d\times d}$ and any positive semidefinite Hermitian matrices $\mathbf{A}_1,\ldots,\mathbf{A}_n\in\mathbb{C}^{d\times d}$, the polynomial
\begin{equation*}
\det[\mathbf{A}_1 z_1+\cdots+\mathbf{A}_n z_n+\mathbf{B}]\in\mathbb{R}[z_1,\ldots,z_n]
\end{equation*}
is real stable in $z_1,\ldots,z_n$ if it is not identically zero.
\end{lemma}

\begin{lemma}\label{real stable2}
Let $r_1,\ldots,r_n$ be a collection of nonnegative integers, and let $p \in \mathbb{R}[z_1,\ldots, z_n]$ be a real stable polynomial of degree at most $r_i$ in $z_i$, $i=1,\ldots,n$. Then the following polynomials are real stable if they are not identically zero:
\begin{enumerate}[{\rm (i)}]
 \item $p(z_1, z_2 ,\ldots, z_n)|_{z_1=a}\in\mathbb{R}[z_2,z_3,\ldots,z_n], \text{ for any } a\in\mathbb{R}$;
 \item $p(z_1, z_2 ,\ldots, z_n)|_{z_1=z_2=x}\in\mathbb{R}[x,z_3,\ldots,z_n]$;
 \item $(\sum_{i=1}^{n}a_i\cdot \partial_{z_i})\cdot p(z_1,\ldots,z_n),\ \text{ for any } a_1\geq 0,\ldots, a_n\geq 0$;
 \item $(1\pm \partial_{z_i})\cdot p(z_1,\ldots,z_n),\ \text{ for any } i\in[n]$.
\end{enumerate}
\end{lemma}
\begin{proof}
The assertions (i), (ii), (iii), and (iv) follow from \cite[Lemma 2.4 (d)]{wag}, Definition \ref{defstable}, \cite[Theorem 1.3]{Branden01}, and \cite[Corollary 3.8]{inter2} respectively.
\end{proof}

\section{Properties of the Expected Polynomial  $P_{k,r}(x;\mathbf{A},\mathbf{B},\mathbf{C})$}\label{section: basic prop}

In this section, we present some basic properties of the expected polynomial $P_{k,r}(x;\mathbf{A},\mathbf{B},\mathbf{C})$ defined in Definition \ref{def2}.
Recall that for a matrix $\mathbf{M}$ and two subsets $S$ and $R$, we use $\mathbf{M}_{R,S}^{\dagger}$ to denote the Moore-Penrose pseudoinverse of the submatrix $\mathbf{M}_{R,S}$, i.e., $\mathbf{M}_{R,S}^{\dagger}=(\mathbf{M}_{R,S})^{\dagger}$.
Throughout this paper, when the source matrices $\mathbf{B}\in\mathbb{R}^{n\times d_{{B}}}$ and $\mathbf{C}\in\mathbb{R}^{n_{{C}}\times d}$ are specified,  we will denote the projection matrices $\mathbf{Q}_S\in\mathbb{R}^{n\times n}$ and $\mathbf{P}_R\in\mathbb{R}^{d\times d}$ as follows:
\begin{equation}\label{QSPR}
\mathbf{Q}_S:=\mathbf{I}_n-\mathbf{B}_{:,S}\mathbf{B}_{:,S}^{\dagger}\quad \text{and}\quad \mathbf{P}_R:=\mathbf{I}_d-\mathbf{C}_{R,:}^{\dagger}\mathbf{C}_{R,:}=\mathbf{I}_d-(\mathbf{C}^{\rm T})_{:,R}((\mathbf{C}^{\rm T})_{:,R})^{\dagger},
\end{equation}
where $S\subset [d_B]$ and $R\subset [n_C]$ are two subsets.

We start with proving a simple but useful property of $P_{k,r}(x;\mathbf{A},\mathbf{B},\mathbf{C})$.

\begin{prop}\label{P-sym}
Let $\mathbf{A}\in\mathbb{R}^{n\times d}$, $\mathbf{B}\in\mathbb{R}^{n\times d_{{B}}}$ and $\mathbf{C}\in\mathbb{R}^{n_{{C}}\times d}$. Let $k$ and $r$ be two integers such that $k\in[d_B]$ and $r\in[n_C]$. Then we have
\begin{equation*}
P_{k,r}(x;\mathbf{A},\mathbf{B},\mathbf{C})=x^{d-n}\cdot P_{r,k}(x;\mathbf{A}^{\rm T},\mathbf{C}^{\rm T},\mathbf{B}^{\rm T}).
\end{equation*}
\end{prop}

\begin{proof}
By the Weinstein-Aronszajn identity \eqref{Weinstein-Aronszajn2}, we can express $p_{S,R}(x;\mathbf{A},\mathbf{B},\mathbf{C})$ as follows:
\begin{equation}\label{eq:proofp31}
\begin{aligned}
p_{S,R}(x;\mathbf{A},\mathbf{B},\mathbf{C})&=\det[x\cdot \mathbf{I}_{d}-  \mathbf{P}_R \mathbf{A}^{\rm T}\mathbf{Q}_S\mathbf{A} \mathbf{P}_R]\\
&=x^{d-n}\cdot \det[x\cdot \mathbf{I}_{n}- \mathbf{Q}_S\mathbf{A} \mathbf{P}_R\mathbf{A}^{\rm T}\mathbf{Q}_S]\\
&=x^{d-n}\cdot p_{R,S}(x;\mathbf{A}^{\rm T},\mathbf{C}^{\rm T},\mathbf{B}^{\rm T}).
\end{aligned}
\end{equation}
Substituting (\ref{eq:proofp31}) into the definition of $P_{k,r}(x;\mathbf{A},\mathbf{B},\mathbf{C})$ gives the desired result.
\end{proof}

We next introduce two multivariate polynomials that have a close relationship with $P_{k,r}(x;\mathbf{A},\mathbf{B},\mathbf{C})$.

\begin{definition}
Let $\mathbf{A}\in\mathbb{R}^{n\times d}$, $\mathbf{B}\in\mathbb{R}^{n\times d_{{B}}}$, and $\mathbf{C}\in\mathbb{R}^{n_{{C}}\times d}$. Let $\mathbf{Y}=\mathrm{diag}(y_1,\ldots,y_{d_B})$ and $\mathbf{Z}=\mathrm{diag}(z_1,\ldots,z_{n_C})$, where $y_i$ and $z_i$ are variables. We define the multivariate polynomials $H(x,w,\mathbf{Y},\mathbf{Z};\mathbf{A},\mathbf{B},\mathbf{C})$ as:
\begin{equation*}
H(x,w,\mathbf{Y},\mathbf{Z};\mathbf{A},\mathbf{B},\mathbf{C}):=	\det\left[\begin{matrix}
w\cdot \mathbf{I}_n  & \mathbf{0} & \mathbf{B} & \mathbf{A}  \\
 \mathbf{0}& \mathbf{Z} & \mathbf{0}  & \mathbf{C}  \\
\mathbf{B}^{\rm T}  &  \mathbf{0} & \mathbf{Y} &  \mathbf{0} \\
\mathbf{A}^{\rm T}  & \mathbf{C}^{\rm T} & \mathbf{0} & x\cdot  \mathbf{I}_d
\end{matrix}\right]
\in\mathbb{R}[x,w,y_1,\ldots,y_{d_B},z_1,\ldots,z_{n_C}].
\end{equation*}
If we set $\mathbf{Y}=y\cdot \mathbf{I}_{d_B}$ and $\mathbf{Z}=z\cdot \mathbf{I}_{n_C}$, the polynomial
$H(x,w,\mathbf{Y},\mathbf{Z};\mathbf{A},\mathbf{B},\mathbf{C})$  simplifies to:
\begin{equation*}
H(x,y,z,w;\mathbf{A},\mathbf{B},\mathbf{C}):=H(x,w,y\cdot \mathbf{I}_{d_B}, z\cdot \mathbf{I}_{n_C};\mathbf{A},\mathbf{B},\mathbf{C})
\in\mathbb{R}[x,y,z,w].
\end{equation*}

\end{definition}

In the following proposition, we show that each polynomial $p_{S,R}(x;\mathbf{A},\mathbf{B},\mathbf{C})$ defined in Definition \ref{def2} can be expressed in terms of the multivariate polynomial $H(x,w,\mathbf{Y},\mathbf{Z};\mathbf{A},\mathbf{B},\mathbf{C})$.

\begin{prop}\label{P-exp-base}
Let $\mathbf{A}\in\mathbb{R}^{n\times d}$, $\mathbf{B}\in\mathbb{R}^{n\times d_{{B}}}$ and $\mathbf{C}\in\mathbb{R}^{n_{{C}}\times d}$.
 Let $\mathbf{Y}=\mathrm{diag}(y_1,\ldots,y_{d_B})$ and $\mathbf{Z}=\mathrm{diag}(z_1,\ldots,z_{n_C})$, where $y_i$ and $z_i$ are variables.
 Let $k$ and $r$ be two integers satisfying $0\leq k\leq \mathrm{rank}(\mathbf{B})$ and $0\leq r\leq \mathrm{rank}(\mathbf{C})$. For any subset $S\subset[d_B]$ of size $k$ and for any subset $R\subset[n_C]$ of size $r$, we have
\begin{equation}\label{eq:P-exp-base}
\begin{aligned}
&\det[\mathbf{C}_{R,:}(\mathbf{C}_{R,:})^{\rm T}] \cdot \det[(\mathbf{B}_{:,S})^{\rm T}\mathbf{B}_{:,S}]\cdot p_{S,R}(x\cdot w;\mathbf{A},\mathbf{B},\mathbf{C})\\
&=(-1)^{k+r}\cdot x^r\cdot w^{d+k-n}\cdot \partial_{\mathbf{Y}^{S^C}} \partial_{\mathbf{Z}^{R^C}}\ H(x,w,\mathbf{Y},\mathbf{Z};\mathbf{A},\mathbf{B},\mathbf{C})  \ \Bigg\vert_{\substack{y_i=0,\forall i\in[d_B]\\z_j=0,\forall j\in[n_C]}}.
\end{aligned}
\end{equation}
Here, $\partial_{\mathbf{Y}^{S^C}}:=\prod_{i\in [d_B]\backslash S}\partial_{y_i}$ and $\partial_{\mathbf{Z}^{R^C}}:=\prod_{i\in [n_C]\backslash R}\partial_{z_i}$.
\end{prop}

\begin{proof}
See Appendix \ref{proof of P-exp-base}.		
\end{proof}

With the help of Proposition \ref{P-exp-base}, we derive several equivalent expressions for $P_{k,r}(x;\mathbf{A},\mathbf{B},\mathbf{C})$.

\begin{prop}\label{P-three-exp}
Let $\mathbf{A}\in\mathbb{R}^{n\times d}$, $\mathbf{B}\in\mathbb{R}^{n\times d_{{B}}}$, and $\mathbf{C}\in\mathbb{R}^{n_{{C}}\times d}$. Let $k$ and $r$ be two integers satisfying $0\leq k\leq \mathrm{rank}(\mathbf{B})$ and $0\leq r\leq \mathrm{rank}(\mathbf{C})$. Then we have
\begin{equation}\label{P-exp}
P_{k,r}(x;\mathbf{A},\mathbf{B},\mathbf{C})=\frac{(-1)^{k+r}}{ (d_B-k)!\cdot (n_C-r)!}\cdot x^{r} \cdot \partial_y^{d_B-k}\partial_z^{n_C-r}\ H(x,y,z,1;\mathbf{A},\mathbf{B},\mathbf{C})  \ \Big\vert_{y=z=0}.
\end{equation}
Furthermore, we have the following two alternative expressions for $P_{k,r}(x;\mathbf{A},\mathbf{B},\mathbf{C})$:
\begingroup\fontsize{10pt}{12pt}\selectfont
\begin{subequations}
\begin{align}
P_{k,r}(x;\mathbf{A},\mathbf{B},\mathbf{C})&=\frac{(-1)^{k}\cdot x^{r}}{ r!(d_B-k)!} \cdot \partial_y^{d_B-k}\partial_z^{r}\ \det\left[
\begin{pmatrix}
x\cdot \mathbf{I}_d +z\cdot
\mathbf{C}^{\rm T}\mathbf{C}  & \mathbf{0}  \\
 \mathbf{0} & y\cdot \mathbf{I}_{d_B}
\end{pmatrix}
-
\begin{pmatrix}
\mathbf{A}^{\rm T}   \\
\mathbf{B}^{\rm T}
\end{pmatrix}
\begin{pmatrix}
\mathbf{A} & \mathbf{B}
\end{pmatrix}
\right]    \ \Bigg\vert_{y=z=0}\label{P-exp-1}\\
P_{k,r}(x;\mathbf{A},\mathbf{B},\mathbf{C})&=\frac{(-1)^{r}\cdot x^{d-n+k}}{k! (n_C-r)!} \cdot \partial_y^{k}\partial_z^{n_C-r}\ \det\left[
\begin{pmatrix}
x\cdot \mathbf{I}_n +y\cdot
\mathbf{B}\mathbf{B}^{\rm T}  &  \mathbf{0} \\
\mathbf{0}  & z\cdot \mathbf{I}_{n_C}
\end{pmatrix}
-
\begin{pmatrix}
\mathbf{A}  \\
\mathbf{C}
\end{pmatrix}
\begin{pmatrix}
\mathbf{A}^{\rm T}  & \mathbf{C}^{\rm T}
\end{pmatrix}
\right]    \ \Bigg\vert_{y=z=0}\label{P-exp-2}.
\end{align}
\end{subequations}
\endgroup
\end{prop}

\begin{proof}
See Appendix \ref{proof of P-three-exp}.
\end{proof}

\begin{remark}\label{remark:realrootedness of Pkr}
Lemma \ref{real stable} establishes that $H(x,y,z,1;\mathbf{A},\mathbf{B},\mathbf{C})$ is real stable in the variables $x, y$, and $z$. Lemma \ref{real stable2} further demonstrates that the real stability is preserved by the differential operators $\partial_y$, $\partial_z$ and the setting $y=z=0$. Consequently, by applying equation \eqref{P-exp}, we can deduce that $P_{k,r}(x;\mathbf{A},\mathbf{B},\mathbf{C})$ is real-rooted when $0\leq k\leq \mathrm{rank}(\mathbf{B})$ and $0\leq r\leq \mathrm{rank}(\mathbf{C})$.
Moreover, note that $P_{k,r}(x;\mathbf{A},\mathbf{B},\mathbf{C})$ is a convex linear combination of real-rooted polynomials $p_{S,R}(x;\mathbf{A},\mathbf{B},\mathbf{C})$ whose leading coefficients and roots are all nonnegative.  Therefore, the roots of $P_{k,r}(x;\mathbf{A},\mathbf{B},\mathbf{C})$ are all nonnegative.
\end{remark}

In the following proposition we simplify the expression of $P_{k,r}(x;\mathbf{A},\mathbf{B},\mathbf{C})$ for the case when $r=0$.

\begin{prop}\label{P-exp-r=0}
Let $\mathbf{A}\in\mathbb{R}^{n\times d}$ and $\mathbf{B}\in\mathbb{R}^{n\times d_{{B}}}$. For each  integer $k\in [0, \mathrm{rank}(\mathbf{B})]$,
we can derive the equivalent expressions for $P_{k}(x;\mathbf{A},\mathbf{B})$ as follows:
\begin{subequations}
\begin{align}
P_{k}(x;\mathbf{A},\mathbf{B})&=\frac{(-1)^k}{ (d_{{B}}-k)!}\cdot \partial_y^{d_{{B}}-k}\ \det\left[
\begin{pmatrix}
x\cdot \mathbf{I}_d  &  \mathbf{0} \\
\mathbf{0}  & y\cdot \mathbf{I}_{d_B}
\end{pmatrix}
-
\begin{pmatrix}
\mathbf{A}^{\rm T}   \\
\mathbf{B}^{\rm T}
\end{pmatrix}
\begin{pmatrix}
\mathbf{A} & \mathbf{B}
\end{pmatrix}
\right] \ \Bigg|_{y=0}\label{lemma_f_k:expression1},\\
P_{k}(x;\mathbf{A},\mathbf{B})&=\frac{(-1)^k }{k!}\cdot x^{d-n+k} \cdot \partial_y^{k} \ \det[x\cdot \mathbf{I}_n-\mathbf{A}\mathbf{A}^{\rm T}-y\cdot  \mathbf{B}\mathbf{B}^{\rm T}]	 \ \big|_{y=0}\label{lemma_f_k:expression2}.		
\end{align}	
\end{subequations}
 Here,  the definition of $P_{k}(x;\mathbf{A},\mathbf{B})$ is given in Definition \ref{def2}, i.e., $P_{k}(x;\mathbf{A},\mathbf{B}):=P_{k,0}(x;\mathbf{A},\mathbf{B},\mathbf{0})$.
\end{prop}

\begin{proof}
See Appendix \ref{proof of P-exp-r=0}.
\end{proof}

\begin{remark}\label{CSS-GCSS}
Consider the special case when $\mathbf{B}=\mathbf{A}\in\mathbb{R}^{n\times d}$. In this case we have
\begin{equation}\label{remark32:eq1}
\begin{aligned}
& \partial_y^{k} \ \det[x\cdot \mathbf{I}_n-\mathbf{A}\mathbf{A}^{\rm T}-y\cdot  \mathbf{B}\mathbf{B}^{\rm T}]\\
&=\partial_y^{k}\ \det[x\cdot \mathbf{I}_n-(y+1)\cdot \mathbf{A}\mathbf{A}^{\rm T}] \ \big|_{y=0} =x^n\cdot  \partial_y^{k}\ \det[\mathbf{I}_n-\frac{y+1}{x}\cdot \mathbf{A}\mathbf{A}^{\rm T}] \ \big|_{y=0}\\
&\overset{(a)}=x^{n-k}\cdot\partial_y^{k}\ \det[\mathbf{I}_n-y\cdot \mathbf{A}\mathbf{A}^{\rm T}] \ \big|_{y=1/x} = x^{n-k-d}\cdot \mathcal{R}_{x,d}^+\ \partial_x^{k}\ \det[\mathbf{I}_n-x\cdot \mathbf{A}\mathbf{A}^{\rm T}] \\
&\overset{(b)}=x^{n-k-d}\cdot \mathcal{R}_{x,d}^+\ \partial_x^{k}\ \det[\mathbf{I}_d-x\cdot \mathbf{A}^{\rm T}\mathbf{A}]=x^{n-k-d}\cdot \mathcal{R}_{x,d}^+\ \partial_x^{k}\ \mathcal{R}_{x,d}^+\ \det[x\cdot\mathbf{I}_d-  \mathbf{A}^{\rm T}\mathbf{A}],
\end{aligned}
\end{equation}
where the flip operator $\mathcal{R}_{x,d}^+$ is defined as $\mathcal{R}_{x,d}^+\ p(x):=x^d\cdot p(1/x)$ for any polynomial $p(x)$ of degree at most $d$. Equation $(a)$ follows from the application of the chain rule, while equation $(b)$ is derived using the Weinstein-Aronszajn identity \eqref{Weinstein-Aronszajn}.
Then, substituting \eqref{remark32:eq1} into \eqref{lemma_f_k:expression2} we have
\begin{equation*}
P_k(x;\mathbf{A},\mathbf{A})	=\frac{(-1)^k}{k!}\cdot\mathcal{R}_{x,d}^+\cdot  \partial_x^{k}\cdot  \mathcal{R}_{x,d}^+ \det[x\cdot\mathbf{I}_d-  \mathbf{A}^{\rm T}\mathbf{A}],\\
\end{equation*}
which recovers the first result in \cite[Proposition 3.1]{CXX23}.

\end{remark}

In the following proposition, we simplify the expression of $P_{k,r}(x;\mathbf{A},\mathbf{B},\mathbf{C})$ specifically for the case when $\mathbf{A}\in\mathbb{R}^{d\times d}$ is a square matrix and $\mathbf{B}=\mathbf{C}=\mathbf{I}_d$.

\begin{prop}\label{P-exp-I}
Let $\mathbf{A}\in\mathbb{R}^{d\times d}$ and $\mathbf{B}=\mathbf{C}=\mathbf{I}_d$. Let $k,r\leq d$  be two nonnegative integers. 
Let $S\subset[d]$ be a subset of size $|S|\leq k$, 
and let $R\subset[d]$ be a subset of size $|R|\leq r$.
Then we have
\begingroup\fontsize{10pt}{12pt}\selectfont
\begin{equation}\label{P-exp-II-oct}
P_{k-|S|,r-|R|}(x;\mathbf{Q}_{S}\mathbf{A}\mathbf{P}_{R},\mathbf{Q}_{S}\mathbf{B},\mathbf{C}\mathbf{P}_{R})
=\frac{x^k}{(k-|S|)!(r-|R|)!} \partial_x^{k-|S|}
\big( x^{r-|S|} \partial_x^{r-|R|} \det[x\cdot\mathbf{I}_{d-|R|}-(\mathbf{A}_{S^C,R^C})^{\rm T}\mathbf{A}_{S^C,R^C} ]\big),
\end{equation}
\endgroup
where the projection matrices $\mathbf{Q}_S=\mathbf{I}_d-\mathbf{B}_{:,S}\mathbf{B}_{:,S}^{\dagger}$ and $\mathbf{P}_R=\mathbf{I}_d-\mathbf{C}_{R,:}^{\dagger}\mathbf{C}_{R,:}$ are defined as in \eqref{QSPR}.
In particular, we have
\begin{equation}\label{P-exp-II}
P_{k,r}(x;\mathbf{A},\mathbf{I}_d,\mathbf{I}_d)=\frac{1}{k!\cdot r!}\cdot x^k	\cdot\partial_x^{k}\cdot \big( x^r\cdot \partial_x^{r}\ \det[x\cdot \mathbf{I}_d-\mathbf{A}^{\rm T}\mathbf{A}]\big)
\end{equation}
and
\begin{equation}\label{P-exp-II-square}
P_{k,k}(x;\mathbf{A},\mathbf{I}_d,\mathbf{I}_d)=\frac{1}{(k!)^2}\cdot x^k\cdot   (\partial_x\cdot x\cdot \partial_x)^k\  \det[x\cdot \mathbf{I}_{d}-\mathbf{A}^{\rm T}\mathbf{A}].
\end{equation}

\end{prop}
\begin{proof}
See Appendix \ref{proof of P-exp-I}.
\end{proof}

We next present a recursive formula for $P_{k,r}(x;\mathbf{A},\mathbf{B},\mathbf{C})$, which plays a key role in the proof of Theorem \ref{mth1-spr}.

\begin{prop}\label{P-recursive2}
Let $\mathbf{A}\in\mathbb{R}^{n\times d}$, $\mathbf{B}=[\mathbf{b}_1,\ldots,\mathbf{b}_{d_B}]\in\mathbb{R}^{n\times d_{{B}}}$, and $\mathbf{C}=[\mathbf{c}_1,\ldots,\mathbf{c}_{n_C}]^{\rm T}\in\mathbb{R}^{n_{{C}}\times d}$.
Let $k$ and $r$ be two integers satisfying $1\leq k\leq \mathrm{rank}(\mathbf{B})$ and $1\leq r\leq \mathrm{rank}(\mathbf{C})$.
Let $S$ be an $l$-subset of $[d_B]$, where $0\leq l\leq k-1$, and let $R$ be a $t$-subset of $[n_C]$, where $0\leq t\leq r-1$.
Then we have
\begin{equation}\label{P-recursive2-AB}
P_{k-l,r-t}(x;\mathbf{Q}_{S}\mathbf{A}\mathbf{P}_{R},\mathbf{Q}_{S}\mathbf{B},\mathbf{C}\mathbf{P}_{R})=\frac{1}{k-l}	\sum_{i:\Vert \mathbf{Q}_{S}\mathbf{b}_{i} \Vert\neq 0}^{}\Vert \mathbf{Q}_{S}\mathbf{b}_{i} \Vert^2\cdot P_{k-l-1,r-t}(x;\mathbf{Q}_{S\cup\{i\}}\mathbf{A}\mathbf{P}_{R},\mathbf{Q}_{S\cup\{i\}}\mathbf{B},\mathbf{C} \mathbf{P}_{R})
\end{equation}	
and
\begin{equation}\label{P-recursive2-AC}
P_{k-l,r-t}(x;\mathbf{Q}_{S}\mathbf{A}\mathbf{P}_{R},\mathbf{Q}_{S}\mathbf{B},\mathbf{C}\mathbf{P}_{R})=\frac{1}{r-t}	\sum_{i:\Vert \mathbf{c}_{i}^{\rm T}\mathbf{P}_{R}\Vert\neq 0}^{}\Vert \mathbf{c}_{i}^{\rm T}\mathbf{P}_{R}\Vert^2\cdot P_{k-l,r-t-1}(x;\mathbf{Q}_{S}\mathbf{A}\mathbf{P}_{R\cup\{i\}},\mathbf{Q}_{S}\mathbf{B},\mathbf{C} \mathbf{P}_{R\cup\{i\}}).
\end{equation}	
Here, the projection matrices $\mathbf{Q}_S$ and $\mathbf{P}_R$ are defined in \eqref{QSPR}.
\end{prop}

\begin{proof}
See Appendix \ref{proof of P-recursive2}.	
\end{proof}

\section{Proof of Theorem \ref{mth1-spr}  Utilizing the Method of Interlacing Polynomials}

The aim of this section is to prove Theorem \ref{mth1-spr}. With the help of Proposition \ref{P-three-exp} and \ref{P-recursive2}, we can prove the following lemma using the method of interlacing polynomials.

\begin{lemma}\label{lemma:interlacing}
Let $\mathbf{A}\in\mathbb{R}^{n\times d}$, $\mathbf{B}=[\mathbf{b}_1,\ldots,\mathbf{b}_{d_B}]\in\mathbb{R}^{n\times d_{{B}}}$, and $\mathbf{C}=[\mathbf{c}_1,\ldots,\mathbf{c}_{n_C}]^{\rm T}\in\mathbb{R}^{n_{{C}}\times d}$.
Let $k$ and $r$ be two integers satisfying $1\leq k\leq \mathrm{rank}(\mathbf{B})$ and $1\leq r\leq \mathrm{rank}(\mathbf{C})$.
Let $S$ be an $l$-subset of $[d_B]$ such that $\mathrm{rank}(\mathbf{B}_{:,S})=l$, and let $R$ be a $t$-subset of $[n_C]$ such that $\mathrm{rank}(\mathbf{C}_{R,:})=t$, where
  $l$ and $t$ are two integers satisfying $0\leq l\leq k-1$ and $0\leq t\leq r-1$.
Set $W:=\{i\in[d_B]\backslash S : \Vert \mathbf{Q}_S \mathbf{b}_{i}\Vert\neq0\}$ and $V:=\{i\in[n_C]\backslash R : \Vert \mathbf{c}_{i}^{\rm T}\mathbf{P}_R \Vert\neq0\}$. Then we have the following results.
\begin{enumerate}
\item[{\rm (i)}] The polynomials $P_{k-l-1,r-t}(x;\mathbf{Q}_{S\cup\{i\}}\mathbf{A}\mathbf{P}_{R},\mathbf{Q}_{S\cup\{i\}}\mathbf{B},\mathbf{C} \mathbf{P}_{R}), i\in W$ have a common interlacing. Moreover, there exists an integer $j\in W$, such that
\begin{equation}\label{lemma:interlacing-AB}
\mathrm{maxroot}\ 	P_{k-l-1,r-t}(x;\mathbf{Q}_{S\cup\{j\}}\mathbf{A}\mathbf{P}_{R},\mathbf{Q}_{S\cup\{j\}}\mathbf{B},\mathbf{C} \mathbf{P}_{R})\leq \mathrm{maxroot}\ 	P_{k-l,r-t}(x;\mathbf{Q}_{S}\mathbf{A}\mathbf{P}_{R},\mathbf{Q}_{S}\mathbf{B},\mathbf{C} \mathbf{P}_{R}).
\end{equation}
\item [{\rm (ii)}] The polynomials $P_{k-l,r-t-1}(x;\mathbf{Q}_{S}\mathbf{A}\mathbf{P}_{R\cup\{i\}},\mathbf{Q}_{S}\mathbf{B},\mathbf{C} \mathbf{P}_{R\cup\{i\}}), i\in V$ have a common interlacing. Moreover, there exists an integer $j\in V$, such that
\begin{equation*}
\mathrm{maxroot}\ 	P_{k-l,r-t-1}(x;\mathbf{Q}_{S}\mathbf{A}\mathbf{P}_{R\cup\{j\}},\mathbf{Q}_{S}\mathbf{B},\mathbf{C} \mathbf{P}_{R\cup\{j\}})\leq \mathrm{maxroot}\ 	P_{k-l,r-t}(x;\mathbf{Q}_{S}\mathbf{A}\mathbf{P}_{R},\mathbf{Q}_{S}\mathbf{B},\mathbf{C} \mathbf{P}_{R}).
\end{equation*}

\end{enumerate}

\end{lemma}

\begin{proof}
(i) We first prove that the polynomials
\[
P_{k-l-1,r-t}(x;\mathbf{Q}_{S\cup\{i\}}\mathbf{A}\mathbf{P}_{R},\mathbf{Q}_{S\cup\{i\}}\mathbf{B},\mathbf{C} \mathbf{P}_{R}),\qquad i\in W
\]
 have a common interlacing.
By Definition \ref{def2.1}, it is enough to show that for any two distinct integers $i_1,i_2\in W$, the polynomial  $P_{k-l-1,r-t}(x;\mathbf{Q}_{S\cup\{i_1\}}\mathbf{A}\mathbf{P}_{R},\mathbf{Q}_{S\cup\{i_1\}}\mathbf{B},\mathbf{C} \mathbf{P}_{R})$ and the polynomial $P_{k-l-1,r-t}(x;\mathbf{Q}_{S\cup\{i_2\}}\mathbf{A}\mathbf{P}_{R},\mathbf{Q}_{S\cup\{i_2\}}\mathbf{B},\mathbf{C} \mathbf{P}_{R})$ have a common interlacing. This is equivalent to show that the polynomial
\begin{equation*}
\begin{aligned}
P_{\mu}(x):=&\mu\cdot P_{k-l-1,r-t}(x;\mathbf{Q}_{S\cup\{i_1\}}\mathbf{A}\mathbf{P}_{R},\mathbf{Q}_{S\cup\{i_1\}}\mathbf{B},\mathbf{C} \mathbf{P}_{R})\\
&+(1-\mu)\cdot P_{k-l-1,r-t}(x;\mathbf{Q}_{S\cup\{i_2\}}\mathbf{A}\mathbf{P}_{R},\mathbf{Q}_{S\cup\{i_2\}}\mathbf{B},\mathbf{C} \mathbf{P}_{R})
\end{aligned}
\end{equation*}
is real-rooted for any $\mu\in[0,1]$ and for any two distinct integers $i_1,i_2\in W$.

By \eqref{P-exp-1} in Proposition \ref{P-three-exp}, we can write $P_{\mu}(x)$ as
\begin{equation}\label{eq:Pmu}
P_{\mu}(x)=\frac{(-1)^{k-l-1}\cdot x^{r-t}}{ (r-t)!(d_B-k+l+1)!} \cdot \partial_y^{d_B-k+l+1}\partial_z^{r-t}\  Q_{\mu}(x,y,z) \ \big|_{y=z=0},
\end{equation}
where
\begin{equation*}
\begin{aligned}
Q_{\mu}(x,y,z)
&:=\mu\cdot \det\left[
\begin{pmatrix}
x\cdot \mathbf{I}_d +z\cdot
(\mathbf{C}\mathbf{P}_{R})^{\rm T}\mathbf{C}\mathbf{P}_{R}  & \mathbf{0}  \\
\mathbf{0}  & y\cdot \mathbf{I}_{d_B}
\end{pmatrix}
-\mathbf{M}{(S\cup\{i_1\},R)}
\right]\\
&\quad +(1-\mu)\cdot \det\left[
\begin{pmatrix}
x\cdot \mathbf{I}_d +z\cdot
(\mathbf{C}\mathbf{P}_{R})^{\rm T}\mathbf{C}\mathbf{P}_{R}  & \mathbf{0}  \\
 \mathbf{0} & y\cdot \mathbf{I}_{d_B}
\end{pmatrix}
-\mathbf{M}{(S\cup\{i_2\},R)}
\right].
\end{aligned}
\end{equation*}
Here, the $(d+d_B)\times(d+d_B)$ matrix $\mathbf{M}{(T_1,T_2)}$ corresponding to
 $T_1\subset[d_B]$ and $T_2\subset[n_C]$ is defined as
\begin{equation*}
\mathbf{M}{(T_1,T_2)}:=\begin{pmatrix}
(\mathbf{Q}_{T_1}\mathbf{A}\mathbf{P}_{T_2})^{\rm T}   \\
(\mathbf{Q}_{T_1}\mathbf{B})^{\rm T}
\end{pmatrix}
\begin{pmatrix}
\mathbf{Q}_{T_1}\mathbf{A}\mathbf{P}_{T_2} & \mathbf{Q}_{T_1}\mathbf{B}
\end{pmatrix}=\begin{pmatrix}
\mathbf{P}_{T_2}\mathbf{A}^{\rm T}   \\
\mathbf{B}^{\rm T}
\end{pmatrix}\mathbf{Q}_{T_1}
\begin{pmatrix}
\mathbf{A}\mathbf{P}_{T_2} & \mathbf{B}
\end{pmatrix}.
\end{equation*}
Note that for each $i\in W$, by Lemma \ref{lemma2.1}, we have
\begin{equation*}
\mathbf{M}{(S\cup\{i\},R)}=\begin{pmatrix}
\mathbf{P}_{R}\mathbf{A}^{\rm T}      \\
\mathbf{B}^{\rm T}
\end{pmatrix} \mathbf{Q}_{S\cup\{i\}}  \begin{pmatrix}
\mathbf{A}\mathbf{P}_{R}  &  \mathbf{B}
\end{pmatrix}=\begin{pmatrix}
\mathbf{P}_{R}\mathbf{A}^{\rm T}      \\
\mathbf{B}^{\rm T}
\end{pmatrix}  (\mathbf{Q}_{S}-\frac{\mathbf{Q}_{S}\mathbf{b}_{i}\mathbf{b}_{i}^{\rm T}\mathbf{Q}_{S}}{\Vert\mathbf{Q}_{S}\mathbf{b}_{i} \Vert^2})  \begin{pmatrix}
\mathbf{A}\mathbf{P}_{R}  &  \mathbf{B}
\end{pmatrix}=\mathbf{M}{(S,R)}-\mathbf{c}_{i}\mathbf{c}_{i}^{\rm T},
\end{equation*}
where $\mathbf{c}_{i}:=\frac{1}{\Vert\mathbf{Q}_{S}\mathbf{b}_{i} \Vert}\begin{pmatrix}
\mathbf{b}_{i}^{\rm T}\mathbf{Q}_{S} \mathbf{A}\mathbf{P}_{R}& \mathbf{b}_{i}^{\rm T}\mathbf{Q}_{S} \mathbf{B}\end{pmatrix}^{\rm T} \in\mathbb{R}^{(d+d_B)\times 1}$.
Then we can write $Q_{\mu}(x,y,z)$ as
\begin{equation*}
\begin{aligned}
 Q_{\mu}(x,y,z)&=\mu\cdot 	\det\left[
\begin{pmatrix}
x\cdot \mathbf{I}_d +z\cdot
(\mathbf{C}\mathbf{P}_{R})^{\rm T}\mathbf{C}\mathbf{P}_{R}  &  \mathbf{0} \\
\mathbf{0} & y\cdot \mathbf{I}_{d_B}
\end{pmatrix}
-\mathbf{M}{(S,R)}+\mathbf{c}_{i_1}\mathbf{c}_{i_1}^{\rm T}
\right]\\
&\quad+(1-\mu)\cdot 	\det\left[
\begin{pmatrix}
x\cdot \mathbf{I}_d +z\cdot
(\mathbf{C}\mathbf{P}_{R})^{\rm T}\mathbf{C}\mathbf{P}_{R}  &  \mathbf{0} \\
\mathbf{0}  & y\cdot \mathbf{I}_{d_B}
\end{pmatrix}
-\mathbf{M}{(S,R)}+\mathbf{c}_{i_2}\mathbf{c}_{i_2}^{\rm T}
\right]\\
&=(1+\partial_t) \det\left[\begin{pmatrix}
x\cdot \mathbf{I}_d +z\cdot
(\mathbf{C}\mathbf{P}_{R})^{\rm T}\mathbf{C}\mathbf{P}_{R}  & \mathbf{0}  \\
\mathbf{0}  & y\cdot \mathbf{I}_{d_B}
\end{pmatrix}
-\mathbf{M}{(S,R)}+t\cdot (\mu\cdot \mathbf{c}_{i_1}\mathbf{c}_{i_1}^{\rm T}+(1-\mu)\cdot \mathbf{c}_{i_2}\mathbf{c}_{i_2}^{\rm T}) \right] \Bigg|_{t=0},
\end{aligned}
\end{equation*}
where the last equation  follows from Lemma \ref{rank-one takeout}.
By Lemmas \ref{real stable} and \ref{real stable2}, the polynomial $Q_{\mu}(x,y,z)$ is real stable in $x$, $y$, and $z$ for any $\mu\in[0,1]$. Since the differential operators $\partial_y$, $\partial_z$ and setting $y=z=0$ also preserve the real stability, according to (\ref{eq:Pmu}), the polynomial $P_{\mu}(x)$ is real stable and hence real-rooted for each $\mu\in[0,1]$. Therefore, the polynomials
\[
P_{k-l-1,r-t}(x;\mathbf{Q}_{S\cup\{i\}}\mathbf{A}\mathbf{P}_{R},\mathbf{Q}_{S\cup\{i\}}\mathbf{B},\mathbf{C} \mathbf{P}_{R}),\qquad i\in W
\]
 have a common interlacing.

We next turn to prove \eqref{lemma:interlacing-AB}.  Since $\Vert\mathbf{Q}_S\mathbf{b}_{i}\Vert\neq0$ for each $i\in W$,  the polynomials
\begin{equation*}
\Vert\mathbf{Q}_S\mathbf{b}_{i}\Vert^2\cdot P_{k-l-1,r-t}(x;\mathbf{Q}_{S\cup\{i\}}\mathbf{A}\mathbf{P}_{R},\mathbf{Q}_{S\cup\{i\}}\mathbf{B},\mathbf{C} \mathbf{P}_{R}),\qquad i\in W
\end{equation*}
also have a common interlacing. Then Lemma \ref{lemma2.6} shows that there exists an integer $j\in W$ such that
\begin{equation}\label{lemma_interlacing:eq3}
\begin{aligned}
&\mathrm{maxroot}\ \Vert\mathbf{Q}_S\mathbf{b}_{j}\Vert^2\cdot	P_{k-l-1,r-t}(x;\mathbf{Q}_{S\cup\{j\}}\mathbf{A}\mathbf{P}_{R},\mathbf{Q}_{S\cup\{j\}}\mathbf{B},\mathbf{C} \mathbf{P}_{R})\\
&\leq \mathrm{maxroot}\ \sum_{i\in W}^{}\Vert\mathbf{Q}_S\mathbf{b}_{i}\Vert^2\cdot	P_{k-l-1,r-t}(x;\mathbf{Q}_{S\cup\{i\}}\mathbf{A}\mathbf{P}_{R},\mathbf{Q}_{S\cup\{i\}}\mathbf{B},\mathbf{C} \mathbf{P}_{R}).
\end{aligned}
\end{equation}
Since $\Vert\mathbf{Q}_S\mathbf{b}_{j}\Vert\neq 0$, substituting \eqref{P-recursive2-AB} into the right hand side of \eqref{lemma_interlacing:eq3}, we arrive at our conclusion.

(ii) By a similar analysis as in (i) and using \eqref{P-exp-2} in Proposition \ref{P-three-exp}, we can obtain the desired result. The proof is not included here for the sake of conciseness.

\end{proof}

Now we can prove Theorem \ref{mth1-spr} by iterating our argument in Lemma \ref{lemma:interlacing}.

\begin{proof}[Proof of Theorem \ref{mth1-spr}]
We only consider the general case where $k>0$ and $r>0$, as the procedure remains valid when either $k$ or $r$ is zero.
We begin by initializing $\widehat{S}=\widehat{R}=\emptyset$. Applying Lemma \ref{lemma:interlacing} (i), we can identify an integer $j_1\in[d_B]$
 such that $\Vert \mathbf{b}_{j_1}\Vert\neq0$  and
\begin{equation}\label{thm1proof:eq1}
\mathrm{maxroot}\ 	P_{k-1,r}(x;\mathbf{Q}_{\{j_1\}}\mathbf{A},\mathbf{Q}_{\{j_1\}}\mathbf{B},\mathbf{C})\leq \mathrm{maxroot}\ 	P_{k,r}(x;\mathbf{A},\mathbf{B},\mathbf{C}).
\end{equation}
We then update $\widehat{S}:=\{j_1\}$.
Next, we apply Lemma \ref{lemma:interlacing} (i) to the updated set $\widehat{S}=\{j_1\}$
 and $\widehat{R}=\emptyset$.
 This yields an integer $j_2$
  such that 
 $\Vert \mathbf{Q}_{\{j_1\}}\mathbf{b}_{j_2}\Vert\neq0$ and
\begin{equation*}
\mathrm{maxroot}\ 	P_{k-2,r}(x;\mathbf{Q}_{\{j_1,j_2\}}\mathbf{A},\mathbf{Q}_{\{j_1,j_2\}}\mathbf{B},\mathbf{C})\leq \mathrm{maxroot}\ 	P_{k-1,r}(x;\mathbf{Q}_{\{j_1\}}\mathbf{A},\mathbf{Q}_{\{j_1\}}\mathbf{B},\mathbf{C}).
\end{equation*}
Combining with \eqref{thm1proof:eq1}, we obtain
\begin{equation*}
\mathrm{maxroot}\ 	P_{k-2,r}(x;\mathbf{Q}_{\{j_1,j_2\}}\mathbf{A},\mathbf{Q}_{\{j_1,j_2\}}\mathbf{B},\mathbf{C})\leq \mathrm{maxroot}\ P_{k,r}(x;\mathbf{A},\mathbf{B},\mathbf{C}).
\end{equation*}
We then update $\widehat{S}:=\{j_1,j_2\}$.
 Since $\Vert \mathbf{b}_{j_1}\Vert\neq0$ and $\Vert \mathbf{Q}_{\{j_1\}}\mathbf{b}_{j_2}\Vert\neq0$, we have $\mathrm{rank}(\mathbf{B}_{:,\{j_1,j_2\}})=2$. Repeating this argument $k-2$ more times, we obtain a $k$-subset $ \widehat{S}:=\{j_1,\ldots,j_{k}\}$ such that $\mathrm{rank}(\mathbf{B}_{:, \widehat{S}})=k$ and
\begin{equation}\label{thm1proof:eq2}
\mathrm{maxroot}\ 	P_{0,r}(x;\mathbf{Q}_{\widehat{S}}\mathbf{A},\mathbf{Q}_{\widehat{S}}\mathbf{B},\mathbf{C})\leq \mathrm{maxroot}\ P_{k,r}(x;\mathbf{A},\mathbf{B},\mathbf{C}).
\end{equation}
Then, applying Lemma \ref{lemma:interlacing} (ii) to the sets $ \widehat{S}=\{j_1,\ldots,j_{k}\}$ and $\widehat{R}=\emptyset$, there exists an integer $j_1'\in[n_C]$, such that $\Vert \mathbf{c}_{j_1'}\Vert\neq0$ and
\begin{equation*}
\mathrm{maxroot}\ 	P_{0,r-1}(x;\mathbf{Q}_{\widehat{S}}\mathbf{A}\mathbf{P}_{\{j_1'\}},\mathbf{Q}_{\widehat{S}}\mathbf{B},\mathbf{C}\mathbf{P}_{\{j_1'\}})\leq \mathrm{maxroot}\ 	P_{0,r}(x;\mathbf{Q}_{\widehat{S}}\mathbf{A},\mathbf{Q}_{\widehat{S}}\mathbf{B},\mathbf{C}).
\end{equation*}
Combining with \eqref{thm1proof:eq2}, we obtain
\begin{equation*}
\mathrm{maxroot}\ 	P_{0,r-1}(x;\mathbf{Q}_{\widehat{S}}\mathbf{A}\mathbf{P}_{\{j_1'\}},\mathbf{Q}_{\widehat{S}}\mathbf{B},\mathbf{C}\mathbf{P}_{\{j_1'\}})\leq \mathrm{maxroot}\ P_{k,r}(x;\mathbf{A},\mathbf{B},\mathbf{C}).
\end{equation*}
Then, we set $ \widehat{R}:=\{j_1'\}$. Similarly, repeating this argument for $r-1$ more times, we obtain an $r$-subset $ \widehat{R}:=\{j_1',\ldots,j_{r}'\}$ such that $\mathrm{rank}(\mathbf{C}_{\widehat{R},:})=r$ and
\begin{equation}\label{thm1proof:eq22}
\mathrm{maxroot}\ 	P_{0,0}(x;\mathbf{Q}_{\widehat{S}}\mathbf{A}\mathbf{P}_{\widehat{R}},\mathbf{Q}_{\widehat{S}}\mathbf{B},\mathbf{C}\mathbf{P}_{\widehat{R}})\leq \mathrm{maxroot}\ P_{k,r}(x;\mathbf{A},\mathbf{B},\mathbf{C}).
\end{equation}
According to Definition \ref{def2}, we have
\[
P_{0,0}(x;\mathbf{Q}_{\widehat{S}}\mathbf{A}\mathbf{P}_{\widehat{R}},\mathbf{Q}_{\widehat{S}}\mathbf{B},\mathbf{C}\mathbf{P}_{\widehat{R}})=\det[x\cdot \mathbf{I}_d-\mathbf{P}_{\widehat{R}}\mathbf{A}^{\rm T}\mathbf{Q}_{\widehat{S}}\mathbf{A}\mathbf{P}_{\widehat{R}}],
 \]
 which implies
\begin{equation}\label{P00}
\mathrm{maxroot}\ 	P_{0,0}(x;\mathbf{Q}_{\widehat{S}}\mathbf{A}\mathbf{P}_{\widehat{R}},\mathbf{Q}_{\widehat{S}}\mathbf{B},\mathbf{C}\mathbf{P}_{\widehat{R}})=\Vert \mathbf{Q}_{\widehat{S}}\mathbf{A}\mathbf{P}_{\widehat{R}}\Vert_2^2.	
\end{equation}
Recall that
\begin{equation}\label{QSPR1}
\mathbf{Q}_{\widehat{S}}:=\mathbf{I}_n-\mathbf{B}_{:,\widehat{S}}\mathbf{B}_{:,\widehat{S}}^{\dagger}\quad \text{and}\quad \mathbf{P}_{\widehat{R}}:=\mathbf{I}_d-\mathbf{C}_{\widehat{R},:}^{\dagger}\mathbf{C}_{\widehat{R},:}.
\end{equation}
Combining (\ref{P00}), (\ref{QSPR1}), and \eqref{thm1proof:eq22}, we obtain the desired result.

\end{proof}

\subsection{Deterministic Polynomial Time Algorithm for GCRSS}\label{section: algorithm for GCRSS}

In this subsection, we propose our deterministic polynomial time algorithm for the spectral norm case of GCRSS; see Algorithm \ref{alg-spr-GCRSS}. 
Before proceeding, we introduce a pertinent notation. 
For a real-rooted polynomial $p(x)$, we use $\lambda_{\max}^{\eta}(p)\in \mathbb{R}$ to represent an $\eta$-approximation to the largest root of $p(x)$, i.e., 
\begin{equation}\label{eps-approx}
| \lambda_{\max}^{\eta}(p)-\text{\rm maxroot}\ p|\leq \eta	.
\end{equation}


Algorithm \ref{alg-spr-GCRSS} is grounded in our proof of Theorem \ref{mth1-spr}. The algorithm proceeds as follows:
During the $l$-th iteration, we compute an $\eta$-approximation of the largest root of $P_{k-l,r}(x;\mathbf{Q}_{\widehat{S}\cup\{i\}}\mathbf{A},\mathbf{Q}_{\widehat{S}\cup\{i\}}\mathbf{B},\mathbf{C})$
for each $i\in [d_B]\backslash \widehat{S}$.
Subsequently, we add the index $i\in [d_B]\backslash \widehat{S}$ that satisfies $\|\mathbf{Q}_{\widehat{S}}\mathbf{b}_i\|\neq 0$ and minimizes $\lambda_{\max}^{\eta}( P_{k-l,r}(x;\mathbf{Q}_{\widehat{S}\cup\{i\}}\mathbf{A},\mathbf{Q}_{\widehat{S}\cup\{i\}}\mathbf{B},\mathbf{C}))$
   to the subset $\widehat{S}$. Once the $k$-subset $\widehat{S}\subset[d_B]$  is constructed, the algorithm employs a similar approach to build the $r$-subset $\widehat{R}\subset[n_C]$.
Lemma \ref{lemma:alg-GCRSS} presents the time complexity for computing the expected polynomials in Lines 4 and 11 of Algorithm \ref{alg-spr-GCRSS}. 

\begin{lemma}\label{lemma:alg-GCRSS}
Let  $\mathbf{A}=[\mathbf{a}_1,\ldots,\mathbf{a}_d]\in\mathbb{R}^{n\times d}$, $\mathbf{B}=[\mathbf{b}_1,\ldots,\mathbf{b}_{d_B}]\in\mathbb{R}^{n\times d_B}$ and $\mathbf{C}=[\mathbf{c}_1,\ldots,\mathbf{c}_{n_C}]^{\rm T}\in\mathbb{R}^{n_C\times d}$. 
Let $k$ and $r$ be two integers satisfying $0\leq k\leq\mathrm{rank}(\mathbf{B})$ and $0\leq r\leq\mathrm{rank}(\mathbf{C})$.
 Let $\widehat{S}\subset[d_B]$ be a subset of size $|\widehat{S}|\leq k$, and let $\widehat{R}\subset[n_C]$ be a subset of size $|\widehat{R}|\leq r$. 
 Denote $\mathbf{V}=\mathbf{B}\mathbf{B}^{\rm T}\mathbf{Q}_{\widehat{S}}\in\mathbb{R}^{n\times n}$ and $\mathbf{W}=\begin{pmatrix}
\mathbf{A}   \\
\mathbf{C}
\end{pmatrix}\mathbf{P}_{\widehat{R}}[\mathbf{A}^{\rm T}\mathbf{Q}_{\widehat{S}}\ \ \mathbf{C}^{\rm T}]\in\mathbb{R}^{(n+n_C)\times (n+n_C)}$. 
Assume that the matrices $\mathbf{V}$ and $\mathbf{W}$ are known. 
Then we can compute the polynomial $P_{k-|\widehat{S}|,r-|\widehat{R}|}(x;\mathbf{Q}_{\widehat{S}}\mathbf{A}\mathbf{P}_{\widehat{R}},\mathbf{Q}_{\widehat{S}}\mathbf{B},\mathbf{C}\mathbf{P}_{\widehat{R}})$ in time $O((n_C+n)^{w+2})$, where $w \in  (2, 2.373)$ is the matrix multiplication exponent.

\end{lemma}

\begin{proof}
See Appendix \ref{Appendix-alg-GCRSS}.
\end{proof}

\begin{algorithm}[!t]
\caption{Deterministic polynomial time algorithm for the spectral norm case of GCRSS.}\label{alg-spr-GCRSS}
\begin{algorithmic}[1]
\Require $\mathbf{A}=[\mathbf{a}_1,\ldots,\mathbf{a}_d]\in\mathbb{R}^{n\times d}$; $\mathbf{B}=[\mathbf{b}_1,\ldots,\mathbf{b}_{d_B}]\in\mathbb{R}^{n\times d_B}$; $\mathbf{C}=[\mathbf{c}_1,\ldots,\mathbf{c}_{n_C}]^{\rm T}\in\mathbb{R}^{n_C\times d}$;  sampling parameters $k$ and $r$; error parameter $\eta>0$.
\Ensure A subset $\widehat{S}\subset[d_B]$ of cardinality $k$ and a subset $\widehat{R}\subset[n_C]$ of cardinality $r$.
\State Set $\widehat{S}=\widehat{R}=\emptyset$, $\widehat{\mathbf{Q}}=\mathbf{I}_n$ and $\widehat{\mathbf{P}}=\mathbf{I}_d$. 
Compute $\widehat{\mathbf{V}}=\mathbf{B}\mathbf{B}^{\rm T}\widehat{\mathbf{Q}}$ and $\widehat{\mathbf{W}}=\begin{pmatrix}
\mathbf{A}   \\
\mathbf{C}
\end{pmatrix}\widehat{\mathbf{P}} [\mathbf{A}^{\rm T}\widehat{\mathbf{Q}}\ \ \mathbf{C}^{\rm T}]$.

\For{$l=1,\ldots,k$}

\State For each $i\in[d_B]\backslash \widehat{S}$, compute 
$\mathbf{q}_i=\frac{\widehat{\mathbf{Q}}\mathbf{b}_i}{\|\widehat{\mathbf{Q}}\mathbf{b}_i\|}$, 
$\mathbf{Q}_{\widehat{S}\cup\{i\}}= \widehat{\mathbf{Q}}-\mathbf{q}_i\cdot \mathbf{q}_i^{\rm T}$,
$\mathbf{B}\mathbf{B}^{\rm T}\mathbf{Q}_{\widehat{S}\cup \{i\}}=\widehat{\mathbf{V}}-\widehat{\mathbf{V}}\mathbf{q}_i\cdot \mathbf{q}_i^{\rm T}$ 
and $\begin{pmatrix}
\mathbf{A}   \\
\mathbf{C}
\end{pmatrix}\mathbf{P}_{\widehat{R}} [\mathbf{A}^{\rm T}\mathbf{Q}_{\widehat{S}\cup\{i\}}\ \ \mathbf{C}^{\rm T}]
=\widehat{\mathbf{W}}-\widehat{\mathbf{W}} \begin{pmatrix}
\mathbf{q}_i\\
\mathbf{0}	\end{pmatrix}\cdot \begin{pmatrix}
\mathbf{q}_i^{\rm T} & \mathbf{0}	
\end{pmatrix}$.

\State For each $i\in[d_B]\backslash \widehat{S}$, compute the polynomial $P_{k-l,r}(x;\mathbf{Q}_{\widehat{S}\cup\{i\}}\mathbf{A},\mathbf{Q}_{\widehat{S}\cup\{i\}}\mathbf{B},\mathbf{C})$ and find an $\eta$-approximation $\lambda_{\max}^{\eta}(P_{k-l,r}(x;\mathbf{Q}_{\widehat{S}\cup\{i\}}\mathbf{A},\mathbf{Q}_{\widehat{S}\cup\{i\}}\mathbf{B},\mathbf{C}))$ using the standard technique of binary search with a Sturm sequence.
\State Find $j_{l}=\mathop{\mathrm{argmin}}_{i: \|\widehat{\mathbf{Q}}\mathbf{b}_i\|\neq 0} 	\lambda_{\max}^{\eta}\big( P_{k-l,r}(x;\mathbf{Q}_{\widehat{S}\cup\{i\}}\mathbf{A},\mathbf{Q}_{\widehat{S}\cup\{i\}}\mathbf{B},\mathbf{C})\big)$.

%
%

\State Update 
$\widehat{\mathbf{V}}\gets\mathbf{B}\mathbf{B}^{\rm T}\mathbf{Q}_{\widehat{S}\cup \{j_l\}}$,
$\widehat{\mathbf{W}}\gets \begin{pmatrix}
\mathbf{A}   \\
\mathbf{C}
\end{pmatrix}\mathbf{P}_{\widehat{R}} [\mathbf{A}^{\rm T}\mathbf{Q}_{\widehat{S}\cup\{j_l\}}\  \mathbf{C}^{\rm T}]$,
$\widehat{\mathbf{Q}}\gets \mathbf{Q}_{\widehat{S}\cup\{j_l\}}$ 
and $\widehat{S}\gets \widehat{S}\cup \{j_{l}\}$.
\EndFor 
\State 
Compute $\widehat{\mathbf{H}}_1=\begin{pmatrix}
\mathbf{A}   \\
\mathbf{C}
\end{pmatrix}\widehat{\mathbf{P}} \in\mathbb{R}^{(n+n_C)\times d}$ 
and $\widehat{\mathbf{H}}_2=\widehat{\mathbf{P}}[\mathbf{A}^{\rm T}\mathbf{Q}_{\widehat{S}}\ \ \mathbf{C}^{\rm T}]\in\mathbb{R}^{d\times (n+n_C)}$.

\For{$l=1,\ldots,r$}

\State For each $i\in[n_C]\backslash \widehat{R}$, compute
$\mathbf{p}_i=\frac{\widehat{\mathbf{P}}\mathbf{c}_i}{\|\widehat{\mathbf{P}}\mathbf{c}_i\|}$ and  
$\begin{pmatrix}
\mathbf{A}   \\
\mathbf{C}
\end{pmatrix}\mathbf{P}_{\widehat{R}\cup\{i\}} [\mathbf{A}^{\rm T}\mathbf{Q}_{\widehat{S}}\ \ \mathbf{C}^{\rm T}]
=\widehat{\mathbf{W}}-\widehat{\mathbf{H}}_1\mathbf{p}_i\cdot \mathbf{p}_i^{\rm T}\widehat{\mathbf{H}}_2$.  

\State For each $i\in[n_C]\backslash \widehat{R}$, compute $P_{0,r-l}(x;\mathbf{Q}_{\widehat{S}}\mathbf{A}\mathbf{P}_{\widehat{R}\cup\{i\}},\mathbf{Q}_{\widehat{S}}\mathbf{B},\mathbf{C}\mathbf{P}_{\widehat{R}\cup\{i\}})$ and find an $\eta$-approximation $\lambda_{\max}^{\eta}(P_{0,r-l}(x;\mathbf{Q}_{\widehat{S}}\mathbf{A}\mathbf{P}_{\widehat{R}\cup\{i\}},\mathbf{Q}_{\widehat{S}}\mathbf{B},\mathbf{C}\mathbf{P}_{\widehat{R}\cup\{i\}}))$ for each $i\in[n_C]\backslash \widehat{R}$, using the standard technique of binary search with a Sturm sequence.
\State Find $j'_{l}=\mathop{\mathrm{argmin}}_{i: \|\mathbf{c}_i^{\rm T}\widehat{\mathbf{P}}\|\neq 0} 	\lambda_{\max}^{\eta}\big(P_{0,r-l}(x;\mathbf{Q}_{\widehat{S}}\mathbf{A}\mathbf{P}_{\widehat{R}\cup\{i\}},\mathbf{Q}_{\widehat{S}}\mathbf{B},\mathbf{C}\mathbf{P}_{\widehat{R}\cup\{i\}}) \big)$.

\State Update 
$\widehat{\mathbf{W}}\gets\begin{pmatrix}
\mathbf{A}   \\
\mathbf{C}
\end{pmatrix}\mathbf{P}_{\widehat{R}\cup\{j_l'\}} [\mathbf{A}^{\rm T}\mathbf{Q}_{\widehat{S}}\ \ \mathbf{C}^{\rm T}]$, 
$\widehat{\mathbf{H}}_1\gets \widehat{\mathbf{H}}_1-\widehat{\mathbf{H}}_1 \mathbf{p}_{j_{l}'}\cdot \mathbf{p}_{j_{l}'}^{\rm T}$,
$\widehat{\mathbf{H}}_2\gets \widehat{\mathbf{H}}_2-\mathbf{p}_{j_{l}'}\cdot \mathbf{p}_{j_{l}'}^{\rm T}\widehat{\mathbf{H}}_2 $,
$\widehat{\mathbf{P}}\gets \widehat{\mathbf{P}}-\widehat{\mathbf{P}}\mathbf{p}_{j_{l}'}\cdot \mathbf{p}_{j_{l}'}^{\rm T} $ 
and $\widehat{R}\gets \widehat{R}\cup \{j_{l}'\}$. 
\EndFor 
\\
\Return $\widehat{S}=\{j_1,j_2,\ldots,j_{k}\}\subset[d_B]$ and $\widehat{R}=\{j'_1,j'_2,\ldots,j'_{r}\}\subset[n_C]$.
\end{algorithmic}
\end{algorithm}

Now we analyze the time complexity of Algorithm \ref{alg-spr-GCRSS} and present a proof of Theorem \ref{mth:alg:spr}.

\begin{proof}[Proof of Theorem \ref{mth:alg:spr}]
We first prove \eqref{alg:eq1-GCRSS}.
At the $l$-th iteration for constructing the subset $\widehat{S}$, we let
\begin{equation*}
i_{l}:=\mathop{\mathrm{argmin}}_{i: \|{\mathbf{Q}_{\widehat{S}}}\mathbf{b}_i\|\neq 0} 	\text{\rm maxroot}\  P_{k-l,r}(x;\mathbf{Q}_{\widehat{S}\cup\{i\}}\mathbf{A},\mathbf{Q}_{\widehat{S}\cup\{i\}}\mathbf{B},\mathbf{C}).
\end{equation*}
By our choice of $j_l$ in Line 5 of Algorithm \ref{alg-spr-GCRSS}, we have
\begin{equation*}
\begin{aligned}
&\text{\rm maxroot}\  P_{k-l,r}(x;\mathbf{Q}_{\widehat{S}\cup\{j_{l}\}}\mathbf{A},\mathbf{Q}_{\widehat{S}\cup\{j_{l}\}}\mathbf{B},\mathbf{C})\\
&\overset{(a)}  \leq \lambda_{\max}^{\eta}\Big( P_{k-l,r}(x;\mathbf{Q}_{\widehat{S}\cup\{j_{l}\}}\mathbf{A},\mathbf{Q}_{\widehat{S}\cup\{j_{l}\}}\mathbf{B},\mathbf{C})\Big)+\eta 
\leq \lambda_{\max}^{\eta}\Big(P_{k-l,r}(x;\mathbf{Q}_{\widehat{S}\cup\{i_{l}\}}\mathbf{A},\mathbf{Q}_{\widehat{S}\cup\{i_{l}\}}\mathbf{B},\mathbf{C})\Big)+\eta
\\&
 \overset{(b)}\leq  \text{\rm maxroot}\  P_{k-l,r}(x;\mathbf{Q}_{\widehat{S}\cup\{i_{l}\}}\mathbf{A},\mathbf{Q}_{\widehat{S}\cup\{i_{l}\}}\mathbf{B},\mathbf{C})+ 2\eta 
 \overset{(c)}\leq  \text{\rm maxroot}\ P_{k-l+1,r}(x;\mathbf{Q}_{\widehat{S}}\mathbf{A},\mathbf{Q}_{\widehat{S}}\mathbf{B},\mathbf{C})  +2\eta.
\end{aligned}
\end{equation*}
Here, ($a$), ($b$) follow from the definition of an $\eta$-approximation in \eqref{eps-approx},  and (c) follows from Lemma \ref{lemma:interlacing} (i).
Starting with $\widehat{S}=\emptyset$ and applying the above argument to $l=1,2,\ldots,k$, 
we can iteratively construct the subset $\widehat{S}=\{j_1,\ldots,j_k\}\subset[d_B]$ such that
\begin{equation*}
\mathrm{maxroot}\ 	P_{0,r}(x;\mathbf{Q}_{\widehat{S}}\mathbf{A},\mathbf{Q}_{\widehat{S}}\mathbf{B},\mathbf{C})
\leq 2k\eta+\mathrm{maxroot}\ P_{k,r}(x;\mathbf{A},\mathbf{B},\mathbf{C}).
\end{equation*}
Similarly, applying this argument for constructing the subset $\widehat{R}$, we can iteratively construct the subset $\widehat{R}=\{j_1',\ldots,j_r'\}\subset[n_C]$ such that
\begin{equation*}
\begin{aligned}
\mathrm{maxroot}\ 	P_{0,0}(x;\mathbf{Q}_{\widehat{S}}\mathbf{A}\mathbf{P}_{\widehat{R}},\mathbf{Q}_{\widehat{S}}\mathbf{B},\mathbf{C}\mathbf{P}_{\widehat{R}})
&\leq 2r\eta+ \mathrm{maxroot}\ 	P_{0,r}(x;\mathbf{Q}_{\widehat{S}}\mathbf{A},\mathbf{Q}_{\widehat{S}}\mathbf{B},\mathbf{C}) \\
&\leq 2(k+r)\eta+\mathrm{maxroot}\ P_{k,r}(x;\mathbf{A},\mathbf{B},\mathbf{C}).	
\end{aligned}
\end{equation*}
Recall that the largest root of $P_{0,0}(x;\mathbf{Q}_{\widehat{S}}\mathbf{A}\mathbf{P}_{\widehat{R}},\mathbf{Q}_{\widehat{S}}\mathbf{B},\mathbf{C}\mathbf{P}_{\widehat{R}})$ is $	\Vert(\mathbf{I}_n-\mathbf{B}_{:, \widehat{S}}\mathbf{B}_{:, \widehat{S}}^{\dagger})\mathbf{A}(\mathbf{I}_d-\mathbf{C}_{\widehat{R}, :}^{\dagger}\mathbf{C}_{\widehat{R}, :}) \Vert_{2}^2$, so we arrive at \eqref{alg:eq1-GCRSS}.

We now turn to analyze the time complexity of the spectral norm case of Algorithm \ref{alg-spr-GCRSS}. 
We first show that the time complexity for constructing the column subset $\widehat{S}$ at each iteration is $O(d_B\cdot (n+n_C)^{w+2}+d_B\cdot d^2\log\frac{1}{\eta})$. 
Specifically, for each $i\in[d_B]\backslash \widehat{S}$, we first 
compute 
$\mathbf{q}_i\in\mathbb{R}^n$, 
$\mathbf{Q}_{\widehat{S}\cup\{i\}}\in\mathbb{R}^{n\times n}$,
$\mathbf{B}\mathbf{B}^{\rm T}\mathbf{Q}_{\widehat{S}\cup \{i\}}\in\mathbb{R}^{n\times n}$ 
and $\begin{pmatrix}
\mathbf{A}   \\
\mathbf{C}
\end{pmatrix}\mathbf{P}_{\widehat{R}} [\mathbf{A}^{\rm T}\mathbf{Q}_{\widehat{S}\cup\{i\}}\ \ \mathbf{C}^{\rm T}]
\in\mathbb{R}^{(n+n_C)\times (n+n_C)}$
in time $O(n^2)$, $O(n^2)$, $O(n^2)$ and $O(n(n+n_C))$, respectively. 
Lemma \ref{lemma:alg-GCRSS} shows that we can compute each polynomial $P_{k-l,r}(x;\mathbf{Q}_{\widehat{S}\cup\{i\}}\mathbf{A},\mathbf{Q}_{\widehat{S}\cup\{i\}}\mathbf{B},\mathbf{C})$ in time $O((n+n_C)^{w+2})$.
The time complexity to obtain an $\eta$-approximation to the largest root of a degree $d$ polynomial is $O(d^2\log \frac{1}{\eta})$ \cite{inter3}.
Additionally, finding $j_l$ in  Line 5 takes $O(d_B)$ time.
Thus, the time complexity of constructing the subset $\widehat{S}$ at each iteration is 
\begin{equation}\label{construct-S}
d_B\cdot O\Big( 3n^2+n(n+n_C)+(n+n_C)^{w+2}+d^2\log\frac{1}{\eta}\Big)+	O(d_B)=O\Big(d_B\cdot (n+n_C)^{w+2}+d_B\cdot d^2\log\frac{1}{\eta}\Big).
\end{equation} 

We next analyze the time complexity of constructing the subset $\widehat{R}$ at each iteration. 
For each $i\in[n_C]\backslash \widehat{R}$, we compute 
$\mathbf{p}_i\in\mathbb{R}^{d}$ 
and $\begin{pmatrix}
\mathbf{A}   \\
\mathbf{C}
\end{pmatrix}\mathbf{P}_{\widehat{R}\cup\{i\}} [\mathbf{A}^{\rm T}\mathbf{Q}_{\widehat{S}}\ \ \mathbf{C}^{\rm T}]$ in time $O(d^2)$ and $O(d(n+n_C)+(n+n_C)^2)$, respectively. 
Then, we can compute each polynomial $P_{k-l,r}(x;\mathbf{Q}_{\widehat{S}\cup\{i\}}\mathbf{A},\mathbf{Q}_{\widehat{S}\cup\{i\}}\mathbf{B},\mathbf{C})$ and its $\eta$-approximation in time $O((n+n_C)^{w+2}+d^2\log \frac{1}{\eta})$.
Finding $j_l'$ in Line 20 takes $O(n_C)$ time. 
Updating the matrices $\widehat{\mathbf{W}}$, $\widehat{\mathbf{H}}_1$, $\widehat{\mathbf{H}}_2$ and $\widehat{\mathbf{P}}$ in Line 25 takes time $O(d(n+n_C)+d^2)$.
Therefore, the time complexity of constructing the subset $\widehat{R}$ at each iteration is 
\begin{equation}\label{construct-R}
\begin{aligned}
&n_C\cdot O\Big( d^2+ d(n+n_C)+(n+n_C)^2+(n+n_C)^{w+2}+d^2\log\frac{1}{\eta}\Big)+	O(n_C)+O(d(n+n_C)+d^2)\\
&=O\Big(n_C\cdot d(n+n_C)+n_C\cdot (n+n_C)^{w+2}+n_C\cdot d^2\log\frac{1}{\eta}\Big).
\end{aligned}
\end{equation} 

Now we can analyze the total running time of the algorithm. 
In Line 1, we compute the matrices $\widehat{\mathbf{W}}\in\mathbb{R}^{(n+n_C)\times (n+n_C)}$ and $\widehat{\mathbf{V}}\in\mathbb{R}^{n\times n}$ in time $O(d(n+n_C)^2)$ and $O(d_B\cdot n^2)$, respectively.  
In Line 14, we need $O(dn^2)$ time to compute $\mathbf{A}^{\rm T}\mathbf{Q}_{\widehat{S}}\in\mathbb{R}^{d\times n}$. Then we form the matrices $\widehat{\mathbf{H}}_1$ and $\widehat{\mathbf{H}}_2$ in time $O(1)$. 
Finally, combining with \eqref{construct-S} and \eqref{construct-R}, the total time complexity of the algorithm is
\begingroup\fontsize{10pt}{12pt}\selectfont 
\begin{equation*}
\begin{aligned}
&r\cdot O\Big(n_C\cdot d(n+n_C)+n_C\cdot (n+n_C)^{w+2}+n_C\cdot d^2\log\frac{1}{\eta}\Big)+
k\cdot O\Big(d_B\cdot (n+n_C)^{w+2}+d_B\cdot d^2\log\frac{1}{\eta}\Big)\\
&+O(d(n+n_C)^2+d_B\cdot n^2+dn^2)=O\Big(dn^2+r\cdot n_C\cdot  d(n+n_C)+(k\cdot d_B+r\cdot n_C)\cdot \big((n+n_C)^{w+2}+d^2\log\frac{1}{\eta}\big)\Big).
\end{aligned}	
\end{equation*} 
\endgroup
This completes the proof.

\end{proof}

\begin{remark}\label{re:41}
Algorithm \ref{alg-spr-GCRSS} can be efficiently applied to both the generalized column subset selection (GCSS) and submatrix selection problems with reduced computational complexity.
Specifically, for the GCSS problem where $r=0$ and $\mathbf{C}=\mathbf{0}$,   we can use equation \eqref{lemma_f_k:expression2} to compute $P_{k-|\widehat{S}|}(x;\mathbf{Q}_{\widehat{S}}\mathbf{A},\mathbf{Q}_{\widehat{S}}\mathbf{B})$ at each iteration.
Since equation \eqref{lemma_f_k:expression2} involves only two variables, this computation requires only a one-dimensional Fast Fourier Transform for polynomial interpolation. 
Consequently, the time complexity for calculating each $P_{k-|\widehat{S}|}(x;\mathbf{Q}_{\widehat{S}}\mathbf{A},\mathbf{Q}_{\widehat{S}}\mathbf{B})$ is reduced to $O(n^{w+1})$ when the matrices $\mathbf{A}\mathbf{A}^{\rm T}\mathbf{Q}_{\widehat{S}}$ and $\mathbf{B}\mathbf{B}^{\rm T}\mathbf{Q}_{\widehat{S}}$ are known. 
For the submatrix selection problem where $\mathbf{B}=\mathbf{C}=\mathbf{I}_d$, we can use equation \eqref{P-exp-II-oct} to compute  $P_{k-|\widehat{S}|,r-|\widehat{R}|}(x;\mathbf{Q}_{\widehat{S}}\mathbf{A}\mathbf{P}_{\widehat{R}},\mathbf{Q}_{\widehat{S}}\mathbf{B},\mathbf{C}\mathbf{P}_{\widehat{R}})$, which does not require polynomial interpolation. Thus, the time complexity for computing each $P_{k-|\widehat{S}|,r-|\widehat{R}|}(x;\mathbf{Q}_{\widehat{S}}\mathbf{A}\mathbf{P}_{\widehat{R}},\mathbf{Q}_{\widehat{S}}\mathbf{B},\mathbf{C}\mathbf{P}_{\widehat{R}})$ is reduced to $O(d^{w})$ when the matrix $(\mathbf{A}_{S^C,R^C})^{\rm T}\mathbf{A}_{S^C,R^C}$ or $\mathbf{A}_{S^C,R^C}(\mathbf{A}_{S^C,R^C})^{\rm T}$ is known.
For brevity, we omit the proof of runtime analysis for these two cases. 
\end{remark}

\section{Reconstruction Bound for the Spectral Norm Case of GCSS: Proof of Theorem \ref{main thm-spr}}

The aim of this subsection is to prove Theorem \ref{main thm-spr}. 
We first prove the following lemma, which provides an upper bound on the spectral norm of a residual matrix.

\begin{lemma}\label{mth2-spr}
Let  $\mathbf{A}\in\mathbb{R}^{n\times d}$ and $\mathbf{B}\in\mathbb{R}^{n\times d_B}$. For each positive integer $1\leq k\leq \mathrm{rank}(\mathbf{B})$, we can iteratively select a subset $ \widehat{S}=\{j_1,\ldots,j_k\}\subset[d_B]$ of size $k$, such that $\mathrm{rank}(\mathbf{B}_{:, \widehat{S}})=k$ and
\begin{equation}\label{spr:bound2}
\Vert\mathbf{A}-\mathbf{B}_{:, \widehat{S}}\mathbf{B}_{:, \widehat{S}}^{\dagger}\mathbf{A} \Vert_{2}^2\leq \Vert\mathbf{A}-\mathbf{B}_{}\mathbf{B}_{}^{\dagger}\mathbf{A} \Vert_{2}^2+  \mathrm{maxroot}\ P_{k}(x;\mathbf{B}_{}\mathbf{B}_{}^{\dagger}\mathbf{A},\mathbf{B}).
\end{equation}	
\end{lemma}

\begin{proof}
Note that for each subset $S\subset[d_B]$, we have $\mathbf{A}-\mathbf{B}_{:,{S}}\mathbf{B}_{:,{S}}^{\dagger}\mathbf{A}=(\mathbf{A}-\mathbf{B}_{ }\mathbf{B}_{ }^{\dagger}\mathbf{A})+(\mathbf{B}_{ }\mathbf{B}_{ }^{\dagger}\mathbf{A}-\mathbf{B}_{:,{S}}\mathbf{B}_{:,{S}}^{\dagger}\mathbf{A})$, so we obtain
\begin{equation}\label{Frob:eq21}
\begin{aligned}
\Vert \mathbf{A}-\mathbf{B}_{:,{S}}\mathbf{B}_{:,{S}}^{\dagger}\mathbf{A}\Vert_{2}^2&= \Vert (\mathbf{A}-\mathbf{B}_{:,{S}}\mathbf{B}_{:,{S}}^{\dagger}\mathbf{A})^{\rm T} (\mathbf{A}-\mathbf{B}_{:,{S}}\mathbf{B}_{:,{S}}^{\dagger}\mathbf{A}) \Vert_2\\
&\overset{(a)}=\Vert(\mathbf{A}-\mathbf{B}_{ }\mathbf{B}_{ }^{\dagger}\mathbf{A})^{\rm T}(\mathbf{A}-\mathbf{B}_{ }\mathbf{B}_{ }^{\dagger}\mathbf{A})+(\mathbf{B}_{ }\mathbf{B}_{ }^{\dagger}\mathbf{A}-\mathbf{B}_{:,{S}}\mathbf{B}_{:,{S}}^{\dagger}\mathbf{A})^{\rm T}(\mathbf{B}_{ }\mathbf{B}_{ }^{\dagger}\mathbf{A}-\mathbf{B}_{:,{S}}\mathbf{B}_{:,{S}}^{\dagger}\mathbf{A}) \Vert_2\\
&\overset{(b)}\leq\Vert \mathbf{A}-\mathbf{B}_{}\mathbf{B}_{}^{\dagger}\mathbf{A}\Vert_{2}^2+\Vert (\mathbf{I}_n-\mathbf{B}_{:,{S}}\mathbf{B}_{:,{S}}^{\dagger})\mathbf{B}_{ }\mathbf{B}_{ }^{\dagger}\mathbf{A}\Vert_{2}^2.
\end{aligned}
\end{equation}
Here in ($a$) and ($b$), we used the facts that $(\mathbf{I}_n-\mathbf{B}_{ }\mathbf{B}_{ }^{\dagger})\mathbf{B}_{:,S }\mathbf{B}_{:,S }^{\dagger}=\mathbf{0}_{n\times n}$ and $ \mathbf{B}_{:,S }\mathbf{B}_{:,S }^{\dagger}=\mathbf{B}_{:,S }\mathbf{B}_{:,S }^{\dagger}\mathbf{B}_{ }\mathbf{B}_{ }^{\dagger}$.
Taking $r=0$ in Theorem \ref{mth1-spr} and applying the result of Theorem \ref{mth1-spr} to the target matrix $\mathbf{B}_{ }\mathbf{B}_{ }^{\dagger}\mathbf{A}\in\mathbb{R}^{n\times d}$ and the source matrix $\mathbf{B}\in\mathbb{R}^{n\times d_B}$, we obtain that there exists a $k$-subset $\widehat{S}$ such that $\mathrm{rank}(\mathbf{B}_{:, \widehat{S}})=k$ and
\begin{equation}\label{Frob:eq22}
\Vert (\mathbf{I}_n-\mathbf{B}_{:,\widehat{S}}\mathbf{B}_{:,\widehat{S}}^{\dagger})\mathbf{B}_{ }\mathbf{B}_{ }^{\dagger}\mathbf{A}\Vert_{2}^2\leq \mathrm{maxroot}\ 	P_{k}(x;\mathbf{B}\mathbf{B}^{\dagger}\mathbf{A},\mathbf{B}).
\end{equation}
Combining \eqref{Frob:eq21} with \eqref{Frob:eq22}, we arrive at \eqref{spr:bound2}.

\end{proof}

By Lemma \ref{mth2-spr}, to obtain an upper bound on the residual $\Vert\mathbf{A}-\mathbf{B}_{:, \widehat{S}}\mathbf{B}_{:, \widehat{S}}^{\dagger}\mathbf{A} \Vert_{2}^2$, it suffices to estimate the largest root of the expected polynomial $P_{k}(x;\mathbf{B}\mathbf{B}^{\dagger}\mathbf{A},\mathbf{B})$.

\subsection{Estimation of the Largest Root of $P_{k}(x;\mathbf{B}\mathbf{B}^{\dagger}\mathbf{A},\mathbf{B})$}

The following lemma is the main result of this subsection, which provides an estimate on the largest root of $P_{k}(x;\mathbf{B}\mathbf{B}^{\dagger}\mathbf{A},\mathbf{B})$. We postpone the proof of Lemma \ref{mth2} to the end of this subsection.

\begin{lemma}\label{mth2}
Let $\mathbf{A}$ be a matrix in $\mathbb{R}^{n\times d}$, and let $\mathbf{B}\in\mathbb{R}^{n\times d_B}$ be a matrix of rank $m$. Let $b_i$ be the square of the $i$-th largest singular value of $\mathbf{B}$, and let $\mathbf{B}\mathbf{B}^{\rm T}=\mathbf{U}\cdot \mathrm{diag}(b_1,\ldots,b_m)\cdot \mathbf{U}^{\rm T}$ be the singular value decomposition of $\mathbf{B}\mathbf{B}^{\rm T}$, where $\mathbf{U}\in\mathbb{R}^{n\times m}$ satisfies $\mathbf{U}^{\rm T}\mathbf{U}=\mathbf{I}_m$. Let $\alpha$ be the largest diagonal element of the matrix $\frac{\mathbf{U}^{\rm T}\mathbf{A}\mathbf{A}^{\rm T}\mathbf{U}}{\Vert \mathbf{U}^{\rm T}\mathbf{A}\mathbf{A}^{\rm T}\mathbf{U}\Vert_2}$. Let $k\in[m-1]$ be an integer such that
\begin{equation*}
\delta_{k}:= \frac{\sum_{S\subset[m-1],\vert S\vert=k} \mathbf{b}^{S}}{	 \sum_{S\subset[m],\vert S\vert=k} \mathbf{b}^{S}}\leq(1-\sqrt{\alpha})^2.
\end{equation*}
Then, we have
\begin{equation*}
\mathrm{maxroot}\ P_{k}(x;\mathbf{B}\mathbf{B}^{\dagger}\mathbf{A},\mathbf{B})\leq \big(\sqrt{\alpha\cdot t_{}}+\sqrt{(1-\alpha)(1-t_{})}\big )^2\cdot \Vert \mathbf{B}\mathbf{B}^{\dagger}\mathbf{A}\Vert_2^2.
\end{equation*}

\end{lemma}

 To prove Lemma \ref{mth2}, we start with presenting a new expression of the expected polynomial $P_{k}(x;\mathbf{B}\mathbf{B}^{\dagger}\mathbf{A},\mathbf{B})$, which establish a close connection between the expected polynomial $P_{k}(x;\mathbf{B}\mathbf{B}^{\dagger}\mathbf{A},\mathbf{B})$ and the convolution of multi-affine polynomials.

\begin{prop}\label{fk:exp3}
Let $\mathbf{A}\in\mathbb{R}^{n\times d}$, and let $\mathbf{B}\in\mathbb{R}^{n\times d_B}$ be a rank-$m$ matrix. Let $\mathbf{B}\mathbf{B}^{\rm T}=\mathbf{U}\mathbf{\Sigma}\mathbf{U}^{\rm T}$ be the singular value decomposition  of $\mathbf{B}\mathbf{B}^{\rm T}$, where $\mathbf{U}\in\mathbb{R}^{n\times m}$ satisfies $\mathbf{U}^{\rm T}\mathbf{U}=\mathbf{I}_m$, $\mathbf{\Sigma}=\mathrm{diag}(b_1,\ldots,b_m)\in\mathbb{R}^{m\times m}$, and $b_i> 0$ for each $i\in[m]$. Then for each integer $k\in [1, m]$ we have
\begin{equation}\label{section5:fk}
P_{k}(x;\mathbf{B}\mathbf{B}^{\dagger}\mathbf{A},\mathbf{B})=\frac{x^{d+k-m}}{k!}\cdot	 \Big(\prod_{i=1}^m \partial_{z_i}\Big)\cdot h_k(z_1,\ldots,z_m)\cdot \det[\mathbf{Z}- \mathbf{U}^{\rm T}\mathbf{A}\mathbf{A}^{\rm T}\mathbf{U} ]\ \Big|_{z_i=\frac{x}{2},\forall i\in[m]},
\end{equation}
where $\mathbf{Z}:=\mathrm{diag}(z_1,\ldots,z_m)$ and $h_k(z_1,\ldots,z_m)$ is a multi-affine polynomial defined as
\begin{equation}\label{def-hkx}
h_k(z_1,\ldots,z_m):=(b_1\partial_{z_1}+\cdots +b_m\partial_{z_m})^k\cdot \prod_{i=1}^mz_i=k!\sum_{S\subset[m],\vert S\vert=m-k} \mathbf{b}^{S^C}\cdot \mathbf{Z}^S.
\end{equation}

\end{prop}

\begin{proof}
See Appendix \ref{Appendix3}.	
\end{proof}

\begin{remark}
By Proposition \ref{fk:exp3}, we can express the expected polynomial $P_{k}(x;\mathbf{B}\mathbf{B}^{\dagger}\mathbf{A},\mathbf{B})$ in terms of the convolution of multi-affine polynomials introduced by Ravichandran in \cite{ravi4}. Given two multi-affine polynomials $p(z_1,\ldots,z_m)=\sum_{S\subset[m]}a_S\cdot \mathbf{Z}^S$ and $q(z_1,\ldots,z_m)=\sum_{S\subset[m]}b_S\cdot \mathbf{Z}^S$, the convolution of $p$ and $q$ is a multi-affine polynomial defined as
\begin{equation*}\label{convolution-def1}
(p\ast q)(z_1,\ldots,z_m):=\Big(\prod_{i=1}^m \partial_{y_i}\Big)\cdot p(y_1,\ldots,y_m)\cdot q(y_1,\ldots,y_m)\ \Big|_{y_i=\frac{z_i}{2},\forall i\in[m]}.
\end{equation*}
With the notion of the convolution of multi-affine polynomials, we can rewrite \eqref{section5:fk} as
\begin{equation*}
P_{k}(x;\mathbf{B}\mathbf{B}^{\dagger}\mathbf{A},\mathbf{B})= \frac{x^{d+k-m}}{k!}\cdot \Big(	h_k(z_1,\ldots,z_m)\ast \det[\mathbf{Z}- \mathbf{U}^{\rm T}\mathbf{A}\mathbf{A}^{\rm T}\mathbf{U}]\Big)\ \Big|_{z_i=x,\forall i\in[m]}.
\end{equation*}
For more details about convolution of multi-affine polynomials and its relationship with the additive convolution of univariate polynomials,  we refer to \cite{inter5,ravi4,ravi2}.

\end{remark}

Remark \ref{remark:realrootedness of Pkr} implies that $P_{k}(x;\mathbf{B}\mathbf{B}^{\dagger}\mathbf{A},\mathbf{B})$ has only nonnegative real roots.
Therefore, Proposition \ref{fk:exp3} shows that, to estimate the largest root of $P_{k}(x;\mathbf{B}\mathbf{B}^{\dagger}\mathbf{A},\mathbf{B})$, it suffices to estimate the largest root of the polynomial on the right-hand side of \eqref{section5:fk}. 
We need Ravichandran's estimate for the zero free region of the convolution of two multi-affine polynomials. The following lemma is essentially proved in \cite[Theorem 22]{ravi4}, using the method of multivariate barrier method introduced by Batson-Marcus-Spielman-Srivastava in \cite{inter0,inter2}.
We would like to mention that the upper bound $(\sqrt{\alpha\cdot  t}+\sqrt{(1-\alpha)(1-t)})^2$ in \eqref{lemma:ravi4:bound} is an intermediate result of \cite[Theorem 22]{ravi4}. 
Ravichandran first derived this bound and then showed that this bound is no greater than $\alpha+4\cdot \delta^{1/4}$.

\begin{lemma}[{\cite[Theorem 22]{ravi4}}]\label{lemma:ravi4}
Let $h(z_1,\ldots,z_m)=\sum_{S\subset[m]}\mu_S\cdot \mathbf{Z}^S$ be a homogeneous multi-affine polynomial, where $\mu_S\geq 0$ for each $S$.
Assume that $h(z_1,\ldots,z_m)$ is real stable and
\begin{equation*}
\frac{\partial_{z_i}h}{h}\ \Big|_{z_1=\cdots=z_m=1}\leq \delta,\quad\forall \ i\in[m].	
\end{equation*}
Let $\mathbf{M}\in\mathbb{R}^{m\times m}$ be a positive semidefinite matrix satisfying $\mathbf{0}\preceq\mathbf{M}\preceq\mathbf{I}_m$. Let $\alpha$ be the largest diagonal element of $\mathbf{M}$. If $\delta\leq (1-\sqrt{\alpha})^2$, then
\begin{equation}\label{lemma:ravi4:bound}
\mathrm{maxroot}\ \Big(\prod_{i=1}^m \partial_{z_i}\Big)\cdot h(z_1,\ldots,z_m)\cdot \det[\mathbf{Z}- \mathbf{M}]\ \Big|_{z_i=\frac{x}{2},\forall i\in[m]}
\leq \big(\sqrt{\alpha\cdot t}+\sqrt{(1-\alpha)(1-t)}\big )^2.
\end{equation}
\end{lemma}

With the help of Proposition \ref{fk:exp3} and Lemma \ref{lemma:ravi4}, we can provide a proof of Lemma \ref{mth2}.

\begin{proof}[Proof of Lemma \ref{mth2}]
Set
\[
h_k(z_1,\ldots,z_m):=(b_1\partial_{z_1}+\cdots +b_m\partial_{z_m})^k\cdot \prod_{i=1}^mz_i.
 \]
 The real stability of $h_k(z_1,\ldots,z_m)$ follows from Lemma \ref{real stable2} (iii).  Note that, for each $i\in[m]$,
\begin{equation*}
\frac{\partial_{z_i}h_k}{h_k}\ \Big|_{z_1=\cdots=z_m=1}=	\frac{\sum_{S\subset[m]\backslash\{i\},\vert S\vert=k} \mathbf{b}^{S}}{	 \sum_{S\subset[m],\vert S\vert=k} \mathbf{b}^{S}}\leq 	\frac{\sum_{S\subset[m-1],\vert S\vert=k} \mathbf{b}^{S}}{	 \sum_{S\subset[m],\vert S\vert=k} \mathbf{b}^{S}}=\delta_{k},
\end{equation*}
where the inequality ``$\leq$" follows from the assumption that $b_1\geq b_2\geq\cdots \geq b_m>0$. When $\delta_{k}\leq(1-\sqrt{\alpha})^2$, applying Lemma \ref{lemma:ravi4} with $\mathbf{M}=\mathbf{U}^{\rm T}\mathbf{A}\mathbf{A}^{\rm T}\mathbf{U}/\Vert\mathbf{U}^{\rm T}\mathbf{A}\mathbf{A}^{\rm T}\mathbf{U} \Vert_2$, we have
\begin{equation*}
\begin{aligned}
&\quad\ \mathrm{maxroot}\ \Big(\prod_{i=1}^m \partial_{z_i}\Big)\cdot h_k(z_1,\ldots,z_m)\cdot \det[\mathbf{Z}- \mathbf{U}^{\rm T}\mathbf{A}\mathbf{A}^{\rm T}\mathbf{U} ]\ \Big|_{z_i=\frac{x}{2},\forall i\in[m]}\\
&\overset{(a)}=\Vert\mathbf{U}^{\rm T}\mathbf{A}\mathbf{A}^{\rm T}\mathbf{U} \Vert_2\cdot \mathrm{maxroot}\ \Big(\prod_{i=1}^m \partial_{z_i}\Big)\cdot h_k(z_1,\ldots,z_m)\cdot \det[\mathbf{Z}- \frac{\mathbf{U}^{\rm T}\mathbf{A}\mathbf{A}^{\rm T}\mathbf{U}}{\Vert\mathbf{U}^{\rm T}\mathbf{A}\mathbf{A}^{\rm T}\mathbf{U} \Vert_2}]\ \Big|_{z_i=\frac{x}{2},\forall i\in[m]}\\
&\leq \big(\sqrt{\alpha\cdot t}+\sqrt{(1-\alpha)(1-t)}\big )^2\cdot \Vert\mathbf{U}^{\rm T}\mathbf{A}\mathbf{A}^{\rm T}\mathbf{U} \Vert_2.
\end{aligned}
\end{equation*}
Here, in the equality $\overset{(a)}=$, we used the chain rule of differentiation and the fact that $h_k(z_1,\ldots,z_m)$ is a homogeneous multi-affine polynomial.
Since $P_{k}(x;\mathbf{B}\mathbf{B}^{\dagger}\mathbf{A},\mathbf{B})$ has only nonnegative real roots, by Proposition \ref{fk:exp3} we obtain
\begin{equation*}
\begin{aligned}
\mathrm{maxroot}\ P_{k}(x;\mathbf{B}\mathbf{B}^{\dagger}\mathbf{A},\mathbf{B})&=\mathrm{maxroot}\ \Big(\prod_{i=1}^m \partial_{z_i}\Big)\cdot h_k(z_1,\ldots,z_m)\cdot \det[\mathbf{Z}- \mathbf{U}^{\rm T}\mathbf{A}\mathbf{A}^{\rm T}\mathbf{U} ]\ \Big|_{z_i=\frac{x}{2},\forall i\in[m]}\\
&\leq \big(\sqrt{\alpha\cdot t}+\sqrt{(1-\alpha)(1-t)}\big )^2\cdot \Vert\mathbf{U}^{\rm T}\mathbf{A}\mathbf{A}^{\rm T}\mathbf{U} \Vert_2.
\end{aligned}
\end{equation*}
Finally, noting that $\mathbf{U}^{\rm T}\mathbf{U}=\mathbf{I}_m$, $\mathbf{U}\mathbf{U}^{\rm T}=\mathbf{B}\mathbf{B}^{\dagger}$ and
\begin{equation*}\label{eqvov5}
\Vert\mathbf{U}^{\rm T}\mathbf{A}\mathbf{A}^{\rm T}\mathbf{U} \Vert_2
=\Vert\mathbf{U}^{\rm T}\mathbf{U}\cdot \mathbf{U}^{\rm T}\mathbf{A}\mathbf{A}^{\rm T}\mathbf{U} \Vert_2
=\Vert\mathbf{U} \mathbf{U}^{\rm T}\mathbf{A}\mathbf{A}^{\rm T}\mathbf{U}\mathbf{U}^{\rm T} \Vert_2=\Vert \mathbf{B}\mathbf{B}^{\dagger}\mathbf{A}\Vert_2^2,	
\end{equation*}
we arrive at our conclusion.
\end{proof}

\subsubsection{Proof of Theorem \ref{main thm-spr}}

Combining Lemma \ref{mth2-spr} and Lemma \ref{mth2}, we can provide a proof of Theorem \ref{main thm-spr}.

\begin{proof}[Proof of Theorem \ref{main thm-spr}]
By Lemma \ref{mth2-spr} we see that the residual $\Vert\mathbf{A}-\mathbf{B}_{:, \widehat{S}}\mathbf{B}_{:, \widehat{S}}^{\dagger}\mathbf{A} \Vert_{2}^2$ is bounded above by $\Vert\mathbf{A}-\mathbf{B}_{}\mathbf{B}_{}^{\dagger}\mathbf{A} \Vert_{2}^2$ plus the largest root of $P_{k}(x;\mathbf{B}\mathbf{B}^{\dagger}\mathbf{A},\mathbf{B})$. Lemma \ref{mth2} shows that the largest root of $P_{k}(x;\mathbf{B}\mathbf{B}^{\dagger}\mathbf{A},\mathbf{B})$ is at most $\varepsilon\cdot \Vert \mathbf{B}\mathbf{B}^{\dagger}\mathbf{A}\Vert_2^2$. Therefore, the residual $\Vert\mathbf{A}-\mathbf{B}_{:, \widehat{S}}\mathbf{B}_{:, \widehat{S}}^{\dagger}\mathbf{A} \Vert_{2}^2$ is bounded above by $\Vert\mathbf{A}-\mathbf{B}_{}\mathbf{B}_{}^{\dagger}\mathbf{A} \Vert_{2}^2+\varepsilon\cdot \Vert \mathbf{B}\mathbf{B}^{\dagger}\mathbf{A}\Vert_2^2$.

\end{proof}

\section{Submatrix with Bounded Spectral Norm: Proof of Theorem \ref{main thm-submatrix}}

In this section, we aim to prove Theorem \ref{main thm-submatrix}.
To begin, we define $\mathbb{P}^+(d)$ as the set of real-rooted univariate polynomials of degree exactly $d$ that have a positive leading coefficient and only nonnegative roots.

\subsection{Estimation of the Largest Root of $P_{k,k}(x;\mathbf{A},\mathbf{I}_d,\mathbf{I}_d)$}

We first estimate the largest root of $P_{k,k}(x;\mathbf{A},\mathbf{B},\mathbf{C})$ for the case where $\mathbf{B}=\mathbf{C}=\mathbf{I}_d$.
Recall that by Proposition \ref{P-exp-I}, we have
\begin{equation*}
P_{k,k}(x;\mathbf{A},\mathbf{I}_d,\mathbf{I}_d)=\frac{1}{(k!)^2}\cdot x^k\cdot   (\partial_x\cdot x\cdot \partial_x)^k\  \det[x\cdot \mathbf{I}_{d}-\mathbf{A}^{\rm T}\mathbf{A}].	
\end{equation*}
Since the roots of $P_{k,k}(x;\mathbf{A},\mathbf{I}_d,\mathbf{I}_d)$ are all nonnegative, to estimate the largest root of $P_{k,k}(x;\mathbf{A},\mathbf{I}_d,\mathbf{I}_d)$, it suffices to estimate the largest root of $(\partial_x\cdot x\cdot \partial_x)^k\  \det[x\cdot \mathbf{I}_{d}-\mathbf{A}^{\rm T}\mathbf{A}]$.

Our main result in this subsection is the following lemma, which provides an estimate on the largest root of $(\partial_x\cdot x\cdot \partial_x)^k\  p(x)$, where $p(x)$ is a polynomial in $\mathbb{P}^+(d)$. We postpone the proof of Lemma \ref{estimate-DxD} to the end of this subsection.

\begin{lemma}\label{estimate-DxD}
Let $d\geq 2$ and let $p(x)=\prod_{i=1}^d(x-\lambda_i)\in \mathbb{P}^+(d)$, where $\lambda_i$ is the $i$-th largest root of $p(x)$. Assume that $0\leq \lambda_i\leq 1$ for each $i\in[d]$. Set $\beta=\frac{1}{d}\sum_{i=1}^d\lambda_i\in[0,1]$. For each positive integer $k>\frac{\beta}{\beta+1}\cdot d$, we have
\begin{equation*}
\mathrm{maxroot}\ (\partial_x\cdot x\cdot \partial_x)^k\ p(x)	\leq   \Bigg(\Big(1-\beta\Big)\Big(1-\frac{k}{d}\Big)+ 2\sqrt{\Big(1-\frac{k}{d}\Big)\cdot \frac{k}{d}\cdot \beta}\Bigg)^2.
\end{equation*}

\end{lemma}

To establish Lemma \ref{estimate-DxD}, we utilize the barrier function method.  We first introduce some definitions and lemmas.

\begin{definition}
Let $p(x)$ be a degree-$d$ real-rooted polynomial with roots $\lambda_1,\ldots,\lambda_d$. We define the barrier function $\Phi_{p}(x)$ as
\begin{equation*}
\Phi_{p}(x):=\frac{p'(x)}{p(x)}=\sum_{i=1}^d\frac{1}{x-\lambda_i}.	
\end{equation*}
For any positive real number $\alpha\geq 0$, we define
\begin{equation*}
\mathbb{U}_{\alpha} p(x):=p(x)-\alpha\cdot p'(x).	
\end{equation*}	
For $p\in \mathbb{P}^+(d)$ we define the operator
\begin{equation*}
\mathbb{S}(p(x)):=p(x^2).
\end{equation*}
\end{definition}

It follows from the definition that the largest root of a real-rooted polynomial $p$ is bounded above by the largest root of $\mathbb{U}_{\alpha}p$ for any $\alpha>0$.
The following lemma plays a key role in our proof of Lemma \ref{estimate-DxD}, which shows how the roots of a univariate polynomial $p\in \mathbb{P}^+(d)$ shrink under the Laguerre derivative operator.

\begin{lemma}\label{MSS-DxD}
{\rm {\cite[Lemma 4.18]{inter5}}} Let $\alpha\geq 0$, $d\geq 2$ and $p(x)\in \mathbb{P}^+(d)$. Then we have
\begin{equation*}
\mathrm{maxroot}\ \mathbb{U}_{\alpha}\mathbb{S}(\partial_x \cdot x\cdot \partial_x p(x))	\leq \mathrm{maxroot}\ \mathbb{U}_{\alpha}\mathbb{S}(p)-2\alpha,
\end{equation*}
where the equality holds only if $p(x)=x^d$.
\end{lemma}

To prove Lemma \ref{estimate-DxD}, we also need the following lemma.

\begin{lemma}\label{ravi1:lemma}
{\rm {\cite[Lemma 4.5]{ravi1}}} Let $\lambda_1,\ldots,\lambda_m$ be a collection of real numbers at most $1$. Then for any real number $b>1$ we have
\begin{equation*}
\sum_{i=1}^m \frac{1}{b-\lambda_i}\leq \frac{ms}{b-1}+\frac{m(1-s)}{b-t}	,
\end{equation*}
where
\begin{equation*}
\alpha:=\frac{1}{m}\sum_{i=1}^m\lambda_i,\quad \beta:=\frac{1}{m}\sum_{i=1}^m\lambda_i^2,\quad t:=\frac{\alpha-\beta}{1-\alpha}\quad \text{and}\quad  s:=\frac{\beta-\alpha^2}{1-2\alpha+\beta}.
\end{equation*}

\end{lemma}

Now we can present a proof of Lemma \ref{estimate-DxD}.

\begin{proof}[Proof of Lemma \ref{estimate-DxD}]
Let $b>1$ be constants to be determined later, and let $\alpha>0$ be a real number such that $\Phi_{p(x^2)}(b)=\frac{1}{\alpha}$. Let $b_1:=b-2\alpha$ and $p_1(x):=(\partial_x\cdot  x\cdot  \partial_x)\ p(x)$. By Lemma \ref{MSS-DxD}, we obtain that
\begin{equation*}
\mathrm{maxroot}\  p_1(x)<b_1^2 \quad\text{and}\quad \Phi_{p_1(x^2)}(b_1)\leq \Phi_{p(x^2)}(b)= \frac{1}{\alpha}.	
\end{equation*}
Applying Lemma \ref{MSS-DxD} to $p_1(x)$, we obtain that
\begin{equation*}
\mathrm{maxroot}\  p_2(x)<b_2^2 \quad\text{and}\quad \Phi_{p_2(x^2)}(b_2)\leq \Phi_{p_1(x^2)}(b_1)\leq \frac{1}{\alpha},
\end{equation*}
where $b_2:=b_1-2\alpha=b-4\alpha$ and $p_2(x):=(\partial_x\cdot  x\cdot  \partial_x)\ p_1(x)=(\partial_x\cdot  x\cdot  \partial_x)^2\ p(x)$. Repeating this argument for $k-2$ times, we obtain
\begin{equation*}
\mathrm{maxroot}\  (\partial_x\cdot  x\cdot  \partial_x)^k\ p(x)<(b-2k\cdot \alpha)^2,
\end{equation*}
implying that
\begin{equation*}
\mathrm{maxroot}\  (\partial_x\cdot  x\cdot  \partial_x)^k\ p(x)< \Big(\inf_{b>1}\ b-\frac{2k}{\Phi_{p(x^2)}(b)}\Big)^2.
\end{equation*}
By Lemma \ref{ravi1:lemma}, for each $b>1$ we have
\begin{equation*}
\Phi_{p(x^2)}(b)=\sum_{i=1}^{d}\Big(\frac{1}{b-\sqrt{\lambda_i}}+\frac{1}{b+\sqrt{\lambda_i}}\Big)\leq \frac{2d}{b-1}\cdot \frac{\beta}{\beta+1}+\frac{2d}{b+\beta}\cdot\frac{1}{\beta+1}=\frac{2d\cdot (b+\beta-1)}{(b-1)(b+\beta)}.
\end{equation*}
Therefore, we have
\begin{equation}\label{eq1:lemma DxD}
\mathrm{maxroot}\  (\partial_x\cdot  x\cdot  \partial_x)^k\ p(x)< \Big(\inf_{b>1}\ b-\frac{2k}{\Phi_{p(x^2)}(b)}\Big)^2 \leq \Big(\inf_{b>1}\ f(b)\Big)^2,
\end{equation}
where
\begin{equation*}
f(b):=b-\frac{2k\cdot (b-1)(b+\beta)}{2d\cdot (b+\beta-1)}=\Big(1-\beta\Big)\Big(1-\frac{k}{d}\Big)+\Big(1-\frac{k}{d}\Big)\Big(b+\beta-1\Big)+\frac{\frac{k}{d}\cdot \beta}{b+\beta-1}.
\end{equation*}
Note that
\begin{equation*}
f(b)\geq \Big(1-\beta\Big)\Big(1-\frac{k}{d}\Big)+2\sqrt{\Big(1-\frac{k}{d}\Big)\cdot \frac{k}{d}\cdot \beta},
\end{equation*}
where the equality holds if and only if  $b=b_*:=\sqrt{\frac{k\cdot \beta}{d-k}}+1-\beta$. This implies that if $b_*>1$, i.e., $k>\frac{\beta}{\beta+1}\cdot d$, then
\begin{equation*}
\Big(\inf_{b>1}\ f(b)\Big)^2= \Bigg( \Big(1-\beta\Big)\Big(1-\frac{k}{d}\Big)+2\sqrt{\Big(1-\frac{k}{d}\Big)\cdot \frac{k}{d}\cdot \beta}\Bigg)^2.
\end{equation*}
Combining with \eqref{eq1:lemma DxD}, we arrive at our conclusion.
\end{proof}

\subsection{Proof of Theorem \ref{main thm-submatrix}}

By combining Theorem \ref{mth1-spr} and Lemma \ref{estimate-DxD}, we can provide a proof of Theorem \ref{main thm-submatrix}.

\begin{proof}[Proof of Theorem \ref{main thm-submatrix}]
By Theorem \ref{mth1-spr}, there exists a  $(d-k)$-subset ${S}\subset[d]$ and a $(d-k)$-subset ${R}\subset[d]$ such that
\begin{equation}\label{proof thm-submatrix: eq1}
\Vert \mathbf{A}_{ {S}^C, {R}^C}\Vert_{2}^2
\leq \mathrm{maxroot}\ P_{d-k,d-k}(x;\mathbf{A},\mathbf{I}_d,\mathbf{I}_d)
\overset{(a)}=\mathrm{maxroot}\ (\partial_x\cdot x\cdot \partial_x)^{d-k}\  \det[x\cdot \mathbf{I}_{d}-\mathbf{A}^{\rm T}\mathbf{A}],
\end{equation}
where equation ($a$) follows from Proposition \ref{P-exp-I}. Lemma \ref{estimate-DxD} shows that if $d-k>\frac{\beta}{\beta+1}\cdot d$, i.e., $k<\frac{1}{\beta+1}\cdot d$, then
\begin{equation}\label{proof thm-submatrix: eq2}
\mathrm{maxroot}\ (\partial_x\cdot x\cdot \partial_x)^{d-k}\  \det[x\cdot \mathbf{I}_{d}-\mathbf{A}^{\rm T}\mathbf{A}]
\leq 	\Bigg( \Big(1-\beta\Big)\cdot \frac{k}{d}+2\sqrt{\Big(1-\frac{k}{d}\Big)\cdot \frac{k}{d}\cdot \beta}\Bigg)^2.
\end{equation}
Combining \eqref{proof thm-submatrix: eq1} with \eqref{proof thm-submatrix: eq2}, we obtain
\begin{equation*}
\Vert \mathbf{A}_{ {S}^C, {R}^C}\Vert_{2}\leq 	\Big(1-\beta\Big)\cdot \frac{k}{d}+2\sqrt{\Big(1-\frac{k}{d}\Big)\cdot \frac{k}{d}\cdot \beta}.
\end{equation*}
Taking $\widehat{S}=[d]\backslash  {S}$ and 	$\widehat{R}=[d]\backslash  {R}$, we arrive at our conclusion.
\end{proof}


%

\appendix

\section{Appendix}

\subsection{Proof of Proposition \ref{P-exp-base}}\label{proof of P-exp-base}

We begin by proving several lemmas concerning the determinant of block matrices. To facilitate our discussion, we introduce a helpful notation.
For a multivariate polynomial $f(z_1,\ldots,z_n)\in\mathbb{R}[z_1,\ldots,z_n]$ of degree $r_i$ in $z_i$, $i=1,\ldots,n$, we use $\mathrm{Coeff}(z_1^{k_1}\cdots z_n^{k_n},f)$ to denote the coefficient of $z_1^{k_1}\cdots z_n^{k_n}$ in $f(z_1,\ldots,z_n)$, where $0\leq k_i\leq r_i$ for each $i\in[n]$.

\begin{lemma}\label{lemmas: P-exp-base}
\begin{enumerate}

\item[{\rm (i)}] Let $d,n$ be two positive integers, and let $\mathbf{A},\mathbf{B}\in\mathbb{R}^{n\times d}$. Then we have
\begin{equation}\label{lemmas: P-exp-base:eq1}
\det\left[\begin{matrix}
x\cdot  \mathbf{I}_{n} & \mathbf{A}  \\
\mathbf{B}^{\rm T}   & y\cdot \mathbf{I}_{d} \\
\end{matrix}\right]=x^{n-d}\cdot \det[xy\cdot \mathbf{I}_{d}-\mathbf{B}^{\rm T}\mathbf{A}]=y^{d-n}\cdot \det[xy\cdot \mathbf{I}_{n}-\mathbf{A}\mathbf{B}^{\rm T}].
\end{equation}
In particular, we have
\begin{equation}\label{lemmas: P-exp-base:eq2}
\det\left[\begin{matrix}
 \mathbf{0}_{} & \mathbf{A}  \\
\mathbf{A}^{\rm T}   & \mathbf{0}_{} \\
\end{matrix}\right]=\begin{cases}
(-1)^d\cdot \det[\mathbf{A}^{\rm T}]\cdot \det[\mathbf{A}] & \text{if $n=d$,} \\
0 & \text{if $n\neq d$.}
\end{cases}
\end{equation}	

\item[{\rm (ii)}]
Let $d,n,k,r$ be positive integers such that $d+k=r+n$. Let $\mathbf{A}\in\mathbb{R}^{n\times d}$, $\mathbf{B}\in\mathbb{R}^{n\times k}$, and $\mathbf{C}\in\mathbb{R}^{r\times d}$. If $d<r$, then we have $\det\left[\begin{matrix}
\mathbf{B} & \mathbf{A}  \\
\mathbf{0}   & \mathbf{C} \\
\end{matrix}\right]=0$.

\item[{\rm (iii)}]
Let $d,n,m$ be three positive integers, and let $\mathbf{A}\in\mathbb{R}^{n\times d}$ and $\mathbf{B}\in\mathbb{R}^{n\times m}$. Let $k\leq d$ and $r\leq m$ be two nonnegative integers. Then we have
\begin{equation}\label{lemmas: P-exp-base:subset}
\sum_{\substack{S\subset[d]\\|S|=k}}\sum_{\substack{R\subset[m]\\|R|=r}}\det\left[\begin{pmatrix}
(\mathbf{A}_{:,S})^{\rm T}  \\
(\mathbf{B}_{:,R})^{\rm T}
\end{pmatrix}
\begin{pmatrix}
\mathbf{A}_{:,S} & \mathbf{B}_{:,R}  \\
\end{pmatrix}\right]=\mathrm{Coeff}\left(x^{2k}y^{2r},\det\left[\begin{pmatrix}
\mathbf{I}_{d} & \mathbf{0}   \\
\mathbf{0}  & \mathbf{I}_{m}
\end{pmatrix}+
\begin{pmatrix}
x\cdot\mathbf{A}^{\rm T}  \\
y\cdot\mathbf{B}^{\rm T}
\end{pmatrix}
\begin{pmatrix}
x\cdot  \mathbf{A}_{} & y\cdot\mathbf{B}  \\
\end{pmatrix}\right]\right).	
\end{equation}

\item[{\rm (iv)}] Let $d>0,k\geq 0,i\geq 0,r\geq 0$ be integers such that $i\leq d$ and $r\leq d$. Let $\mathbf{A}\in\mathbb{R}^{(k+i)\times d}$, $\mathbf{B}\in\mathbb{R}^{(k+i)\times k}$ and $\mathbf{C}\in\mathbb{R}^{r\times d}$. Set $\mathbf{Q}:=\mathbf{I}_d-\mathbf{C}^{\dagger}\mathbf{C}$. If $0\leq i\leq d-r$ then we have
\begingroup\fontsize{10pt}{12pt}\selectfont
\begin{equation}\label{lemmas: P-exp-base:eq3}
\det[\mathbf{C}\mathbf{C}^{\rm T}]\cdot \sum_{\substack{W\subset[d]\\|W|=i}}
\det\left[\begin{matrix}
\mathbf{B}_{}^{\rm T}   \\
(\mathbf{A}\cdot \mathbf{Q}_{:,W} )^{\rm T}
\end{matrix}\right]
\det\left[\begin{matrix}
 \mathbf{B}_{} & \mathbf{A}\cdot \mathbf{Q}_{:,W}  \\
\end{matrix}\right]
=
\sum_{\substack{P\subset[d]\\|P|=r+i}}	
\det\left[\begin{matrix}
\mathbf{B}_{}^{\rm T}  &  \mathbf{0}_{r\times k}^{\rm T} \\
(\mathbf{A}_{:,P})^{\rm T}  & (\mathbf{C}_{:,P})^{\rm T}
\end{matrix}\right]
\det\left[\begin{matrix}
 \mathbf{B}_{} & \mathbf{A}_{:,P}  \\
\mathbf{0}_{r\times k}   & \mathbf{C}_{:,P}  \\
\end{matrix}\right].
\end{equation}
\endgroup
If $d-r<i\leq d$, then the left hand side of \eqref{lemmas: P-exp-base:eq3} is equal to zero.

\end{enumerate}

\end{lemma}
\begin{proof}
(i) Equation \eqref{lemmas: P-exp-base:eq1} follows from Schur's formula \cite{CH69,Zhang06}. Substituting $y=x=0$ and $\mathbf{B}=\mathbf{A}$ into \eqref{lemmas: P-exp-base:eq1}, we obtain \eqref{lemmas: P-exp-base:eq2}.

(ii) When $d<r$, we can observe that $\mathrm{rank}([\mathbf{0}_{r\times k}\ \  \mathbf{C}])=\mathrm{rank}(\mathbf{C})\leq d<r$.
As a consequence,
 the rows of $[\mathbf{0}_{r\times k}\ \  \mathbf{C}]$ are linear dependent, which implies the desired result.

(iii) By \eqref{expand det}, we can expand
\begin{equation*}
\begin{aligned}
\det\left[\begin{pmatrix}
\mathbf{I}_{d} & \mathbf{0}   \\
  \mathbf{0} & \mathbf{I}_{m}
\end{pmatrix}+
\begin{pmatrix}
x\cdot\mathbf{A}^{\rm T}  \\
y\cdot\mathbf{B}^{\rm T}
\end{pmatrix}
\begin{pmatrix}
x\cdot  \mathbf{A}_{} & y\cdot\mathbf{B}  \\
\end{pmatrix}\right]&=\sum_{S\subset[d]}\sum_{R\subset[m]}\det\left[\begin{pmatrix}
x\cdot(\mathbf{A}_{:,S})^{\rm T}  \\
y\cdot(\mathbf{B}_{:,R})^{\rm T}
\end{pmatrix}
\begin{pmatrix}
x\cdot\mathbf{A}_{:,S} & y\cdot\mathbf{B}_{:,R}  \\
\end{pmatrix}\right]\\
&=\sum_{i=0}^d\sum_{j=0}^mx^{2i}y^{2j} \sum_{\substack{S\subset[d]\\|S|=i}}\sum_{\substack{R\subset[m]\\|R|=j}}\det\left[\begin{pmatrix}
(\mathbf{A}_{:,S})^{\rm T}  \\
(\mathbf{B}_{:,R})^{\rm T}
\end{pmatrix}
\begin{pmatrix}
\mathbf{A}_{:,S} & \mathbf{B}_{:,R}  \\
\end{pmatrix}\right].
\end{aligned}
\end{equation*}
By comparing the coefficients of $x^{2k}y^{2r}$ on both sides of the aforementioned equation, we can deduce \eqref{lemmas: P-exp-base:subset}.

(iv) The case $r=0$ is  straightforward. In the case $k=0$ and $r>0$, \eqref{lemmas: P-exp-base:eq3} reduces to
\begin{equation}\label{lemmas: P-exp-base:eq333}
\det[\mathbf{C}\mathbf{C}^{\rm T}]\cdot \sum_{\substack{W\subset[d]\\|W|=i}}
\det[(\mathbf{A}\cdot \mathbf{Q}_{:,W} )^{\rm T}
]\det[\mathbf{A}\cdot \mathbf{Q}_{:,W}]
=
\sum_{\substack{P\subset[d]\\|P|=r+i}}	
\det\left[\begin{matrix}
(\mathbf{A}_{:,P})^{\rm T}  & (\mathbf{C}_{:,P})^{\rm T}
\end{matrix}\right]
\det\left[\begin{matrix}
 \mathbf{A}_{:,P}  \\
 \mathbf{C}_{:,P}  \\
\end{matrix}\right].
\end{equation}
Applying the Cauchy-Binet formula to both sides of the above equation, \eqref{lemmas: P-exp-base:eq333} becomes
\begin{equation*}
\det[\mathbf{C}\mathbf{C}^{\rm T}]\cdot \det[\mathbf{A} \mathbf{Q}\mathbf{A}^{\rm T}]
=\det\left[\left(\begin{matrix}
 \mathbf{A}  \\
 \mathbf{C}  \\
\end{matrix}\right)\left(\begin{matrix}
\mathbf{A}^{\rm T}  & \mathbf{C}^{\rm T}
\end{matrix}\right)\right].
\end{equation*}
This result can be directly deduced from Lemma \ref{lemma2.2}.
 Therefore, it remains to consider the case where both $k$ and $r$ are greater than zero.
For the remainder of our analysis, we will assume that $k>0$ and $r>0$.
Using the result of (iii), we can rewrite the left-hand side of \eqref{lemmas: P-exp-base:eq3} as
\begingroup\fontsize{10pt}{12pt}\selectfont
\begin{equation*}
\begin{aligned}
\text{LHS of \eqref{lemmas: P-exp-base:eq3}}&=\det[\mathbf{C}\mathbf{C}^{\rm T}]\cdot \mathrm{Coeff}\left(x^{2k}y^{2i},\det\left[\begin{pmatrix}
\mathbf{I}_{k} &  \mathbf{0} \\
  \mathbf{0} & \mathbf{I}_{d}
\end{pmatrix}+
\begin{pmatrix}
x\cdot\mathbf{B}_{}^{\rm T}  \\
y\cdot(\mathbf{A}\cdot \mathbf{Q}_{})^{\rm T}
\end{pmatrix}
\begin{pmatrix}
x\cdot  \mathbf{B}_{} & y\cdot  \mathbf{A}\cdot \mathbf{Q}_{}  \\
\end{pmatrix}\right]\right)  \\
&\overset{(a)}=\det[\mathbf{C}\mathbf{C}^{\rm T}]\cdot \mathrm{Coeff}\left(x^{2k}y^{2i},\det\left[\mathbf{I}_{k+i}+
\begin{pmatrix}
x\cdot  \mathbf{B}_{} & y\cdot  \mathbf{A}   \\
\end{pmatrix}\begin{pmatrix}
\mathbf{I}_{k} &\mathbf{0}  \\
 \mathbf{0} &   \mathbf{Q}_{}
\end{pmatrix}
\begin{pmatrix}
x\cdot\mathbf{B}_{}^{\rm T}  \\
y\cdot\mathbf{A}^{\rm T}
\end{pmatrix}\right]\right) \\
&\overset{(b)}=\det[\mathbf{C}\mathbf{C}^{\rm T}]\cdot \mathrm{Coeff}\bigg(x^{2k}y^{2i},\sum_{S\subset[k+i]}\det\left[
\begin{pmatrix}
x\cdot  \mathbf{B}_{S,:} & y\cdot  \mathbf{A}_{S,:}   \\
\end{pmatrix}
\begin{pmatrix}
\mathbf{I}_{k} & \mathbf{0} \\
 \mathbf{0} &   \mathbf{Q}_{}
\end{pmatrix}
\begin{pmatrix}
x\cdot(\mathbf{B}_{}^{\rm T})_{:,S}  \\
y\cdot(\mathbf{A}^{\rm T} )_{:,S}
\end{pmatrix}\right]\bigg) \\
&\overset{(c)}=\mathrm{Coeff}\bigg(x^{2k}y^{2i},\sum_{S\subset[k+i]}\det\left[
\begin{pmatrix}
x\cdot  \mathbf{B}_{S,:} & y\cdot  \mathbf{A}_{S,:}   \\
\mathbf{0} & \mathbf{C}
\end{pmatrix}
\begin{pmatrix}
x\cdot(\mathbf{B}_{}^{\rm T})_{:,S}& \mathbf{0}  \\
y\cdot(\mathbf{A}^{\rm T})_{:,S}   & \mathbf{C}^{\rm T}_{}
\end{pmatrix}\right]\bigg) \\
&=\mathrm{Coeff}\bigg(x^{2k}y^{2r+2i}z^{2r},y^{2r}z^{2r}\sum_{S\subset[k+i]}\det\left[
\begin{pmatrix}
x\cdot  \mathbf{B}_{S,:} & y\cdot  \mathbf{A}_{S,:}   \\
\mathbf{0} & \mathbf{C}
\end{pmatrix}
\begin{pmatrix}
x\cdot(\mathbf{B}_{}^{\rm T})_{:,S}& \mathbf{0}  \\
y\cdot(\mathbf{A}^{\rm T})_{:,S}   & \mathbf{C}^{\rm T}_{}
\end{pmatrix}\right]\bigg) \\
&\overset{(d)}=\mathrm{Coeff}\bigg(x^{2k}y^{2r+2i}z^{2r},\sum_{R\subset[r]}y^{2|R|}z^{2|R|}\sum_{S\subset[k+i]}\det\left[
\begin{pmatrix}
x\cdot  \mathbf{B}_{S,:} & y\cdot  \mathbf{A}_{S,:}   \\
\mathbf{0} & \mathbf{C}_{R,:}
\end{pmatrix}
\begin{pmatrix}
x\cdot(\mathbf{B}_{}^{\rm T})_{:,S}& \mathbf{0}  \\
y\cdot(\mathbf{A}^{\rm T})_{:,S}   & (\mathbf{C}^{\rm T})_{:,R}
\end{pmatrix}\right]\bigg) \\
&\overset{(e)}=\mathrm{Coeff}\bigg(x^{2k}y^{2r+2i}z^{2r},\det\left[\begin{pmatrix}
\mathbf{I}_{k+i} & \mathbf{0}   \\
\mathbf{0}  & \mathbf{I}_{r}
\end{pmatrix}+
\begin{pmatrix}
x\cdot  \mathbf{B} & y\cdot  \mathbf{A}  \\
\mathbf{0} & yz\cdot \mathbf{C}
\end{pmatrix}
\begin{pmatrix}
x\cdot\mathbf{B}_{}^{\rm T}& \mathbf{0}  \\
y\cdot\mathbf{A}^{\rm T}   & yz\cdot \mathbf{C}^{\rm T}_{}
\end{pmatrix}\right]\bigg) \\
&\overset{(f)}=\mathrm{Coeff}\bigg(x^{2k}y^{2r+2i}z^{2r},\det\left[\begin{pmatrix}
\mathbf{I}_{k} &  \mathbf{0} \\
\mathbf{0}  & \mathbf{I}_{d}
\end{pmatrix}+\begin{pmatrix}
x\cdot\mathbf{B}_{}^{\rm T}& \mathbf{0}  \\
y\cdot\mathbf{A}^{\rm T}   & yz\cdot \mathbf{C}^{\rm T}_{}
\end{pmatrix}
\begin{pmatrix}
x\cdot  \mathbf{B} & y\cdot  \mathbf{A}  \\
\mathbf{0} & yz\cdot \mathbf{C}
\end{pmatrix}
\right]\bigg). \\
\end{aligned}	
\end{equation*}
\endgroup
Here, ``LHS of \eqref{lemmas: P-exp-base:eq3}" refers to the left-hand side term of \eqref{lemmas: P-exp-base:eq3}. 
Equations ($a$) and ($f$) follow from the Weinstein-Aronszajn identity \eqref{Weinstein-Aronszajn}. Equations ($b$) and ($e$) can be derived from \eqref{expand det}. Equation ($c$) can be obtained by applying Lemma \ref{lemma2.2}, and equation ($d$) follows from the the fact that the polynomial 
\begin{equation*}
y^{2|R|}z^{2|R|}\sum_{S\subset[k+i]}\det\left[
\begin{pmatrix}
x\cdot  \mathbf{B}_{S,:} & y\cdot  \mathbf{A}_{S,:}   \\
\mathbf{0} & \mathbf{C}_{R,:}
\end{pmatrix}
\begin{pmatrix}
x\cdot(\mathbf{B}_{}^{\rm T})_{:,S}& \mathbf{0}  \\
y\cdot(\mathbf{A}^{\rm T})_{:,S}   & (\mathbf{C}^{\rm T})_{:,R}
\end{pmatrix}\right]	
\end{equation*}
contains the term $x^{2k}y^{2r+2i}z^{2r}$ only if $|R|=r$, i.e., $R=[r]$.
Then, we utilize  \eqref{expand det} to obtain that
\begin{equation}\label{eq:LHScoeff}
\begin{aligned}
\text{LHS of \eqref{lemmas: P-exp-base:eq3}}
&=\mathrm{Coeff}\bigg(x^{2k}y^{2r+2i}z^{2r}, \sum_{T\subset[k]}x^{2|T|}\sum_{P\subset[d]}y^{2|P|}\det\left[\begin{pmatrix}
(\mathbf{B}_{:,T})^{\rm T}& \mathbf{0}  \\
(\mathbf{A}_{:,P})^{\rm T}   & z\cdot (\mathbf{C}_{:,P})^{\rm T}_{}
\end{pmatrix}
\begin{pmatrix}
 \mathbf{B}_{:,T} &  \mathbf{A}_{:,P}  \\
\mathbf{0} & z\cdot \mathbf{C}_{:,P}
\end{pmatrix}
\right]\bigg)\\
&=\mathrm{Coeff}\bigg(y^{2r+2i}z^{2r}, \sum_{P\subset[d]}y^{2|P|}\det\left[\begin{pmatrix}
\mathbf{B}^{\rm T}& \mathbf{0}  \\
(\mathbf{A}_{:,P})^{\rm T}   & z\cdot (\mathbf{C}_{:,P})^{\rm T}_{}
\end{pmatrix}
\begin{pmatrix}
 \mathbf{B} &  \mathbf{A}_{:,P}  \\
\mathbf{0} & z\cdot \mathbf{C}_{:,P}
\end{pmatrix}
\right]\bigg).
\end{aligned}	
\end{equation}
From the above equation, we can observe that if $r+i>d$, then the left-hand side of \eqref{lemmas: P-exp-base:eq3} equals zero. Otherwise, we have
\begin{equation*}
\begin{aligned}
\text{LHS of \eqref{lemmas: P-exp-base:eq3}}
&\overset{(a)}=\mathrm{Coeff}\bigg(z^{2r},\sum_{P\subset[d],|P|=r+i}\det\left[\begin{pmatrix}
\mathbf{B}_{}^{\rm T}& \mathbf{0}  \\
(\mathbf{A}_{:,P})^{\rm T}   & z\cdot (\mathbf{C}_{:,P})^{\rm T}_{}
\end{pmatrix}
\begin{pmatrix}
 \mathbf{B} &  \mathbf{A}_{:,P}  \\
\mathbf{0} & z\cdot \mathbf{C}_{:,P}
\end{pmatrix}
\right]\bigg) \\
&=\mathrm{Coeff}\bigg(z^{2r},\sum_{P\subset[d],|P|=r+i}z^{2r}\det\left[\begin{matrix}
\mathbf{B}_{}^{\rm T}& \mathbf{0}  \\
(\mathbf{A}_{:,P})^{\rm T}   &  (\mathbf{C}_{:,P})^{\rm T}_{}
\end{matrix}
\right]
\det\left[
\begin{matrix}
 \mathbf{B} &  \mathbf{A}_{:,P}  \\
\mathbf{0} &  \mathbf{C}_{:,P}
\end{matrix}
\right]\bigg) \\
&=\sum_{P\subset[d],|P|=r+i}\det\left[\begin{matrix}
\mathbf{B}_{}^{\rm T}& \mathbf{0}  \\
(\mathbf{A}_{:,P})^{\rm T}   &  (\mathbf{C}_{:,P})^{\rm T}_{}
\end{matrix}
\right]
\det\left[
\begin{matrix}
 \mathbf{B} &  \mathbf{A}_{:,P}  \\
\mathbf{0} & \mathbf{C}_{:,P}
\end{matrix}
\right]=\text{RHS of \eqref{lemmas: P-exp-base:eq3}}.
\end{aligned}	
\end{equation*}
Here, equation ($a$) follows from (\ref{eq:LHScoeff}) and
``RHS of \eqref{lemmas: P-exp-base:eq3}" refers to the right-hand side term of \eqref{lemmas: P-exp-base:eq3}.
Therefore, we arrive at our conclusion.

\end{proof}

Now we can prove Proposition \ref{P-exp-base}.
\begin{proof}[Proof of Proposition \ref{P-exp-base}]
Recall that $k\leq \mathrm{rank}(\mathbf{B})\leq \min\{n,d_B\}$, $r\leq \mathrm{rank}(\mathbf{C})\leq \min\{n_C,d\}$, and
\begin{equation*}
H(x,w,\mathbf{Y},\mathbf{Z};\mathbf{A},\mathbf{B},\mathbf{C})=\det\left[\begin{matrix}
w\cdot \mathbf{I}_n  & \mathbf{0} & \mathbf{B} & \mathbf{A}  \\
\mathbf{0}  & \mathbf{Z} & \mathbf{0}  & \mathbf{C}  \\
\mathbf{B}^{\rm T}  &  \mathbf{0} & \mathbf{Y} & \mathbf{0}  \\
\mathbf{A}^{\rm T}  & \mathbf{C}^{\rm T} & \mathbf{0}  & x\cdot  \mathbf{I}_d
\end{matrix}\right].
\end{equation*}
By \eqref{expand det} we can expand $H(x,w,\mathbf{Y},\mathbf{Z};\mathbf{A},\mathbf{B},\mathbf{C})$ as
\begingroup\fontsize{10pt}{12pt}\selectfont
\begin{equation*}
H(x,w,\mathbf{Y},\mathbf{Z};\mathbf{A},\mathbf{B},\mathbf{C})
=\sum_{M\subset[n_C]}\sum_{N\subset[d_B]}\sum_{L\subset[n]}\sum_{P\subset[d]} \det\left[\begin{matrix}
\mathbf{0}  & \mathbf{0} & \mathbf{B}_{L,N} & \mathbf{A}_{L,P}  \\
\mathbf{0}  & \mathbf{0} & \mathbf{0}  & \mathbf{C}_{M,P}  \\
(\mathbf{B}_{L,N})^{\rm T}  & \mathbf{0}  & \mathbf{0} & \mathbf{0}  \\
(\mathbf{A}_{L,P})^{\rm T}  & (\mathbf{C}_{M,P})^{\rm T} & \mathbf{0}  & \mathbf{0}
\end{matrix}\right]\cdot x^{d-|P|}\cdot w^{n-|L|}\cdot \mathbf{Y}^{N^C}\cdot \mathbf{Z}^{M^C}.
\end{equation*}
\endgroup
Then we have
\begingroup\fontsize{10pt}{12pt}\selectfont
\begin{equation*}
\begin{aligned}
\text{RHS of \eqref{eq:P-exp-base}}
&=(-1)^{k+r}\cdot x^r\cdot w^{d+k-n}\cdot \sum_{L\subset[n]}\sum_{P\subset[d]} \det\left[\begin{matrix}
\mathbf{0}  & \mathbf{0} & \mathbf{B}_{L,S} & \mathbf{A}_{L,P}  \\
\mathbf{0}  & \mathbf{0} & \mathbf{0}  & \mathbf{C}_{R,P}  \\
(\mathbf{B}_{L,S})^{\rm T}  & \mathbf{0}  & \mathbf{0} & \mathbf{0}  \\
(\mathbf{A}_{L,P})^{\rm T}  & (\mathbf{C}_{R,P})^{\rm T} &  \mathbf{0} &\mathbf{0}
\end{matrix}\right]\cdot x^{d-|P|}\cdot w^{n-|L|}\\
&\overset{(a)}=
\sum_{i=0}^{\min\{d-r,n-k\}}(-1)^{i}
\sum_{\substack{L\subset[n]\\|L|=k+i}}\sum_{\substack{P\subset[d]\\|P|=r+i}}
\det\left[\begin{matrix}
(\mathbf{B}_{L,S})^{\rm T}  &  \mathbf{0}_{k\times r} \\
(\mathbf{A}_{L,P})^{\rm T}  & (\mathbf{C}_{R,P})^{\rm T}
\end{matrix}\right]
\det\left[\begin{matrix}
 \mathbf{B}_{L,S} & \mathbf{A}_{L,P}  \\
\mathbf{0}_{r\times k}   & \mathbf{C}_{R,P}  \\
\end{matrix}\right]
\cdot (x\cdot w)^{d-i}\\
&\overset{(b)}=\det[\mathbf{C}_{R,:}(\mathbf{C}_{R,:})^{\rm T}]
\sum_{i=0}^{\min\{d,n-k\}}(-1)^{i}\sum_{\substack{L\subset[n]\\|L|=k+i}}
\sum_{\substack{W\subset[d]\\|W|=i}}
\det\left[\begin{matrix}
(\mathbf{B}_{L,S})^{\rm T}   \\
(\mathbf{A}_{L,:} (\mathbf{P}_{R})_{:,W} )^{\rm T}
\end{matrix}\right]
\det\left[\begin{matrix}
 \mathbf{B}_{L,S} & \mathbf{A}_{L,:} (\mathbf{P}_{R})_{:,W}  \\
\end{matrix}\right]
\cdot (x\cdot w)^{d-i}\\
&\overset{(c)}=\det[\mathbf{C}_{R,:}(\mathbf{C}_{R,:})^{\rm T}]
\sum_{i=0}^{\min\{d,n-k\}}(-1)^{i}
\sum_{\substack{W\subset[d]\\|W|=i}}
\det\left[\begin{pmatrix}
(\mathbf{B}_{:,S})^{\rm T}   \\
(\mathbf{A}_{} (\mathbf{P}_{R})_{:,W} )^{\rm T}
\end{pmatrix}\cdot
\begin{pmatrix}
 \mathbf{B}_{:,S} & \mathbf{A} (\mathbf{P}_{R})_{:,W}  \\
\end{pmatrix}\right]
\cdot (x\cdot w)^{d-i}.
\end{aligned}
\end{equation*}
\endgroup
Here, equation ($a$) is derived from \eqref{lemmas: P-exp-base:eq2} and Lemma \ref{lemmas: P-exp-base} (ii). Equation ($b$) is a result of Lemma \ref{lemmas: P-exp-base} (iv), and equation ($c$) is obtained from the application of the Cauchy-Binet formula.

If $|W|>n-k$ then $\begin{pmatrix}
 \mathbf{B}_{:,S} & \mathbf{A} (\mathbf{P}_{R})_{:,W}  \\
\end{pmatrix}\in\mathbb{R}^{n\times (k+|W|)}$ is a short fat matrix, implying that
\begin{equation*}	
\det\left[\begin{pmatrix}
(\mathbf{B}_{:,S})^{\rm T}   \\
(\mathbf{A}_{} (\mathbf{P}_{R})_{:,W} )^{\rm T}
\end{pmatrix}\cdot
\begin{pmatrix}
 \mathbf{B}_{:,S} & \mathbf{A} (\mathbf{P}_{R})_{:,W}  \\
\end{pmatrix}\right]=0.
\end{equation*}
Therefore, we further have
\begin{equation*}
\begin{aligned}
\text{RHS of \eqref{eq:P-exp-base}}&=\det[\mathbf{C}_{R,:}(\mathbf{C}_{R,:})^{\rm T}]\cdot
\sum_{i=0}^{d}(-1)^{i}
\sum_{\substack{W\subset[d]\\|W|=i}}
\det\left[\begin{pmatrix}
(\mathbf{B}_{:,S})^{\rm T}   \\
(\mathbf{A}_{} (\mathbf{P}_{R})_{:,W} )^{\rm T}
\end{pmatrix}\cdot
\begin{pmatrix}
 \mathbf{B}_{:,S} & \mathbf{A} (\mathbf{P}_{R})_{:,W}  \\
\end{pmatrix}\right]
\cdot (x\cdot w)^{d-i}\\
&=\det[\mathbf{C}_{R,:}(\mathbf{C}_{R,:})^{\rm T}]\det[\mathbf{B}_{:,S}^{\rm T} \mathbf{B}_{:,S}]
\sum_{i=0}^{d}(-1)^{i}
\sum_{\substack{W\subset[d]\\|W|=i}}
\det[ (\mathbf{P}_{R})_{W,:}\mathbf{A}_{}^{\rm T} \mathbf{Q}_S \mathbf{A}_{} (\mathbf{P}_{R})_{:,W}]
\cdot (x\cdot w)^{d-i}\\
&=\det[\mathbf{C}_{R,:}(\mathbf{C}_{R,:})^{\rm T}]\cdot \det[\mathbf{B}_{:,S}^{\rm T} \mathbf{B}_{:,S}] \cdot \det[x\cdot w\cdot \mathbf{I}_{d}-\mathbf{P}_{R}\mathbf{A}^{\rm T}\mathbf{Q}_S \mathbf{A}_{}\mathbf{P}_{R}]=\text{LHS of \eqref{eq:P-exp-base}},
\end{aligned}	
\end{equation*}
where the second equation follows from Lemma \ref{lemma2.2}. This completes the proof.

\end{proof}

\subsection{Proof of Proposition \ref{P-three-exp}}\label{proof of P-three-exp}

Here we present a proof of Proposition \ref{P-three-exp}.

\begin{proof}[Proof of Proposition \ref{P-three-exp}]
We begin by demonstrating the validity of \eqref{P-exp}.
  By Leibniz formula, for any multi-affine polynomial $g(y_1,\ldots,y_m)$ we have the following algebraic identity:
\begin{equation}\label{Leibniz formula}
\sum_{S\subset[m],\vert S\vert=k}  \partial_{\mathbf{Y}^{S^C}}\ 	g(y_1,\ldots,y_m)\ \Big|_{y_i=y,\forall i}=\frac{1}{(m-k)!}\cdot  \partial_y^{m-k}\ g(y,\ldots,y).
\end{equation}
Therefore, we arrive at
\begin{equation*}
\begin{aligned}
P_{k,r}(x;\mathbf{A},\mathbf{B},\mathbf{C})
&\overset{(a)}=(-1)^{k+r}\cdot  x^r\sum_{R\subset[n_C],|R|=r}	 \sum_{S\subset[d_B],|S|=k} \partial_{\mathbf{Y}^{S^C}} \partial_{\mathbf{Z}^{R^C}}\ H(x,1,\mathbf{Y},\mathbf{Z};\mathbf{A},\mathbf{B},\mathbf{C})  \ \Bigg\vert_{\substack{y_i=0,\forall i\in[d_B]\\z_j=0,\forall j\in[n_C]}}\\
&\overset{(b)}=\frac{(-1)^{k+r}}{ (d_B-k)!(n_C-r)!}\cdot x^{r} \cdot \partial_y^{d_B-k}\partial_z^{n_C-r}\ H(x,y,z,1;\mathbf{A},\mathbf{B},\mathbf{C})  \ \Big\vert_{y=z=0},
\end{aligned}
\end{equation*}
where ($a$) follows from Proposition \ref{P-exp-base} by taking $w=1$,  and ($b$) follows from \eqref{Leibniz formula}.
We then proceed to prove  \eqref{P-exp-1}. A straightforward calculation yields:
\begin{equation*}
\begin{aligned}
H(x,y,z,1;\mathbf{A},\mathbf{B},\mathbf{C})&=x^{d}y^{d_B}z^{n_C}\cdot \det\left[\begin{matrix}
\mathbf{I}_n  & \mathbf{0} & \frac{1}{y}\cdot  \mathbf{B} & \frac{1}{x}\cdot\mathbf{A}  \\
\mathbf{0}  & \mathbf{I}_{n_C} &  \mathbf{0} & \frac{1}{xz}\cdot\mathbf{C}  \\
\mathbf{B}^{\rm T}  &  \mathbf{0} & \mathbf{I}_{d_B} & \mathbf{0}  \\
\mathbf{A}^{\rm T}  & \mathbf{C}^{\rm T} &  \mathbf{0} & \mathbf{I}_d
\end{matrix}\right]\\
&\overset{(a)}=x^{d}y^{d_B}z^{n_C}\cdot \det\left[\mathbf{I}_{d_B+d}-
\begin{pmatrix}
\mathbf{B}^{\rm T}  &   \mathbf{0}   \\
\mathbf{A}^{\rm T}  & \mathbf{C}^{\rm T}
\end{pmatrix}\begin{pmatrix}
\frac{1}{y}\cdot  \mathbf{B} & \frac{1}{x}\cdot\mathbf{A}\\
\mathbf{0}  & \frac{1}{xz}\cdot\mathbf{C}
\end{pmatrix}\right]\\
&=z^{n_C} \det\left[
\begin{matrix}
y\cdot \mathbf{I}_{d_B}-\mathbf{B}^{\rm T}\mathbf{B}  &  - \mathbf{B}^{\rm T}\mathbf{A}   \\
-\mathbf{A}^{\rm T} \mathbf{B} &  x\cdot \mathbf{I}_{d}- \mathbf{A}^{\rm T} \mathbf{A}-\frac{1}{z}\cdot\mathbf{C}^{\rm T} \mathbf{C}
\end{matrix}\right]\\
&=z^{n_C}
\det\left[
\begin{pmatrix}
x\cdot \mathbf{I}_d  -\frac{1}{z}\cdot
\mathbf{C}^{\rm T}\mathbf{C} & \mathbf{0}  \\
\mathbf{0}  & y\cdot \mathbf{I}_{d_B}
\end{pmatrix}
-
\begin{pmatrix}
\mathbf{A}^{\rm T}   \\
\mathbf{B}^{\rm T}
\end{pmatrix}
\begin{pmatrix}
\mathbf{A} & \mathbf{B}
\end{pmatrix}
\right],
\end{aligned}	
\end{equation*}
where ($a$) follows from \eqref{lemmas: P-exp-base:eq1} in Lemma \ref{lemmas: P-exp-base} (i).
Also note that
\begin{equation*}
\frac{1}{(n_C-r)!}\cdot \partial_z^{n_C-r}\ H(x,y,z,1;\mathbf{A},\mathbf{B},\mathbf{C})  \ \Big\vert_{z=0}= \frac{1}{r!}\cdot \partial_z^{r}\ \big( z^{n_C}\cdot  H(x,y,1/z,1;\mathbf{A},\mathbf{B},\mathbf{C})\big)  \ \Big\vert_{z=0}.
\end{equation*}
Therefore, combining with \eqref{P-exp} we arrive at
\begin{equation*}
\begin{aligned}
P_{k,r}(x;\mathbf{A},\mathbf{B},\mathbf{C})&=\frac{(-1)^{k+r}}{r! (d_B-k)!}\cdot x^{r} \cdot \partial_y^{d_B-k}\partial_z^{r}\big( z^{n_C}\cdot  H(x,y,1/z,1;\mathbf{A},\mathbf{B},\mathbf{C})  \big)\ \Big\vert_{y=z=0}\\
&=\frac{(-1)^{k+r}}{ r!(d_B-k)!}\cdot x^{r} \cdot \partial_y^{d_B-k}\partial_z^{r}\
\det\left[
\begin{pmatrix}
x\cdot \mathbf{I}_d -z\cdot \mathbf{C}^{\rm T}\mathbf{C} & \mathbf{0}  \\
\mathbf{0}  & y\cdot \mathbf{I}_{d_B}
\end{pmatrix}
-
\begin{pmatrix}
\mathbf{A}^{\rm T}   \\
\mathbf{B}^{\rm T}
\end{pmatrix}
\begin{pmatrix}
\mathbf{A} & \mathbf{B}
\end{pmatrix}
\right]\ \Bigg\vert_{y=z=0}\\
&=\frac{(-1)^{k}}{ r!(d_B-k)!}\cdot x^{r} \cdot \partial_y^{d_B-k}\partial_z^{r}\
\det\left[
\begin{pmatrix}
x\cdot \mathbf{I}_d +z\cdot \mathbf{C}^{\rm T}\mathbf{C} & \mathbf{0}  \\
\mathbf{0}  & y\cdot \mathbf{I}_{d_B}
\end{pmatrix}
-
\begin{pmatrix}
\mathbf{A}^{\rm T}   \\
\mathbf{B}^{\rm T}
\end{pmatrix}
\begin{pmatrix}
\mathbf{A} & \mathbf{B}
\end{pmatrix}
\right]\ \Bigg\vert_{y=z=0},
\end{aligned}	
\end{equation*}
where the last equation follows from the chain rule.

Finally, applying \eqref{P-exp-1} to the matrices $\mathbf{A}^{\rm T},\mathbf{C}^{\rm T}$, $\mathbf{B}^{\rm T}$ and using Lemma \ref{P-sym}, we immediately obtain \eqref{P-exp-2}.
	
\end{proof}

\subsection{Proof of Proposition \ref{P-exp-r=0}}\label{proof of P-exp-r=0}

Here we present a proof of Proposition \ref{P-exp-r=0}.

\begin{proof}[Proof of Proposition \ref{P-exp-r=0}]
Substituting $r=0$ into \eqref{P-exp-1} we immediately obtain \eqref{lemma_f_k:expression1}. We next prove \eqref{lemma_f_k:expression2}. A simple calculation shows that
\begin{equation*}
\begin{aligned}
f(x,y)&:=\det\left[\begin{pmatrix}
x\cdot \mathbf{I}_d  &  \mathbf{0}  \\
\mathbf{0} & y\cdot\mathbf{I}_{d_B}
\end{pmatrix}-\begin{pmatrix}
\mathbf{A}^{\rm T}     \\
\mathbf{B}^{\rm T}
\end{pmatrix}\begin{pmatrix}
\mathbf{A}  &  \mathbf{B}
\end{pmatrix}\right]=x^dy^{d_B}\cdot \det\left[\begin{pmatrix}
\mathbf{I}_d  &  \mathbf{0}  \\
\mathbf{0} & \mathbf{I}_{d_B}
\end{pmatrix}-\begin{pmatrix}
\mathbf{A}^{\rm T}     \\
\mathbf{B}^{\rm T}
\end{pmatrix}\begin{pmatrix}
\frac{1}{x}\cdot \mathbf{A}  & \frac{1}{y}\cdot  \mathbf{B}
\end{pmatrix}\right]	\\
&\overset{(a)}=x^dy^{d_B}\cdot  \det[\mathbf{I}_n-\frac{1}{x}\cdot \mathbf{A}\mathbf{A}^{\rm T}-\frac{1}{y}\cdot  \mathbf{B}\mathbf{B}^{\rm T}],\\
\end{aligned}
\end{equation*}
where equation ($a$) follows from the Weinstein-Aronszajn identity \eqref{Weinstein-Aronszajn}. Note that
\begin{equation}\label{lemma4proof:eq2}
\begin{aligned}
\frac{1}{(d_B-k)!}\cdot\partial_y^{d_B-k} f(x,y)\ \big|_{y=0}&=\frac{1}{k!}\cdot\partial_y^{k}\big(  y^{d_B}\cdot f(x,1/y)\big) \ \big|_{y=0}\\
&=\frac{1}{k!}\cdot     \partial_y^{k}\big(x^d\cdot   \det[\mathbf{I}_n- \frac{1}{x}\cdot \mathbf{A}\mathbf{A}^{\rm T}-y\cdot  \mathbf{B}\mathbf{B}^{\rm T}] \big) \ \big|_{y=0}\\
&=\frac{1}{k!}\cdot x^{d-n}\cdot \partial_y^{k}\big( \det[x\cdot \mathbf{I}_n- \mathbf{A}\mathbf{A}^{\rm T}-xy\cdot  \mathbf{B}\mathbf{B}^{\rm T}] \big) \ \big|_{y=0}\\
&=\frac{1}{k!}\cdot x^{d-n+k}\cdot \partial_y^{k}\big( \det[x\cdot \mathbf{I}_n- \mathbf{A}\mathbf{A}^{\rm T}-y\cdot  \mathbf{B}\mathbf{B}^{\rm T}] \big) \ \big|_{y=0},
\end{aligned}
\end{equation}
where the last equation follows from the chain rule.
Substituting \eqref{lemma4proof:eq2} into \eqref{lemma_f_k:expression1}, we arrive at \eqref{lemma_f_k:expression2}.

\end{proof}

\subsection{Proof of Proposition \ref{P-exp-I}}\label{proof of P-exp-I}

Here we present a proof of Proposition \ref{P-exp-I}.
We need the following lemma. It can be proved by repeatedly using Thompson's formula \cite{Thompson}, which states that the derivative of a matrix's characteristic polynomial equals the sum of the characteristic polynomials of its defect $1$ submatrices.

\begin{lemma}{\rm \cite[Lemma 4.1]{ravi3}}\label{lemmas: det}
Let $\mathbf{M}\in\mathbb{R}^{d\times d}$ be a square matrix. For any integer $k\in[d]$, we have
\begin{equation*}
\frac{1}{k!}\sum_{S\subset[d],|S|=d-k}\det[x\cdot \mathbf{I}_{d-k}-\mathbf{M}_{S,S}]	= \partial_x^{k}\det[x\cdot \mathbf{I}_{d}-\mathbf{M}].
\end{equation*} 
\end{lemma}

Now we present a proof of Proposition \ref{P-exp-I}.

\begin{proof}[Proof of Proposition \ref{P-exp-I}]
We first prove equation \eqref{P-exp-II-oct}.
By Definition \ref{def2} we have
\begin{equation}\label{octeq-2024-4}
\begin{aligned}
P_{k-|S|,r-|R|}(x;\mathbf{Q}_{S}\mathbf{A}\mathbf{P}_{R},\mathbf{Q}_{S},\mathbf{P}_{R})
&=\sum_{\substack{T\subset R^C\\|T|=r-|R|}}\sum_{\substack{W\subset S^C\\|W|=k-|S|}}\det[x\cdot\mathbf{I}_d-\mathbf{P}_{R\cup T}\mathbf{A}^{\rm T}\mathbf{Q}_{S\cup W}\mathbf{A}\mathbf{P}_{R\cup T} ],\\
&=\sum_{\substack{T\subset R^C\\|T|=r-|R|}}\sum_{\substack{W\subset S^C\\|W|=k-|S|}}\det[x\cdot\mathbf{I}_d-\mathbf{Q}_{S\cup W}\mathbf{A}\mathbf{P}_{R\cup T}\mathbf{A}^{\rm T}\mathbf{Q}_{S\cup W} ],\\
\end{aligned}
\end{equation}
where we use $\mathbf{B}=\mathbf{C}=\mathbf{I}_d$ in the first equation and use the Weinstein-Aronszajn identity \eqref{Weinstein-Aronszajn2} in the second equation.  
Note that
\begingroup\fontsize{10pt}{12pt}\selectfont
\begin{equation*}
\begin{aligned}
\sum_{\substack{W\subset S^C\\|W|=k-|S|}}\det[x\cdot\mathbf{I}_d-\mathbf{Q}_{S\cup W}\mathbf{A}\mathbf{P}_{R\cup T}\mathbf{A}^{\rm T}\mathbf{Q}_{S\cup W} ]
&\overset{(a)}=x^k \sum_{\substack{W\subset S^C\\|W|=k-|S|}}\det[x\cdot\mathbf{I}_{d-k}-\mathbf{A}_{S^C\backslash W,R^C\backslash T}(\mathbf{A}_{S^C\backslash W,R^C\backslash T})^{\rm T} ]\\
&\overset{(b)}=\frac{x^k}{(k-|S|)!} \partial_x^{k-|S|} \det[x\cdot\mathbf{I}_{d-|S|}-\mathbf{A}_{S^C,R^C\backslash T}(\mathbf{A}_{S^C,R^C\backslash T})^{\rm T} ]\\
&\overset{(c)}=\frac{x^k}{(k-|S|)!} \partial_x^{k-|S|}\Big( x^{r-|S|}  \det[x\cdot\mathbf{I}_{d-r}-(\mathbf{A}_{S^C,R^C\backslash T})^{\rm T}\mathbf{A}_{S^C,R^C\backslash T} ]\Big),
\end{aligned}
\end{equation*}
\endgroup
where equation ($a$) follows from the assumption that $\mathbf{B}=\mathbf{C}=\mathbf{I}_d$,  equation ($b$) follows from Lemma \ref{lemmas: det}, and equation ($c$) follow from the Weinstein-Aronszajn identity \eqref{Weinstein-Aronszajn2}. Substituting the above equation into \eqref{octeq-2024-4}, we obtain
\begingroup\fontsize{10pt}{12pt}\selectfont
\begin{equation*}
\begin{aligned}
P_{k-|S|,r-|R|}(x;\mathbf{Q}_{S}\mathbf{A}\mathbf{P}_{R},\mathbf{Q}_{S},\mathbf{P}_{R})
&=\frac{x^k}{(k-|S|)!} \partial_x^{k-|S|}
\Big( x^{r-|S|} \sum_{\substack{T\subset R^C\\|T|=r-|R|}} \det[x\cdot\mathbf{I}_{d-r}-(\mathbf{A}_{S^C,R^C\backslash T})^{\rm T}\mathbf{A}_{S^C,R^C\backslash T} ]\Big)\\
&=\frac{x^k}{(k-|S|)!(r-|R|)!} \partial_x^{k-|S|}
\Big( x^{r-|S|} \partial_x^{r-|R|} \det[x\cdot\mathbf{I}_{d-|R|}-(\mathbf{A}_{S^C,R^C})^{\rm T}\mathbf{A}_{S^C,R^C} ]\Big),
\end{aligned}	
\end{equation*}
\endgroup
where we use the linearity of the differentiation in the first equation, and use Lemma \ref{lemmas: det} again in the second equation.  Thus, we arrive at equation \eqref{P-exp-II-oct}.

We next prove \eqref{P-exp-II}. Take $S=R=\emptyset$. Then we have $\mathbf{Q}_{S}=\mathbf{I}_n$ and $\mathbf{P}_{R}=\mathbf{I}_d$, and equation \eqref{P-exp-II-oct} immediately simplifies to \eqref{P-exp-II}. Finally, note that the Laguerre derivative operator $\partial_x\cdot x\cdot \partial_x$ satisfies the identity $(\partial_x\cdot x\cdot \partial_x)^l=\partial_x^l\cdot x^l\cdot \partial_x^l$ for any positive integer $l$ \cite{RS22,Vis94}. Therefore, \eqref{P-exp-II} becomes \eqref{P-exp-II-square} when $r=k$.

\end{proof}

\subsection{Proof of Proposition \ref{P-recursive2}}\label{proof of P-recursive2}

Here we present a proof of Proposition \ref{P-recursive2}. We need the following lemmas.

\begin{lemma}\label{lemma2.0}{\rm \cite[Appendix]{MK85}}
Let $\mathbf{Q}\in\mathbb{R}^{n\times n}$  be an orthogonal projection matrix and let $\mathbf{B}\in\mathbb{R}^{n\times d}$. Then we have $(\mathbf{Q}\mathbf{B})^{\dagger}\mathbf{Q}=(\mathbf{Q}\mathbf{B})^{\dagger}$.

\end{lemma}

\begin{lemma}
Let $\mathbf{A}\in\mathbb{R}^{n\times d}$, $\mathbf{B}=[\mathbf{b}_1,\ldots,\mathbf{b}_{d_B}]\in\mathbb{R}^{n\times d_{{B}}}$, $\mathbf{C}=[\mathbf{c}_1,\ldots,\mathbf{c}_{n_C}]^{\rm T}\in\mathbb{R}^{n_{{C}}\times d}$. Let $k$ and $r$ be two integers satisfying $1\leq k\leq \mathrm{rank}(\mathbf{B})$ and $1\leq r\leq \mathrm{rank}(\mathbf{C})$. Then we have
\begin{equation}\label{P-recursive-AB}
P_{k,r}(x;\mathbf{A},\mathbf{B},\mathbf{C})=\frac{1}{k}	\sum_{i=1}^{d_B}\Vert \mathbf{b}_{i}\Vert^2\cdot P_{k-1,r}(x;\mathbf{Q}_{\{i\}}\cdot \mathbf{A},\mathbf{Q}_{\{i\}}\cdot\mathbf{B}, \mathbf{C}),
\end{equation}	
and
\begin{equation}\label{P-recursive-AC}
P_{k,r}(x;\mathbf{A},\mathbf{B},\mathbf{C})=	\frac{1}{r}\sum_{i=1}^{n_C}\Vert \mathbf{c}_{i}\Vert^2\cdot P_{k,r-1}(x;\mathbf{A}\cdot \mathbf{P}_{\{i\}},\mathbf{B},\mathbf{C}\cdot \mathbf{P}_{\{i\}}),
\end{equation}	
where $\mathbf{Q}_{\{i\}}$ and $\mathbf{P}_{\{i\}}$ are defined in \eqref{QSPR}.

\end{lemma}

\begin{proof}
We first prove \eqref{P-recursive-AB}.
Note that each $k$-subset $S\subset[d_B]$ can be associated to $k!$ ordered tuples $(i_1,\ldots,i_k)$ such that $S=\{i_1,\ldots,i_k\}$. Therefore, we can rewrite $P_{k,r}	(x;\mathbf{A},\mathbf{B},\mathbf{C})$ as
\begingroup\fontsize{10.5pt}{12pt}\selectfont
\begin{equation}\label{P-recursive-AB:eq1}
\begin{aligned}
&\quad \ P_{k,r}	(x;\mathbf{A},\mathbf{B},\mathbf{C})\\
&=\frac{1}{k!}\sum_{\substack{R\subset[n_C]\\|R|=r}}\sum_{i_1=1}^{d_B} \sum_{i_2\in [d_B]\backslash\{i_1\}}\cdots \sum_{i_k\in [d_B]\backslash\{i_1,\ldots,i_{k-1}\}} \det[\mathbf{B}_{:,\{i_1,\ldots,i_{k}\}}^{\rm T}\mathbf{B}_{:,\{i_1,\ldots,i_{k}\}}]\cdot \det[\mathbf{C}_{R,:}(\mathbf{C}_{R,:})^{\rm T}]\cdot p_{\{i_1,\ldots,i_{k}\},R}(x;\mathbf{A},\mathbf{B},\mathbf{C})\\
&\overset{(a)}=\frac{1}{k!}\sum_{\substack{R\subset[n_C]\\|R|=r}}\sum_{i_1=1}^{d_B} \Vert \mathbf{b}_{i_1}\Vert^2 \sum_{i_2\in [d_B]\backslash\{i_1\}}\cdots \sum_{i_k\in [d_B]\backslash\{i_1,\ldots,i_{k-1}\}}  \\
&\quad\quad \det[\mathbf{B}_{:,\{i_2,\ldots,i_{k}\}}^{\rm T}\mathbf{Q}_{\{i_1\}}\mathbf{B}_{:,\{i_2,\ldots,i_{k}\}}] \det[\mathbf{C}_{R,:}(\mathbf{C}_{R,:})^{\rm T}]\cdot p_{\{i_1,\ldots,i_{k}\},R}(x;\mathbf{A},\mathbf{B},\mathbf{C})\\
&=\frac{(k-1)!}{k!}\sum_{\substack{R\subset[n_C]\\|R|=r}}\sum_{i_1=1}^{d_B} \Vert \mathbf{b}_{i_1}\Vert^2 \sum_{\substack{S\subset[d_B]\backslash\{i_1\}\\|S|=k-1}} \det[\mathbf{B}_{:,S}^{\rm T}\mathbf{Q}_{\{i_1\}}\mathbf{B}_{:,S}]\cdot \det[\mathbf{C}_{R,:}(\mathbf{C}_{R,:})^{\rm T}]\cdot p_{S\cup \{i_1\},R}(x;\mathbf{A},\mathbf{B},\mathbf{C})\\
&\overset{(b)}=\frac{1}{k}\sum_{i_1=1}^{d_B} \Vert \mathbf{b}_{i_1}\Vert^2 \sum_{\substack{R\subset[n_C]\\|R|=r}} \sum_{\substack{S\subset[d_B]\\|S|=k-1}} \det[(\mathbf{Q}_{\{i_1\}}\mathbf{B}_{:,S})^{\rm T}(\mathbf{Q}_{\{i_1\}}\mathbf{B}_{:,S})]\cdot \det[\mathbf{C}_{R,:}(\mathbf{C}_{R,:})^{\rm T}]\cdot p_{S\cup \{i_1\},R}(x;\mathbf{A},\mathbf{B},\mathbf{C}).
\end{aligned}
\end{equation}
\endgroup
Here, equation ($a$) follows from Lemma \ref{lemma2.2}, and equation ($b$) follows from the identity
$\det[\mathbf{B}_{:,S}^{\rm T}\mathbf{Q}_{\{i_1\}}\mathbf{B}_{:,S}]=\det[(\mathbf{Q}_{\{i_1\}}\mathbf{B}_{:,S})^{\rm T}(\mathbf{Q}_{\{i_1\}}\mathbf{B}_{:,S})]$ and the fact that $\det[\mathbf{B}_{:,S}^{\rm T}\mathbf{Q}_{\{i_1\}}\mathbf{B}_{:,S}]=0$ if $i_1\in S$.
Then, to prove \eqref{P-recursive-AB}, it is enough to show that for each $i\in[d_B]$ and for each two subsets $S\subset[d_B]$ and $R\subset[n_C]$,
\begin{equation*}
p_{S\cup \{i\},R}(x;\mathbf{A},\mathbf{B},\mathbf{C})=	p_{S,R}(x;\mathbf{Q}_{\{i\}}\cdot \mathbf{A},\mathbf{Q}_{\{i\}}\cdot\mathbf{B},\mathbf{C}).
\end{equation*}
Note that by Lemma \ref{lemma2.1} we have
\begin{equation*}
\mathbf{Q}_{S\cup \{i\}}=\mathbf{I}_n-\mathbf{B}_{:,S\cup \{i\}}\mathbf{B}_{:,S\cup \{i\}}^{\dagger}=\mathbf{I}_n-\big(\mathbf{b}_{i}\mathbf{b}_{i}^{\dagger}+(\mathbf{Q}_{\{i\}} \mathbf{B}_{:,S})(\mathbf{Q}_{\{i\}} \mathbf{B}_{:,S})^{\dagger}\big)= \mathbf{Q}_{\{i\}}-(\mathbf{Q}_{\{i\}} \mathbf{B}_{:,S})(\mathbf{Q}_{\{i\}} \mathbf{B}_{:,S})^{\dagger},
\end{equation*}
so we can rewrite $p_{S\cup \{i\},R}(x;\mathbf{A},\mathbf{B},\mathbf{C})$ as
\begin{equation*}
p_{S\cup \{i\},R}(x;\mathbf{A},\mathbf{B},\mathbf{C})=\det[x\cdot \mathbf{I}_{d}-\mathbf{P}_R\mathbf{A}^{\rm T}\mathbf{Q}_{S\cup \{i\}}\mathbf{A}\mathbf{P}_R]=\det[x\cdot \mathbf{I}_{d}-\mathbf{P}_R\mathbf{A}^{\rm T}\big(\mathbf{Q}_{\{i\}}-(\mathbf{Q}_{\{i\}} \mathbf{B}_{:,S})(\mathbf{Q}_{\{i\}} \mathbf{B}_{:,S})^{\dagger}\big)\mathbf{A}\mathbf{P}_R].
\end{equation*}
 It is worth noting that according to Lemma
 \ref{lemma2.0} we have $(\mathbf{Q}_{\{i\}} \mathbf{B}_{:,S})^{\dagger}=(\mathbf{Q}_{\{i\}} \mathbf{B}_{:,S})^{\dagger}\mathbf{Q}_{\{i\}}$.
This allows us to further express $p_{S\cup \{i\},R}(x;\mathbf{A},\mathbf{B},\mathbf{C})$ as
\begin{equation*}
\begin{aligned}
p_{S\cup \{i\},R}(x;\mathbf{A},\mathbf{B},\mathbf{C})&=\det[x\cdot \mathbf{I}_{d}-\mathbf{P}_R(\mathbf{Q}_{\{i\}}\mathbf{A})^{\rm T}\big(\mathbf{I}_n-(\mathbf{Q}_{\{i\}} \mathbf{B}_{:,S})(\mathbf{Q}_{\{i\}} \mathbf{B}_{:,S})^{\dagger}\big)(\mathbf{Q}_{\{i\}}\mathbf{A})\mathbf{P}_R]\\
&=p_{S,R}(x;\mathbf{Q}_{\{i\}}\cdot \mathbf{A},\mathbf{Q}_{\{i\}}\cdot\mathbf{B},\mathbf{C}).	
\end{aligned}
\end{equation*}
Substituting the above equation into \eqref{P-recursive-AB:eq1} we obtain \eqref{P-recursive-AB}. Applying the result of \eqref{P-recursive-AB} to the matrices $\mathbf{A}^{\rm T},\mathbf{C}^{\rm T}$ and $\mathbf{B}^{\rm T}$ and using Lemma \ref{P-sym}, we immediately obtain \eqref{P-recursive-AC}. This completes the proof of the lemma.

\end{proof}

Now we can give a proof of Proposition \ref{P-recursive2}.

\begin{proof}[Proof of Proposition \ref{P-recursive2}]
We first prove \eqref{P-recursive2-AB}.
Applying \eqref{P-recursive-AB} to the matrices $\mathbf{Q}_{S}\mathbf{A}\mathbf{P}_{R}\in\mathbb{R}^{n\times d}$, $\mathbf{Q}_{S}\mathbf{B}=[\mathbf{Q}_{S}\mathbf{b}_1,\ldots,\mathbf{Q}_{S}\mathbf{b}_{d_B}]\in\mathbb{R}^{n\times d_B}$ and $\mathbf{C}\mathbf{P}_{R}\in\mathbb{R}^{n_C\times d}$, we obtain that
\begin{equation}\label{Spr:eq2}
\begin{aligned}
& P_{k-l,r-t}(x;\mathbf{Q}_{S}\mathbf{A}\mathbf{P}_{R},\mathbf{Q}_{S}\mathbf{B},\mathbf{C}\mathbf{P}_{R})\\
&=\frac{1}{k-l}	\sum_{i:\Vert \mathbf{Q}_{S}\mathbf{b}_{i} \Vert\neq 0}^{}\Vert \mathbf{Q}_{S}\mathbf{b}_{i} \Vert^2\cdot P_{k-l-1,r-t}(x;\big(\mathbf{I}_n-(\mathbf{Q}_{S}\mathbf{b}_{i})(\mathbf{Q}_{S}\mathbf{b}_{i})^{\dagger}\big)\mathbf{Q}_{S}\mathbf{A}\mathbf{P}_{R},\big(\mathbf{I}_n-(\mathbf{Q}_{S}\mathbf{b}_{i})(\mathbf{Q}_{S}\mathbf{b}_{i})^{\dagger}\big)\mathbf{Q}_{S}\mathbf{B},\mathbf{C} \mathbf{P}_{R}).
\end{aligned}
\end{equation}
Note that
\begingroup\fontsize{11pt}{12pt}\selectfont
\begin{equation*}\label{remark34:eq1}
\big(\mathbf{I}_n-(\mathbf{Q}_{S}\mathbf{b}_{i})(\mathbf{Q}_{S}\mathbf{b}_{i})^{\dagger}\big)\mathbf{Q}_{S}=\mathbf{Q}_{S}-(\mathbf{Q}_{S}\mathbf{b}_{i})(\mathbf{Q}_{S}\mathbf{b}_{i})^{\dagger}\mathbf{Q}_{S}\overset{(a)}=	\mathbf{Q}_{S}-(\mathbf{Q}_{S}\mathbf{b}_{i})(\mathbf{Q}_{S}\mathbf{b}_{i})^{\dagger}\overset{(b)}=\mathbf{Q}_{S\cup\{i\}},
\end{equation*}
\endgroup
where ($a$) follows from Lemma \ref{lemma2.0} and ($b$) follows from Lemma \ref{lemma2.1}. Therefore, combining the above equation with \eqref{Spr:eq2} we obtain \eqref{P-recursive2-AB}.
Similarly, we can prove \eqref{P-recursive2-AC} by applying \eqref{P-recursive-AC} to the matrices $\mathbf{Q}_{S}\mathbf{A}\mathbf{P}_{R}$, $\mathbf{Q}_{S}\mathbf{B}$ and $\mathbf{C}\mathbf{P}_{R}$.	
\end{proof}

\subsection{Proof of Lemma \ref{lemma:alg-GCRSS}}\label{Appendix-alg-GCRSS}

Here we present a proof of Lemma \ref{lemma:alg-GCRSS}.

\begin{proof}[Proof of Lemma \ref{lemma:alg-GCRSS}]

For convenience, we denote $\widehat{k}=k-|\widehat{S}|$ and $\widehat{r}=r-|\widehat{R}|$.
Recall that we denote $\mathbf{V}=\mathbf{B}\mathbf{B}^{\rm T}\mathbf{Q}_{\widehat{S}}\in\mathbb{R}^{n\times n}$ and $\mathbf{W}=\begin{pmatrix}
\mathbf{A}   \\
\mathbf{C}
\end{pmatrix}\mathbf{P}_{\widehat{R}}[\mathbf{A}^{\rm T}\mathbf{Q}_{\widehat{S}}\ \ \mathbf{C}^{\rm T}]\in\mathbb{R}^{(n+n_C)\times (n+n_C)}$.
By equation \eqref{P-exp-2} in Proposition \ref{P-exp-r=0} we have
\begingroup\fontsize{10pt}{12pt}\selectfont 
\begin{equation}\label{noveq2}
\begin{aligned}
&P_{k-|\widehat{S}|,r-|\widehat{R}|}(x;\mathbf{Q}_{\widehat{S}}\mathbf{A}\mathbf{P}_{\widehat{R}},\mathbf{Q}_{\widehat{S}}\mathbf{B},\mathbf{C}\mathbf{P}_{\widehat{R}})\\
&=\frac{(-1)^{\widehat{r}}\cdot x^{d-n+\widehat{k}}}{ \widehat{k}!(n_C-\widehat{r})!}  \partial_y^{\widehat{k}}\partial_z^{n_C-\widehat{r}}\ 
\det\left[
\begin{pmatrix}
x\cdot \mathbf{I}_n +y\cdot
\mathbf{Q}_{\widehat{S}}\mathbf{B}\mathbf{B}^{\rm T}\mathbf{Q}_{\widehat{S}}  & \mathbf{0}  \\
 \mathbf{0} & z\cdot \mathbf{I}_{n_C}
\end{pmatrix}
-
\begin{pmatrix}
\mathbf{Q}_{\widehat{S}}\mathbf{A}   \\
\mathbf{C}
\end{pmatrix}
\mathbf{P}_{\widehat{R}}\begin{pmatrix}
\mathbf{A}^{\rm T}\mathbf{Q}_{\widehat{S}} & \mathbf{C}^{\rm T}
\end{pmatrix}
\right]    \ \Bigg\vert_{y=z=0}\\
&\overset{(a)}=\frac{(-1)^{\widehat{r}}\cdot x^{d-n+\widehat{k}}}{ \widehat{k}!(n_C-\widehat{r})!} \partial_y^{\widehat{k}}\partial_z^{n_C-\widehat{r}}\ 
\det\left[
\begin{pmatrix}
x\cdot \mathbf{I}_n +y\cdot
\mathbf{Q}_{\widehat{S}}\mathbf{B}\mathbf{B}^{\rm T}\mathbf{Q}_{\widehat{S}}  & \mathbf{0}  \\
 \mathbf{0} & (x+z)\cdot \mathbf{I}_{n_C}
\end{pmatrix}
-
\begin{pmatrix}
\mathbf{Q}_{\widehat{S}}\mathbf{A}   \\
\mathbf{C}
\end{pmatrix}
\mathbf{P}_{\widehat{R}}\begin{pmatrix}
\mathbf{A}^{\rm T}\mathbf{Q}_{\widehat{S}} & \mathbf{C}^{\rm T}
\end{pmatrix}
\right]    \ \Bigg\vert_{y=0,z=-x}\\
&\overset{(b)}=\frac{(-1)^{\widehat{r}}\cdot x^{d-n+\widehat{k}}}{ \widehat{k}!(n_C-\widehat{r})!} \partial_y^{\widehat{k}}\partial_z^{n_C-\widehat{r}}\ 
\det\left[x\cdot \mathbf{I}_{n+n_C} +
\begin{pmatrix}
y\cdot
\mathbf{V}  & \mathbf{0}  \\
 \mathbf{0} & z\cdot \mathbf{I}_{n_C}
\end{pmatrix}
-
\mathbf{W}
\right]    \ \Bigg\vert_{y=0,z=-x},
\end{aligned}
\end{equation}
\endgroup
where ($a$) follows from the chain rule and  ($b$) follows from the Weinstein-Aronszajn identity \eqref{Weinstein-Aronszajn2}. 
By equation \eqref{expand det-x} we can expand
\begin{equation}\label{noveq3}
g(x,y,z):=\det\left[x\cdot \mathbf{I}_{n+n_C} +
\begin{pmatrix}
y\cdot
\mathbf{V}  & \mathbf{0}  \\
 \mathbf{0} & z\cdot \mathbf{I}_{n_C}
\end{pmatrix}
-
\mathbf{W}
\right] 
=\sum_{j=0}^{n_C+n}x^{n_C+n-j} \sum_{c_1=0}^{j}\sum_{c_2=0}^{j-c_1}\mu_{j,c_1,c_2} z^{c_1}	y^{c_2},	
\end{equation}
where $\mu_{j,c_1,c_2}$ denotes the coefficient of $x^{n_C+n-j}z^{c_1}y^{c_2}$ in $g(x,y,z)$. 
Substituting \eqref{noveq3} into \eqref{noveq2} gives
\begingroup\fontsize{10pt}{12pt}\selectfont
\begin{equation}\label{noveq-1}
\begin{aligned}
P_{k-|\widehat{S}|,r-|\widehat{R}|}(x;\mathbf{Q}_{\widehat{S}}\mathbf{A}\mathbf{P}_{\widehat{R}},\mathbf{Q}_{\widehat{S}}\mathbf{B},\mathbf{C}\mathbf{P}_{\widehat{R}})
&=\frac{(-1)^{\widehat{r}}\cdot x^{d-n+\widehat{k}}}{ \widehat{k}!(n_C-\widehat{r})!}  
\sum_{j=0}^{n_C+n}  \sum_{c_1=0}^{j}\sum_{c_2=0}^{j-c_1}\mu_{j,c_1,c_2}\cdot x^{n_C+n-j}\cdot 
 \partial_y^{\widehat{k}}\partial_z^{n_C-\widehat{r}}\  z^{c_1}	y^{c_2}   \ \Bigg\vert_{y=0,z=-x}\\
&=\frac{(-1)^{\widehat{r}}\cdot x^{d-n+\widehat{k}}}{(n_C-\widehat{r})!}  
\sum_{j=\widehat{k}}^{n_C+n}   \sum_{c_1=0}^{j-\widehat{k}}\mu_{j,c_1,\widehat{r}}\cdot x^{n_C+n-j}\cdot  \partial_z^{n_C-\widehat{r}}\  z^{c_1}	   \ \Bigg\vert_{z=-x}\\
&=\frac{x^{d-n+\widehat{k}}}{(n_C-\widehat{r})!}  
\sum_{j=\widehat{k}+n_C-\widehat{r}}^{n_C+n}   \sum_{c_1=n_C-\widehat{r}}^{j-\widehat{k}} \frac{(-1)^{c_1-n_C}\cdot c_1!}{(c_1-n_C+\widehat{r})!}\cdot 
 \mu_{j,c_1,\widehat{r}}\cdot x^{n-j+c_1+\widehat{r}}. \\
\end{aligned}
\end{equation}
\endgroup
We claim that, when the matrices $\mathbf{V}$ and $\mathbf{W}$ are known, all the coefficients $\mu_{j,c_1,c_2}$, $0\leq j\leq n_C+n$, $0\leq c_1\leq j$, $0\leq c_2\leq j-c_1$, can be computed in time $O((n_C+n)^{w+2})$, where $w \in  (2, 2.373)$ is the matrix multiplication exponent.
Then, using equation \eqref{noveq-1}, we can compute the polynomial $P_{k-|\widehat{S}|,r-|\widehat{R}|}(x;\mathbf{Q}_{\widehat{S}}\mathbf{A}\mathbf{P}_{\widehat{R}},\mathbf{Q}_{\widehat{S}}\mathbf{B},\mathbf{C}\mathbf{P}_{\widehat{R}})$ in time $O((n+\widehat{r}-\widehat{k})^2)$. Note that $\widehat{r}\leq r\leq \mathrm{rank}(\mathbf{C})\leq n_C$, so the total time complexity for computing $P_{k-|\widehat{S}|,r-|\widehat{R}|}(x;\mathbf{Q}_{\widehat{S}}\mathbf{A}\mathbf{P}_{\widehat{R}},\mathbf{Q}_{\widehat{S}}\mathbf{B},\mathbf{C}\mathbf{P}_{\widehat{R}})$ is 
\begin{equation*}
O\big ((n_C+n)^{w+2}\big)+O\big((n+\widehat{r}-\widehat{k})^2\big)=O\big((n_C+n)^{w+2}\big).	
\end{equation*}

It remains to prove that all the coefficients $\mu_{j,c_1,c_2}$ can be computed in time $O((n_C+n)^{w+2})$. 
To compute the coefficients $\mu_{j,c_1,c_2}$, we use the polynomial interpolation. 
For convenience, we let $N=n_C+n+1$ and let $\omega_{s,N}=e^{2\pi \frac{s}{N}\sqrt{-1}}$ be the $N$-th roots of unity, $s=1,\ldots,N$. 
We first compute the matrix $\begin{pmatrix}
y\cdot
\mathbf{V}  & \mathbf{0}  \\
 \mathbf{0} & z\cdot \mathbf{I}_{n_C}
\end{pmatrix}
-
\mathbf{W}\in\mathbb{R}^{(N-1)\times(N-1)}$ and its characteristic polynomial by setting $(y,z)=(\omega_{s_1,N},\omega_{s_2,N})$ for each pair $(s_1,s_2)\in[N]\times [N]$. 
This yields $\{g_j(\omega_{s_1,N},\omega_{s_2,N})\}_{0\leq j\leq N-1,1\leq s_1\leq N,1\leq s_2\leq N}$. 
Note that the time complexity for computing a $(N-1)\times(N-1)$ matrix is $O(N^w)$ \cite{KG85}. Therefore, the total time complexity here is $O(N^2\cdot (n^2+N^{w}))=O(N^{w+2})$. 
Then, for any fixed $j\in\{0,1,\ldots,N-1\}$, we can recover the coefficients $\{\mu_{j,c_1,c_2}\}_{0\leq c_1\leq j,0\leq c_2\leq j-c_1}$ from $\{g_j(\omega_{s_1,N},\omega_{s_2,N})\}_{1\leq s_1\leq N,1\leq s_2\leq N}$, using the 2-dimensional Fast Fourier Transform. 
This takes $O(N^2 \log N)$ for each $0\leq j\leq N-1$ \cite[Chapter 23]{CG99}. As such, we can compute all the coefficients $\mu_{j,c_1,c_2}$ in time 
\begin{equation*}
O(N^{w+2})+N\cdot O( N^2 \log N)=O(N^{w+2})=O\big((n_C+n)^{w+2}\big).	
\end{equation*}
This completes the proof.

\end{proof}

\subsection{Proof of Proposition \ref{fk:exp3}}\label{Appendix3}

Here we present a proof of Proposition \ref{fk:exp3}.

\begin{proof}[Proof of Proposition \ref{fk:exp3}]
For convenience, we let
\begin{equation*}
\widetilde{\mathbf{A}}:=\mathbf{U}_{ }^{\rm T}\mathbf{A}\mathbf{A}^{\rm T}\mathbf{U}\in \mathbb{R}^{m\times m}\quad\text{and}\quad a_{J}:=\det[\widetilde{\mathbf{A}}_{J,J}],\ \forall  J\subset[m]
\end{equation*}
Noting that $\mathbf{B}_{ }\mathbf{B}_{ }^{\dagger}=\mathbf{U}_{ }\mathbf{U}_{ }^{\rm T}$,  we have
\begin{equation}\label{Frob:eq18}
\begin{aligned}
\det[x\cdot \mathbf{I}_n-( \mathbf{B}_{ }\mathbf{B}_{ }^{\dagger}\mathbf{A}\mathbf{A}^{\rm T}\mathbf{B}_{ }\mathbf{B}_{ }^{\dagger}+y\cdot \mathbf{B}\mathbf{B}^{\rm T})]&=\det[x\cdot\mathbf{I}_n- (\mathbf{U}_{ }\mathbf{U}_{ }^{\rm T}\mathbf{A}\mathbf{A}^{\rm T}\mathbf{U}_{ }+y\cdot \mathbf{U}_{ }\mathbf{\Sigma})\mathbf{U}_{ }^{\rm T}]\\
&\overset{(a)}=x^{n-m} \det[x\cdot\mathbf{I}_m-\mathbf{U}_{ }^{\rm T}(\mathbf{U}_{ }\mathbf{U}_{ }^{\rm T}\mathbf{A}\mathbf{A}^{\rm T}\mathbf{U}_{}+y\cdot \mathbf{U} \mathbf{\Sigma})]\\
&=x^{n-m} \det[\mathrm{diag}(x-y\cdot b_1,\ldots,x-y\cdot b_m)-\widetilde{\mathbf{A}}]\\
&\overset{(b)}=x^{n-m}  \sum_{J\subset[m]} (-1)^{\vert J\vert}\cdot \det[\widetilde{\mathbf{A}}_{J,J}]\cdot \prod_{l\in J^C}(x-y\cdot b_l)\\
&=\sum_{j=0}^m \sum_{i=0}^{m-j}(-1)^{i+j}\cdot  c_{j,i}\cdot x^{n-i-j}\cdot y^i,
\end{aligned}
\end{equation}
where
\begin{equation*}
c_{j,i}:=\sum_{J\subset [m],\vert J\vert =j}\sum_{S\subset J^C,\vert S\vert=i} a_J\cdot \mathbf{b}^{S},\quad\forall j\in[m],\forall i\in[m-j].	
\end{equation*}
Here, equation ($a$) follows from the Weinstein-Aronszajn identity \eqref{Weinstein-Aronszajn2} and equation ($b$) follows from \eqref{expand det}.
Then, substituting \eqref{Frob:eq18} into \eqref{lemma_f_k:expression2} we obtain
\begin{equation}\label{Frob:eq20}
P_{k}(x;\mathbf{B}_{ }\mathbf{B}_{ }^{\dagger}\mathbf{A},\mathbf{B})=\sum_{j=0}^{m-k}(-1)^{j}\cdot c_{j,k}\cdot x^{d-j}.
\end{equation}
It remains to show that the right hand side of \eqref{section5:fk} is equal to the right hand side of \eqref{Frob:eq20}.
Combining \eqref{def-hkx} with \eqref{expand det} we have
\begin{equation*}
h_k(z_1,\ldots,z_m)\cdot \det[\mathbf{Z}-\widetilde{\mathbf{A}}]=k!	\sum_{S\subset[m],\vert S\vert=m-k} \sum_{R\subset[m]}  (-1)^{\vert R^C\vert}\cdot a_{R^C}\cdot \mathbf{b}^{S^C}\cdot \mathbf{Z}^S \cdot \mathbf{Z}^R.
\end{equation*}
Since
\begin{equation*}
\Big(\prod_{i=1}^m \partial_{z_i}\Big)\cdot 	\Big(\mathbf{Z}^S \cdot \mathbf{Z}^R\Big)=\begin{cases}
2^{\vert S\cap R\vert}\cdot \mathbf{Z}^{S\cap R} & \text{if $S^C\subset R$}\\
0 & \text{else}	
\end{cases},
\end{equation*}
we further obtain
\begin{equation*}
\begin{aligned}
\Big(\prod_{i=1}^m \partial_{z_i}\Big)\cdot	 h_k(z_1,\ldots,z_m)\cdot \det[\mathbf{Z}-\widetilde{\mathbf{A}}]&=k!	\sum_{\substack{S\subset[m]\\ \vert S\vert=m-k}} \sum_{\substack{R\subset [m]\\ R\supset S^C}}  (-1)^{\vert R^C\vert}\cdot a_{R^C}\cdot \mathbf{b}^{S^C}\cdot 2^{\vert S\cap R\vert}\cdot \mathbf{Z}^{S\cap R}.
\end{aligned}
\end{equation*}
For each two subsets $S$ and $R$ in the above summation, by letting $W:=R^C$ and $T:=S\backslash W$, we can rewrite the right hand side of the above equation as
\begin{equation*}
\Big(\prod_{i=1}^m \partial_{z_i}\Big)\cdot	 h_k(z_1,\ldots,z_m)\cdot \det[\mathbf{Z}- \widetilde{\mathbf{A}}]=k!  \sum_{\substack{W\subset[m]\\ \vert W\vert\leq m-k}}\sum_{\substack{T\subset W^C\\ \vert T\vert=m-k-\vert W\vert}} (-1)^{\vert W\vert}\cdot a_{W}\cdot \mathbf{b}^{(T\cup W)^C}\cdot 2^{\vert T\vert}\cdot  \mathbf{Z}^{T}.
\end{equation*}
Therefore,
\begin{equation}\label{section5:fk4}
\Big(\prod_{i=1}^m \partial_{z_i}\Big)\cdot	 h_k(z_1,\ldots,z_m)\cdot \det[\mathbf{Z}- \widetilde{\mathbf{A}}]\ \Big|_{z_i=\frac{x}{2},\forall i\in[m]}=k!\sum_{j=0}^{m-k}(-1)^{j}\cdot  \widehat{c}_j\cdot x^{m-k-j},	
\end{equation}
where
\begin{equation*}
\begin{aligned}
\widehat{c}_j&:=\sum_{\substack{W\subset[m]\\ \vert W\vert=j}}\sum_{\substack{T\subset W^C\\ \vert T\vert=m-k-j}} a_W\cdot \mathbf{b}^{(T\cup W)^C}=\sum_{\substack{W\subset[m]\\ \vert W\vert=j}}\sum_{\substack{Q\subset W^C\\ \vert Q\vert=k}} a_W\cdot \mathbf{b}^{Q}.	
\end{aligned}
\end{equation*}
Observe that $\widehat{c}_j=c_{j,k}$ for each $0\leq j\leq m-k$. By substituting $\widehat{c}_j=c_{j,k}$ and \eqref{section5:fk4} into the right-hand side of \eqref{section5:fk}, we obtain:
\begin{equation*}
\frac{x^{d+k-m}}{k!}\cdot	\Big(\prod_{i=1}^m \partial_{z_i}\Big)\cdot h_k(z_1,\ldots,z_m)\cdot \det[\mathbf{Z}- \widetilde{\mathbf{A}}]\ \Big|_{z_i=\frac{x}{2},\forall i\in[m]}=\sum_{j=0}^{m-k}(-1)^{j}\cdot  c_{j,k}\cdot x^{d-j},
\end{equation*}
which is identical to the right-hand side of \eqref{Frob:eq20}.  Consequently, we have reached our desired conclusion.

\end{proof}


\nocite{1}

\Addresses

\end{document}